\documentclass [11pt,reqno]{amsart}
\usepackage{amsmath,amssymb,verbatim,geometry,color}
\usepackage[all]{xy}
\usepackage{mathrsfs}
\usepackage{hyperref}

\usepackage[pdftex]{graphicx}
\usepackage{mathtools}
\usepackage[style=alphabetic, sorting=nyt, maxnames=99, maxalphanames=10, giveninits=true, isbn=false, doi=false, url=false]{biblatex}
\usepackage{quiver}

\def\XXint#1#2#3{{\setbox0=\hbox{$#1{#2#3}{\int}$ }
\vcenter{\hbox{$#2#3$ }}\kern-.6\wd0}}

\renewbibmacro{in:}{}

\newcommand{\A}{\mathbb{A}}

\newcommand{\G}{\mathbb{G}}

\newcommand{\N}{\mathbb{N}}
\renewcommand{\P}{\mathbb{P}}
 \newcommand{\Q}{\mathbb{Q}}
 \newcommand{\R}{\mathbb{R}}
 \newcommand{\Z}{\mathbb{Z}}

\newcommand{\fa}{\mathfrak{a}}
\newcommand{\fb}{\mathfrak{b}}

\renewcommand{\hom}{\mathrm{hom}}

\newcommand{\cC}{\mathcal{C}}

\newcommand{\cE}{\mathcal{E}}
\newcommand{\cF}{\mathcal{F}}

\newcommand{\cH}{\mathcal{H}}
\newcommand{\cJ}{\mathcal{J}}
\newcommand{\cL}{\mathcal{L}}
\newcommand{\cM}{\mathcal{M}}
\newcommand{\cN}{\mathcal{N}}
\newcommand{\cO}{\mathcal{O}}
\newcommand{\cP}{\mathcal{P}}

\newcommand{\cR}{\mathcal{R}}
\newcommand{\cT}{\mathcal{T}}

\newcommand{\Xdiv}{X^{\mathrm{div}}}

\renewcommand{\a}{\alpha}
\renewcommand{\b}{\beta}
\renewcommand{\d}{\delta}
\newcommand{\e}{\varepsilon}
\newcommand{\f}{\varphi}

\newcommand{\g}{\gamma}
\newcommand{\la}{\lambda}

\newcommand{\La}{\Lambda}
\newcommand{\p}{\psi}

\newcommand{\cf}{{\rm cf.\ }} 
\newcommand{\eg}{{\rm e.g.\ }} 
\newcommand{\ie}{{\rm i.e.\ }} 
\newcommand{\loccit}{\textit{loc.\,cit.}}
\renewcommand{\=}{\coloneqq}

\newcommand{\an}{\mathrm{an}}

\DeclareMathOperator{\ii}{I}
\DeclareMathOperator{\jj}{J}

\DeclareMathOperator{\Num}{N^1}
\DeclareMathOperator{\CNum}{CN^1}
\DeclareMathOperator{\WNum}{WN}
\DeclareMathOperator{\Curv}{WN_1}

\DeclareMathOperator{\BPF}{BPF}
\DeclareMathOperator{\BBPF}{BBPF}

\DeclareMathOperator{\CbDiv}{CbDiv}
\DeclareMathOperator{\CDiv}{CDiv}

\DeclareMathOperator{\FS}{FS}

\DeclareMathOperator{\MA}{MA}

\DeclareMathOperator{\PbDiv}{PbDiv}
\DeclareMathOperator{\PDiv}{PDiv}

\DeclareMathOperator{\Spec}{Spec}
\DeclareMathOperator{\Sym}{Sym}

\DeclareMathOperator{\supp}{supp}
\DeclareMathOperator{\vol}{vol}

\DeclareMathOperator{\WbDiv}{WbDiv}
\DeclareMathOperator{\WDiv}{WDiv}

\DeclareMathOperator{\ord}{ord}

\DeclareMathOperator{\NS}{NS}

\DeclareMathOperator{\Proj}{Proj}

\DeclareMathOperator{\PSH}{PSH}
\DeclareMathOperator{\CPSH}{CPSH}
\DeclareMathOperator{\IN}{IN}

\DeclareMathOperator{\gr}{gr}

\DeclareMathOperator{\Nbdd}{\cN_{\mathrm{bdd}}}
\DeclareMathOperator{\PCnum}{\overline{\cP\cC^{\mathrm{int}}}}

\DeclareMathOperator{\PL}{PL}

\DeclareMathOperator{\Hnot}{H^0}

\newcommand{\ddc}{\mathrm{dd^c}}

\DeclareMathOperator{\de}{d}

\DeclareMathOperator{\du}{\widetilde{d}_1}

\DeclareMathOperator{\trdeg}{tr.deg}

\renewcommand{\div}{\mathrm{div}}

\newcommand{\triv}{\mathrm{triv}}
\newcommand{\val}{\mathrm{val}}

\newcommand{\lin}{\mathrm{lin}}

\newcommand{\na}{\mathrm{na}}

\newcommand{\hto}{\hookrightarrow}

\newcommand{\simto}{\overset\sim\to}

\numberwithin{equation}{section}       

\newtheorem{prop} {Proposition} [section]
\newtheorem{thm}[prop] {Theorem} 
\newtheorem{defi}[prop] {Definition}
\newtheorem{lem}[prop] {Lemma}
\newtheorem{cor}[prop]{Corollary}
\newtheorem{prop-def}[prop]{Proposition-Definition}

\newtheorem*{thmA}{Theorem A} 
\newtheorem*{thmAp}{Theorem A'} 
 
\newtheorem*{thmB}{Theorem B} 
\newtheorem*{thmC}{Theorem C} 
\newtheorem*{thmD}{Theorem D}

\newtheorem{exam}[prop]{Example}
\newtheorem{rmk}[prop]{Remark}

\newtheorem*{thm*}{Theorem} 
\theoremstyle{remark}

\title{The space of polarized compactifications of an affine variety}
\date{\today}
\bibliography{infty}
\author{Mattias Jonsson \and Yueqiao Wu}

\address{Dept of Mathematics\\
  University of Michigan\\
  530 Church St\\
  Ann Arbor, MI 48109-1043\\
  USA}
\email{mattiasj@umich.edu}

\address{Dept of Mathematics\\
  Johns Hopkins University\\
  3400 N. Charles Street\\
  Baltimore, MD 21218\\
  USA}
\email{ywu347@jh.edu}

\begin{document}

\begin{abstract}
Let $X$ be a normal affine variety over an algebraically closed field of characteristic zero. Through the lens of non-Archimedean geometry, we study the space of polarized compactifications of $X$, consisting of pairs $(\bar X, D)$ where $\bar X$ is a normal projective variety containing $X$ as a dense open subset, and $D$ is an effective ample $\Q$-Cartier $\Q$-divisor with support equal to the complement of $X$ in $\bar X$. 

We identify this space with a natural class of functions--the Fubini--Study functions--on the Berkovich analytification of $X$. Through a Rees correspondence, we also give a bijection with the space of homogeneous $\Q$-valued, finitely generated negative non-Archimedean norms (or filtrations) on the coordinate ring of $X$.

Using the above identifications, we equip the space of polarized compactifications, and numerical classes thereof,  with natural metrics and identify the metric completions. The latter include a space of continuous plurisubharmonic functions in the sense of Chambert-Loir and Ducros, and a space of birational curve classes as studied by Dang and Favre.

\end{abstract}
\maketitle

\setcounter{tocdepth}{1}
\tableofcontents

\section*{Introduction}
Let $X$ be an affine variety over an algebraically closed field $k$ of characteristic zero. For the purposes of this introduction, $X$ is assumed normal.
By a \emph{polarized compactification} of $X$ we mean a pair $(\bar{X},D)$, where $\bar{X}$ is a projective variety with an open immersion $X\hto\bar{X}$ with dense image, and $D$ is an ample effective $\Q$-Cartier $\Q$-divisor on $\bar{X}$ with support equal to $\bar{X}\setminus X$. Our main goal in this paper is to study the space of polarized compactifications from several points of view. The situation is reminiscent of the space of (ample) test configurations of a polarized projective variety, as studied in~\cite{nakstab1}. Our results will be used in~\cite{Wu26} to study polarized affine Calabi--Yau varieties.

\medskip\noindent\textbf{Polarized compactifications and psh functions.}
Denote by $\cP\cC^{\mathrm{int}}(X)$ the space of normal polarized compactifications of $X$.\footnote{The superscript stands for "integrally closed", which here means normal, as $X$ is normal.}

Any $(\bar X, D)\in \cP\cC^{\mathrm{int}}(X)$ defines a function on the \emph{Berkovich analytification} $X^\na$ of $X$ with respect to the trivial absolute value on $k$: if $v\in X^\na\subset\bar{X}^\na$, we set $\p_D(v)=m^{-1}v(h_{mD})$, where $m$ is sufficiently divisible, and $h_{mD}$ is a local equation of $mD$ at the center of $v$ on $\bar{X}$. Then $\p_D\colon X^\na\to \R_+$ is continuous and homogeneous, vanishes on $X^\beth$, the set of semivaluations centered on $X$, and is strictly positive on the complement $X^\na_\infty$, the set of semivaluations centered at infinity. 

The functions defined by polarized compactifications also admit a different description. Write $X=\Spec R$. For $f\in R$, set $\log^+|f|(v)\=\max\{-v(f), 0\}$. Given a subgroup $\Lambda\subset \R$, the set $\cH_\Lambda$ of $\Lambda$-Fubini--Study functions consists of functions of the form
\[
\f=\max_j \lambda_j^{-1}\log^+|f_j|,
\]
where $\lambda_j\in \Lambda\cap\R_{>0}$ and $(f_j)$ is a finite set of $k$-algebra generators of $R$. The space $\cH_\Lambda\subset C^0(X^\na, \R_+)$ is stable under max and satisfies $\cH_\Lambda=\cH_{\Q\Lambda}$. Any Fubini--Study function is also plurisubharmonic in the sense of~\cite{CLD25}.
Using the explicit description of Fubini--Study functions, we show that $(\bar X,D)\mapsto\p_D$ defines a bijection
\[
\cP\cC^{\mathrm{int}}(X)\simto \cH(X)\=\cH_\Q(X).
\]

The space $\cH(X)$ carries a natural metric, which gives a metric space structure on $\cP\cC^{\mathrm{int}}(X)$ via the above correspondence. Indeed, any function in $\cH(X)$ is continuous, positive, and homogeneous on $X^\na_\infty$, and one can show that $X^\na_\infty/\R_+^\times$ is compact, so we have a well-defined metric $\de_\infty$ on $\cH(X)$ given by 
\[
\de_\infty(\f, \p)\=\inf\{\delta\ge 0\mid e^{-\delta}\f\le \p\le e^\delta\f\}.
\]
This is the analog of the sup norm metric for continuous functions on compact topological spaces. Motivated by complex and non-Archimedean pluripotential theory, we define a plurisubharmonic (psh) function to be a function $\f\colon X^\na\to \R_+$ that can be written as a pointwise limit of a decreasing net of functions in $\cH(X)$.\footnote{Such functions are psh in the sense of Chambert-Loir and Ducros~\cite{CLD25}, but our notion is more restrictive; for example, our functions are nonnegative and homogeneous.} 
The space $\PSH(X)$ of psh functions is equipped with the topology of pointwise convergence on the subset $X^\div_\infty$ of (rational) divisorial valuations centered at infinity. Such functions are nonnegative, vanish identically on $X^\beth$, and they are uniquely determined by their values on $X^\div_\infty$. By Dini's lemma, the set $\CPSH(X)$ of \emph{continuous} psh functions is the closure of $\cH(X)$ under locally uniform convergence, and hence the metric completion of $\cP\cC^{\mathrm{int}}(X)\simeq\cH(X)$ with respect to $\de_\infty$ can be identified with the space of functions in $\CPSH(X)$ that are strictly positive on $X^\na_\infty$. 

\medskip
\noindent\textbf{Numerical classes of polarized compactifications.}
Assume $n\=\dim X\ge 2$.
We next turn to describing natural metrics on the space $\PCnum(X)$ of \emph{numerical classes} defined by normal polarized compactifications. 
More precisely, $\PCnum(X)$ is the space of pairs $(\bar X,\a)$, where $\bar{X}$ is a normal compactification of $X$, and $\a\in N^1(\bar{X})$ is a numerical class represented by an ample effective $\Q$-Cartier $\Q$-divisor with support equal to $\bar{X}\setminus X$.

It is natural to work birationally, using the notion of b-divisors~\cite{Shokurov1, Shokurov2} and b-divisor classes~\cite{BFJ09, bdff, DangFavre}. Let $\BPF_{n-1}(X)$ be the closure of the cone generated by base point free Cartier b-divisor classes, the closure taken inside the space of Weil b-divisor classes on (the Riemann-Zariski space of) $X$. The volume function on Cartier divisors induces a natural function $\vol\colon\BPF_{n-1}(X)\to\R_+$, and we write $\BBPF_{n-1}(X)$ for the subcone of classes $\a$ that are \emph{big} in the sense that $\vol(\a)>0$.

Similarly, inside the space of Weil b-curve classes, we have the closed cone $\BPF_1(X)$, defined as the closure of the cone generated by intersections of $n-1$ base point free  Cartier b-divisor classes. We also have a \emph{dual volume function} $\vol^\vee\colon\BPF_1(X)\to\R_+$ defined by 
\[
\vol^\vee(\mu):=\inf_{\a\in\BBPF_{n-1}(X)}\frac{(\a,\mu)}{\vol(\a)^{1/n}},
\]
and $\BBPF_1(X)\subset\BPF_1(X)$ is defined by the condition $\vol^\vee>0$.

Dang and Favre in~\cite{DangFavre} proved that any element of $\BBPF_{n-1}(X)$ is a decreasing limit of semiample Cartier b-divisor classes, and used this to define a map 
\[
\BBPF_{n-1}(X)\to \BBPF_1(X),\quad \alpha\mapsto \alpha^{n-1},
\]
which they also proved to be a bijection, a version of the classical \emph{Minkowski problem} in convex geometry. We note that the setting in~\cite{DangFavre} is that of a smooth projective variety, but the statements are of birational nature, and apply to an affine variety $X$ via compactification and resolution of singularities.

Our first result is a topological strengthening of this result. Both $\BBPF_{n-1}(X)$ and $\BBPF_1(X)$ admit a $\de_\infty$-metric as follows: 
for $\alpha, \beta\in \BBPF_{p}(X)$ for $p=1, n-1$, we set 
\[
\de_\infty(\alpha, \beta)\=\inf\{\delta\ge 0\mid e^{-\delta}\a\le \b\le e^\delta\a\},
\]
where by $\a\ge\beta$ we mean $\alpha-\beta$ is pseudoeffective. 
\begin{thmA}
    The metric spaces $(\BBPF_{n-1}(X),\de_\infty)$ and $(\BBPF_1(X),\de_\infty)$ are complete, and the map $\a\mapsto\a^{n-1}$ between them 
    is a Lipschitz continuous homeomorphism, with a locally uniformly H\"older continuous inverse.
\end{thmA}

Here completeness is established by showing that every Cauchy net converges weakly, and then upgrading this convergence to metric convergence.
As already mentioned, the fact that $\a\mapsto\a^{n-1}$ is a bijection is due to Dang and Favre. To prove the metric estimates, we compare the $\de_\infty$-metric on $\BBPF_{n-1}(X)$ to a variant of the Aubin $\ii$-functional, a quasi-metric defined multiplicatively, by
\[
\bar\ii(\a, \b)\=\log \frac{(\alpha, \beta^{n-1})(\beta, \alpha^{n-1})}{(\alpha^n)(\beta^n)} +|\log (\a^n)-\log(\b^n)|.
\]
The identity map $(\BBPF_{n-1}(X),\de_\infty)\to(\BBPF_{n-1}(X),\bar\ii)$ is easily seen to be Lipschitz continuous, and the Diskant inequality implies that the inverse is locally uniformly H\"older continuous. As the map $(\BBPF_{n-1}(X),\de_\infty)\to(\BBPF_1(X),\de_\infty)$ and the inverse map $(\BBPF_1(X),\de_\infty)\to(\BBPF_{n-1}(X),\bar\ii)$ are both Lipschitz continuous, Theorem~A follows; see~\S\ref{sec:Mink} for details.

As in~\cite{LX,DangFavre}, the volume function on $\BBPF_{n-1}(X)$ induces a function on $\BBPF_1(X)$. If $\a\in\BBPF_{n-1}(X)$, and $\mu:=\a^{n-1}\in\BBPF_1(X)$, then $\vol(\a)=\vol^\vee(\mu)^{\frac{n}{n-1}}$, and $\a$ attains the infimum in the definition of $\vol^\vee(\mu)$. 

\medskip
Theorem~A also holds in the (slightly) more general context of quasiprojective varieties. Let us now return to compactifications. Any polarized compactification $(\bar X,D)$ of $X$ defines a Cartier b-divisor that is big and semiample. Using the identification $\cP\cC^{\mathrm{int}}(X)\simeq\cH(X)$, this defines a~\emph{curvature operator}
\[
\ddc\colon \cH(X)\to \BBPF_{n-1}(X),
\]
whose image we denote by $\bar\cH(X)\simeq\PCnum(X)$.

We then consider the \emph{Monge--Amp\`ere operator} 
\[
\MA\colon\cH(X)\to\BBPF_1(X), \quad\MA(\f)=(\ddc\f)^{n-1},
\]
and aim to describe its image, as well as the image of a natural extension of the Monge--Amp\`ere operator, defined using non-Archimedean pluripotential theory, following the projective case~\cite{BJpp}, itself following the complex analytic situation~\cite{GueZer07}.

To this end, we note that the map $\ddc\colon\cH(X)\to\BBPF_{n-1}(X)$ is increasing. By taking decreasing limits of the associated Cartier
b-divisors, the curvature operator extends to an increasing
map $\ddc\colon\PSH(X)\to\BPF_{n-1}(X)$. The volume function then extends by setting $\vol(\f)\=\vol(\ddc\f)$. It is increasing and usc,
and hence continuous along decreasing limits.

\[
\cE(X)\=\{\f\in \PSH(X)\mid \vol(\f)>0\};
\]
this is the analogue of the space $\cE^1$ of psh metrics of finite energy in the projective setting. We also write $\bar\cE(X):=\ddc\cE(X)\subset\BBPF_{n-1}(X)$. In fact, each function $\f\in\cE(X)$ gives rise to a nef and effective Weil b-divisor ``at infinity" on $X$, and $\ddc\f$ is the corresponding Weil divisor class. 
Mapping $\f\in\cE(X)$ to $(\ddc\f)^{n-1}$ defines an extension of the Monge--Amp\`ere operator to a map $\MA\colon\cE(X)\to\BBPF_1(X)$.

We will show that the image of this operator is contained in the set
\[
\cC_\infty(X)\subset\BBPF_1(X)
\]
consisting of curve classes whose dual volume can be computed by functions in $\cH(X)$, i.e., $\mu\in \BBPF_1(X)$ for which
\[
\vol^\vee(\mu)=\vol_\infty^\vee(\mu)\=\inf_{\f\in \cH}\frac{(\f, \mu)}{\vol(\f)^{1/n}}.
\]
We can now complement Theorem~A as follows. 
\begin{thmAp}
    The subspace $\cC_\infty(X)\subset\BBPF_1(X)$ is closed, and hence complete. Moreover:
    \begin{itemize}
        \item[(i)]
        $\MA(\cH(X))$ is a dense subset of $\cC_\infty(X)$.
        \item[(ii)]
        $\MA(\cE(X))\subset\cC_\infty(X)$, with equality when $X$ is smooth.
    \end{itemize}
\end{thmAp}
As a consequence, the metric completion of $(\PCnum(X), \de_\infty)$ can be topologically identified with $\cC_\infty(X)$, and also with $\bar\cE(X)$ when $X$ is smooth.

The fact that $\MA(\cE(X))\subset\cC_\infty(X)$ follows from $\ddc\f\in\BBPF_{n-1}(X)$ realizing the infimum in the definition of $\vol^\vee(\MA(\f))$, for any $\f\in\cE(X)$. To prove density of $\MA(\cH(X))$ in $\cC_\infty(X)$, we follow~\cite{BJpp}. Namely, given any $\mu\in\cC_\infty$, and any \emph{minimizing net} $(\f_i)_i$ for $\mu$ in $\cH(X)$, that is, a net such that 
$(\f_i,\mu)/\vol(\f_i)^{1/n}\to\vol^\vee(\mu)$ and $\vol(\f_i)^{\frac{n-1}{n}}\to\vol^\vee(\mu)$, we show that $\de_\infty(\MA(\f_i),\mu)\to0$. As in~\cite{BBGZ,BJpp}, this relies on the uniform differentiability property of a suitable volume function; see Theorem~\ref{thm:unifderiv}. Finally, when $X$ is smooth, we prove that the space $\cE(X)$ has a certain compactness property, which allows us to show that the infimum in the definition of $\vol^\vee(\mu)$ is computed by some $\f\in\cE(X)$. This compactness property is derived from the \emph{envelope property}~\cite{BJpp} of the space of homogeneous $\theta$-psh functions on $\bar X$, for smooth compactifications $\bar X$ and numerical classes $\theta\in\Num(\bar X)$.

\medskip
Theorem~A' can be seen as a somewhat more traditional Monge--Amp\`ere equation, in analogy to the polarized projective case~\cite{BBGZ,BFJnama,BJpp}. Namely, if we fix $\p\in\cH(X)$, then the subset $X^\na_\p\=\{\p=1\}\simeq X^\na_\infty/\R_+^\times$ of $X^\na_\infty$ is compact. By homogeneity, any function $\f\in\cE(X)$ is uniquely determined by its restriction to $X^\na_\p$, and any curve class $\mu\in\cC_\infty$ induces a positive Radon measure $\mu_\p$ on $X^\na_\p$ of mass $(\p,\mu)$. We can view the pair $(X^\na_\infty/\R_+^\times,\p)$ as the analogue of a polarized projective variety, on which $\cE(X)|_{X^\na_\p}$ plays the role of the space $\cE^1$ of potentials of finite energy, while the family of measures $\mu_\p$, $\mu\in\cC_\infty$, plays the role of the space $\cM^1$ of measures of finite energy. In this way, Theorem A' becomes an analogue of the results in~\cite{BJpp}. We note that special cases of the spaces $X^\na_\p$, for $X=\A^n$, occur in~\cite{FJEigenvals,BFJIzumi}; see also~\cite{JonDynBerko}.

\medskip
\noindent\textbf{The Rees correspondence.}
By analogy with the study of metrics on the space of test configurations~\cite{nakstab1}, and degenerations of (klt) singularities~\cite{BLQ, Qi}, the space $\cP\cC(X)$ can also be understood via \emph{norms} on the coordinate ring $R$.

A \emph{negative norm} on $R$ is a function $\chi\colon R\to \R\cup\{\infty\}$ such that $\chi|_{k^\times}\equiv 0$, $\chi(f+g)\ge \min\{\chi(f), \chi(g)\}$, $\chi(fg)\ge \chi(f)+\chi(g)$, for any $f, g\in R$, $\chi(f)<0$ for any $f\in R\setminus k$ and $\chi(f)=\infty$ iff $f=0$. It is further said to be \emph{homogeneous} if $\chi(f^d)=d\chi(f)$ for any $f\in R, d\in \Z_{>0}$. 

Any polarized compactification $(\bar X, D)$ with $D$ a Cartier divisor induces a \emph{geometric norm} on $R$ by setting 
\[
\chi_{\bar X, D}(f)\=-\min\{m\in \Z_{\ge0}\mid fs_D^m\in \Hnot(\bar X, mD)\},
\]
for any $0\ne f\in R$ and $\chi_{\bar X, D}(0)=\infty$. 
Negative norms on $R$ are in one-to-one correspondence with (increasing, right-continuous, multiplicative and exhaustive) filtrations of $R$ by $k$-vector spaces, satisfying $\cF^0R=k$, via
\[
\cF_\chi^\lambda R\=\{f\in R\mid \chi(f)\ge -\lambda\}, \quad \chi(f)\=-\inf\{\lambda\ge 0\mid f\in \cF^\lambda R\},
\]
for $f\ne 0$.
We say a negative norm is \emph{finitely generated} if the associated graded ring $\gr_\chi R$ is a finitely generated $k$-algebra. For any subgroup $\Lambda\subset\R$, we write $\cT_\Lambda$ for the set of finitely generated negative norms with values in $\Lambda$, and $\cT_\Lambda^\hom$ for its subset of homogeneous norms. The \emph{homogenization} defines a map
\[
\cT_\Lambda\to \cT^\hom_{\Q\Lambda},\quad \chi\mapsto \chi^\hom,
\]
where $\chi^\hom(f)=\lim_d\frac{\chi(f^d)}{d}=\sup_d\frac{\chi(f^d)}{d}$ for all $f\in R$.
On $\cT_\Z$ we define an equivalence relation by $\chi\sim\chi'$ iff $\cF^m_\chi R=\cF^m_{\chi'}R$ for $m$ sufficiently divisible. Equivalent norms in $\cT_\Z$ have the same image in $\cT^\hom_\Q$.

As in the case for test configurations, see~\cite{BHJ1}, our next result describes a Rees correspondence between polarized compactifications and finitely generated negative norms.
\begin{thmB}
    The Rees construction defines a bijection 
    \[
    \cP\cC(X)\to \cT_\Z/\sim
    \]
    between the space of polarized compactifications and (equivalence classes of) finitely generated $\Z$-valued negative norms. Further, by passing to the integral closure and homogenization, it induces a bijection
    \[
    \cP\cC^{\mathrm{int}}(X)\to \cT^\hom_\Q.
    \]
\end{thmB}
Recall that $\cP\cC^{\mathrm{int}}(X)$ is also in bijection with $\cH(X)$, the space of rational Fubini--Study functions on $X$. We thus get a bijection between $\cH(X)$ and $\cT_\Q^\hom$. In fact, the above theorem is proved by establishing this bijection using the Fubini--Study and infimum norm operators:
\[
\FS\colon\cT_\Z\to \cH(X),\quad \chi\mapsto \FS(\chi) \qquad
\IN\colon \cH(X)\to \cT^\hom_\Q,\quad \f\mapsto \IN(\f),
\]
where for $v\in X^\na$, $\FS(\chi)(v)\=\sup_{f\in R\setminus k}\frac{v(f)}{\chi(f)}$; and for $f\in R$, $\IN(\f)(f)\=\inf_{v\in X^\na_\infty}\frac{v(f)}{\f(v)}$. 

Given a geometric norm $\chi\in \cT_\Z$ associated to a polarized compactification $(\bar X, D)$, we first show that $\FS(\chi)=\FS(\chi^\hom)=\p_D$ and thus the $\FS$ operator factors through homogenization, with inverse given by $\IN$. Using this bijection, the Rees correspondence for $\cP\cC^{\mathrm{int}}(X)$ is realized by describing the process of taking integral closure as the round-down of the homogenization.

\medskip\noindent\textbf{Metric space structures.}
We next discuss the natural metric structures on the space of finitely generated norms. To describe these metrics, we first extend our discussion of norms to the space $\Nbdd(X)$ of \emph{bounded negative norms} on $X$, that is, negative norms $\chi$ such that there exists $C\ge 1$ with $C\chi_{\bar X, D}\le \chi\le C^{-1}\chi_{\bar X, D}$ for some (or equivalently any) geometric norm $\chi_{\bar X, D}$. Any finitely generated norm is in particular bounded. 

The space $\Nbdd(X)$ comes with a complete $\de_\infty$ metric by setting
\[
\de_\infty(\chi, \chi')\=\inf\{\d\ge 0\mid e^\d \chi\le \chi'\le e^{-\d}\chi\}.
\]
The restriction of $\de_\infty$ to $\cT^\hom_\Q$ induces a metric on $\cP\cC^{\mathrm{int}}(X)$ via the Rees correspondence. However, the metric completion of $(\cP\cC^{\mathrm{int}}(X), \de_\infty)$ does not equal $\Nbdd(X)$, as the latter space contains the set of homogeneous norms as a closed proper subset containing $\cT^{\hom}_\Q$. Instead, Theorem~C below identifies $(\cP\cC^{\mathrm{int}}(X),\de_\infty)$ with $(\cH(X),\de_\infty)$,
so its completion can be viewed as the space of functions in $\CPSH(X)$ that are strictly positive on $X^\na_\infty$.

Alternatively, motivated by the study of metric structures of test configurations and filtrations on a local ring, the space $\Nbdd(X)$ can also be equipped with a natural \emph{Darvas pseudometric} $\de_1$, that we now describe.
To any norm $\chi\in \Nbdd(X)$, we associate the \emph{volume}
\[
\vol(\chi)\=\limsup_{\la\to\infty}\frac{n!}{\la^n}\dim \cF^\la_\chi R\in (0, \infty);
\]
this limsup is in fact a limit, which can be understood geometrically as follows. Pick any polarized compactification $(\bar{X},D)$. After passing to a sufficiently divisible multiple, $D$ is Cartier and $\cF^m_\chi R\subset H^0(\bar X,mD)$ for all $m\in\Z_{\ge1}$. Thus $(\cF_\chi^mR)_{m\ge0}$ can be viewed as a graded linear subseries of $R(\bar X,D)$, whose volume then equals $\vol(\chi)$.

The pseudometric $\de_1$ is defined by setting, for each $\chi, \chi'\in \Nbdd(X)$,
\[
\de_1(\chi, \chi')\=\vol(\chi)+\vol(\chi')-2\vol(\chi\wedge\chi'),
\]
where $\chi\wedge \chi'$ denotes the pointwise minimum. Restricting to $\cT^\hom_\Q$, it induces a natural metric on the space $\cP\cC^{\mathrm{int}}(X)$. We now upgrade the bijection between $\cT^\hom_\Q$ and $\cH(X)$ to isometric isomorphisms for the two metrics defined above.

\begin{thmC}
    For any affine variety $X$ and any subgroup $\Lambda\subset\R$, we have
    \begin{itemize}
        \item [(i)] the map $\FS\colon(\cT_{\Q\Lambda}^\hom, \de_\infty)\simto (\cH_\Lambda, \de_\infty)$ is an isometric isomorphism, with inverse given by $\IN$;
        \item[(ii)] the surjection $\FS\colon \cT_\Lambda\twoheadrightarrow \cH_\Lambda$ identifies the Hausdorff quotient of $(\cT_\Lambda, \de_1)$ with $(\cH_\Lambda, \de_1)$, and the bijection $\FS\colon(\cT^\hom_{\Q\Lambda},\de_1)\to(\cH_\Lambda,\de_1)$ is an isometry.
    \end{itemize}
\end{thmC}
The first isomorphism follows immediately from properties of the $\FS,\IN$ operators. To establish the second part of the theorem, we start by noting that if $\chi\=\chi_{\bar X, D}, \chi'\=\chi_{\bar X', D'}\in \cT_\Z$ are geometric norms with $\chi\le \chi'$, then the $\FS$ operator sends the two norms to $\p_D, \p_{D'}$ respectively with $\vol(\chi)=\vol(\p_D)\le \vol(\p_{D'})=\vol(\chi')$. In general, since the pointwise minimum $\chi\wedge\chi'$ need not be finitely generated, we use a finite type approximation, which also gives the following convenient rooftop formula for the $\de_1$ metric on $\cH_\R$:
\[
\de_1(\f, \p)=\vol(\f)+\vol(\p)-2\vol(\f\wedge\p),
\]
where $\vol(\f\wedge \p)\=\sup\{\vol(\tau)\mid \tau\in\cH_\R, \tau\le \f\wedge \p\}=\vol(\IN(\f)\wedge \IN(\p))$.

\medskip
\noindent\textbf{Geodesics in the Darvas metric.}
Our next goal is to study in more detail the metric structure and the Hausdorff quotient of $(\Nbdd(X), \de_1)$. This in turn builds on the study of metric geometry of spaces of test configurations and singularities, see~\cite{RemiGeodesics, nakstab1, BLQ, Qi}.  
To each $\chi\in \Nbdd(X)$, the \emph{saturation} $\bar \chi$ of $\chi$ is defined to be the largest homogeneous norm $\bar \chi\ge \chi$ such that $\FS(\chi)=\FS(\bar\chi)$ on $X^\div$. Passing to the saturation does not change the volume and hence $\de_1(\chi, \bar\chi)=0$. We denote by $\cN^s$ the space of saturated norms, that is, bounded negative norms satisfying $\bar\chi=\chi$.

The space $\Nbdd$ also comes with naturally defined geodesics, by choosing a common orthogonal basis. Given $\chi_0, \chi_1\in \Nbdd$, we may define for $t\in[0,1]$, a norm $\chi_t\in \Nbdd$ corresponding to the filtration given by
\[
\cF^\lambda_{\chi_t}R\=\sum_{(1-t)\mu+t\nu=\la}\cF^\mu_{\chi_0}R\cap \cF^\nu_{\chi_1}R.
\]
We now state the following analog of~\cite[Theorem~1.1]{BLQ},~\cite[Theorem~1.1]{Qi}.
\begin{thmD}
    The saturation map $(\Nbdd, \de_1)\to (\cN^s, \de_1)$ identifies $\cN^s$ as the Hausdorff quotient.
    Further, given a geodesic $\chi_t, t\in[0,1]$ joining two bounded negative norms $\chi_0, \chi_1$, the volume function $t\mapsto \vol(\chi_t)$ satisfies
    \begin{itemize}
        \item [(i)] $\vol(\chi_t)$ is smooth, and $\vol(\chi_t)^{-1/n}$ is concave;
        \item[(ii)]  $\vol(\chi_t)^{-1/n}$ is affine iff $\bar\chi_1=c\bar\chi_0$ for some $c>0$.
    \end{itemize}
\end{thmD}
We remark that, as in~\loccit, the geodesics need not have constant speed.
The proof relies on the study of relative measures between two norms, analogous to the one in~\cite{BLQ}, which further rests on closely related relative limit measures in the K-stability setting, see~\cite{BLXZ}. To compare this with our previous discussions on isometric isomorphisms between $\cT^\hom_\Q$ and $\cH(X)$, the Fubini--Study operator extends automatically to the space of bounded negative norms, but the resulting function $\FS(\chi)$ is only an lsc function on $X^\na$ in general, and hence need not be psh. However, if $X$ has the envelope property, \eg when $X$ is smooth, then the upper semicontinuous (usc) regularization of $\FS(\chi)$ will be a function in $\cE_\uparrow(X)\subset\cE(X)$, where $\cE_\uparrow(X)$ is the space of psh functions that can be regularized from below. In this case, $\FS$ extends to a surjective isometry $\FS^\star\colon (\Nbdd, \de_1)\twoheadrightarrow (\cE_\uparrow(X), \de_1)$, which restricts to an isometric isomorphism on $\cN^s$, with inverse given by $\IN$.

\medskip
\noindent\textbf{Metric completions with respect to Darvas metrics.}
Let us end with a brief discussion of metric completions with respect to Darvas metrics, referring the reader to~\S\ref{sec:metriccomp} for details. 

First, as in the case of polarized projective varieties, using decreasing approximation by functions in $\cH(X)$, we extend $\de_1$ uniquely to a metric on $\cE(X)$ that is continuous along decreasing nets.
It is easy to see that $(\cE(X), \de_1)$ will never be complete since the completion always contains the zero function. A slight modification as in the definition of $\bar\ii$-functional (see also~\cite[\S12.1]{BJpp}) prevents this degeneration,  and we may ask for the completion with respect to the modified metric. More precisely, for $\f, \p$, we set 
\[
\du(\f, \p)\=\de_1(\f, \p)+|\log\vol(\f)-\log\vol(\p)|.
\]
By taking numerical classes, we get a natural continuous map 
\[
(\cE(X), \du)\to (\bar\cE(X), \de_\infty).
\]
The Monge--Amp\`ere operator therefore extends this uniquely to a continuous map from the metric completion of $(\cE(X),\du)$ to $(\cC_\infty(X),\de_\infty)$, with dense image.

We do not expect the space $(\cE(X), \du)$ to be complete in general. 
However, when $X$ is smooth, has no non-constant global invertible functions, and admits a smooth compactification with irregularity zero, we have $\cE(X)=\bar\cE(X)$.
In this case, using Theorem~A', we prove that $(\cE(X), \du)$ is complete; see Proposition~\ref{prop:Ecomplete}.
This is for example the case when $X$ is an affine open subset of a smooth Fano variety, with no non-constant global invertible functions.

Alternatively, motivated by the rooftop characterization of the $\de_1$ metric on $\cH(X)$, we may define a metric on $\BBPF_{n-1}(X)$, by setting
\[
\de_1(\alpha, \beta)\=\vol(\alpha)+\vol(\beta)-2\vol(\alpha\wedge \beta),
\]
where $\vol(\alpha\wedge\beta)\=\sup\{\vol(\gamma)\mid \gamma\le \alpha, \beta, \gamma \text{ is a nef b-Cartier divisor class} \}$. 
We similarly set, for each $\alpha, \beta\in \BBPF_{n-1}(X)$,
\[
\du(\alpha, \beta)\=\de_1(\alpha, \beta)+|\log\vol(\a)-\log\vol(\b)|.
\]
This restricts to a metric on $\bar \cE(X)$.
By comparing this metric with $\de_\infty$, one can show that $(\BBPF_{n-1}(X), \du)$ is in fact complete, and thus the metric completion of $(\bar \cE(X), \du)$ is canonically homeomorphic to $\cC_\infty(X)$ via Theorem~A'.

\medskip\noindent\textbf{Organization.}
After setting up basic terminology for compactifications and recalling numerical notions in the Riemann--Zariski space, we define and study $\de_\infty$ metrics on numerical classes and prove Theorem~A in~\S\ref{sec:prelim}. The Fubini--Study functions are defined and studied in~\S\ref{sec:FSfcn}. In~\S\ref{sec:psh}, we extend the class of Fubini--Study functions to psh functions, and study the topology and envelope properties. The intersection pairing and Monge--Amp\`ere operator are defined in~\S\ref{sec:numint}, in which we also study the $\de_\infty$ metric on $\bar\cE(X)$. In~\S\ref{sec:curvinf}, we define and study curve classes at infinity, and prove Theorem~A'. The notion of a bounded negative norm is introduced in~\S\ref{sec:norm}, where we also investigate the finite generation property and prove the first part of Theorem~B. In~\S\ref{sec:volume}, we define the volume of a bounded negative norm, introduce the $\de_1$ pseudo-metric, and prove most of Theorem~D. The Fubini--Study and infimum operators, as well as saturation, are studied in detail in~\S\ref{sec:FSop}, where we also complete the proof of Theorems~B, C and D. Finally, in~\S\ref{sec:metriccomp} we extend the Darvas metrics to $\cE(X)$ and $\bar\cE(X)$, study the regularized Fubini--Study operator and metric completions with respect to Darvas metrics.

\medskip\noindent\textbf{AI disclosure.}
ChatGPT was used to provide Example~\ref{exam:curvmeasneq}, Lemma~\ref{lem:iquasitriangle}, to strengthen the statement of Proposition~\ref{prop:volbounded}, to correct the proof of the last assertion in Proposition~\ref{prop:ft}, and to proofread.

\medskip\noindent\textbf{Acknowledgements.}
The first author was supported by grants DMS-2154380 and  DMS-2452797 from the NSF. The second author was supported by DMS-1926686 and the Simons Collaboration Grant \emph{Moduli of varieties}.

\section{Polarized compactifications and cycle classes}\label{sec:prelim}
We work over an algebraically closed field $k$ of characteristic zero. Let $X=\Spec R$ be an affine variety (reduced and irreducible) of dimension $n\ge1$.
\subsection{Compactifications}\label{sec:comp}
By a \emph{compactification} of $X$ we mean a projective variety $\bar{X}$ together with an isomorphism $\iota\colon X\hto\bar{X}$ onto a dense open affine subset of $\bar{X}$. 
For simplicity, we view $X$ as an open subset of $\bar{X}$, and set $\partial\bar{X}:=\bar{X}\setminus X$ with the reduced scheme structure. 
By~\cite[Proposition~1]{Goodman}, 
every irreducible component of $\partial\bar{X}$ has codimension 1. 

If $\bar X$ is a compactification of $X$, and $Y\subset X$ a subvariety, then the closure $\bar{Y}$ in $\bar X$ is a compactification of $Y$.

Given compactifications $\bar{X}$, $\bar{X}'$, we say $\bar{X}'$ \emph{dominates} $\bar{X}$ if the birational map $\bar{X}'\dashrightarrow\bar{X}$ induced by the identity map on $X$ is a morphism, that we then call a \emph{morphism of compactifications}. Any two compactifications are dominated by a third.

We say a compactification $\bar{X}$ is \emph{polarizable} if there exists an ample effective Cartier divisor $D$ on $\bar{X}$ with support equal to $\partial\bar{X}$. Equivalently, $\bar{X}$ is polarizable if there exists an embedding $\iota\colon\bar{X}\hto\P^N$ with $X=\iota^{-1}(\A^N)$. Polarizable compactifications exist: just pick any closed immersion $X\hto\A^N$, let $\bar{X}$ be the (scheme-theoretic) closure of $X$ in $\P^N$, and $D$ the pullback of the hyperplane at infinity.

In general, not every compactification is polarizable, but we have
\begin{lem}\label{lem:hasample}
    Any compactification of $X$ is dominated by a polarizable compactification.
\end{lem}
\begin{proof}
Pick any polarizable compactification $\bar{X}''$ of $X$, and an ample effective Cartier divisor $D''$ on $\bar{X}''$ with support equal to $\partial\bar{X}''$. Let $\bar{X}$ be any other compactification, and consider the birational map $\bar{X}''\dashrightarrow\bar{X}$ induced by the identity on $X$. There exists a coherent ideal sheaf $\cJ$ on $\bar{X}''$ that is trivial on $X$ such that if $\pi\colon\bar{X'}\to\bar{X}''$ is the blowup along $\cJ$, then $\bar{X}'$ dominates $\bar{X}$. Write $\cJ\cdot\cO_{\bar{X}'}=\cO_{\bar{X}'}(-E')$ for an effective Cartier divisor $E'$ on $\bar{X}'$ supported on $\partial\bar{X}'$. Then $-E'$ is $\pi$-ample, and for $m\gg1$, the Cartier divisor $D':=m\pi^*D''-E'$ is ample, effective, and its support equals $\partial\bar{X'}$.
\end{proof}

\begin{defi}
    A \emph{polarized compactification} of $X$ is a polarizable compactification $\bar{X}$ of $X$ together with an ample effective $\Q$-divisor $D$ on $\bar{X}$ with support equal to $\partial\bar{X}$.
\end{defi}

We say that a compactification $\bar X$ is \emph{integrally closed} if it is integrally closed in $X$, that is, $\bar{X}$ can be covered by finitely many open affine subsets $U$ such that $\cO_{\bar{X}}(U)$ is integrally closed in $\cO_X(X\cap U)$. When $X$ is normal, this means precisely that $\bar{X}$ is normal.
\begin{exam}
    If $Y$ is a nodal cubic, and $p\in Y$ is a closed point different from the node, then $\bar{X}:=Y$ is an integrally closed compactification of $X:=Y\setminus\{p\}$.
\end{exam}
If $\bar{X}$ is an integrally closed compactification of $X$, and $\xi$ is the generic point of an irreducible component of $\partial\bar X$, then the local ring $\cO_{\bar{X},\xi}$ is a DVR. In particular, $\xi$ defines a divisorial valuation on the function field $k(X)$.
Any compactification $\bar{X}$ admits an \emph{integral closure} $\bar{X}'$. This is the unique integrally closed compactification of $X$ that dominates $\bar{X}$ via a finite morphism; see~\cite[Tag~0BAK]{stacks-project}. When $\bar{X}$ is polarizable, so is $\bar{X}'$.
As in~\cite[proof of Lemma~1.12]{BJpp}, we have a version of Zariski's main theorem.
\begin{lem}\label{lem:Zmain}
    If $\mu\colon\bar{X}'\to\bar{X}$ is a morphism of compactifications, with $\bar{X}$ integrally closed, then $\mu_*\cO_{\bar{X}'}=\cO_{\bar{X}}$.
\end{lem}
\begin{proof}
    By coherence of direct images, $\mu_*\cO_{\bar{X'}}$ is a finite $\cO_{\bar{X}}$-module. If $U\subset\bar{X}$ is an open set, a section $s\in(\mu_*\cO_{\bar{X}'})(U)$ is a rational function on $U$ that is regular on $X\cap U$ and integral over $\cO_{\bar{X}}(U)$. As $\bar{X}$ is integrally closed, we get $s\in\cO_{\bar{X}}(U)$.
\end{proof}
\begin{defi}
    We write $\cP\cC$ for the set of polarized compactifications $(\bar X,D)$ of $X$, and $\cP\cC^{\mathrm{int}}$ for the subset with $\bar X$ integrally closed.
\end{defi}
More generally, a \emph{projective model of $X$} is a normal projective variety $Y$ together with a birational morphism $Y\to \bar X$ for some compactification $\bar X$ of $X$. 
\subsection{Cartier divisors at infinity}\label{sec:Cdivinfty}
If $\bar{X}$ is a compactification of $X$, we denote by $\CDiv_\infty(\bar{X})$ the group of \emph{Cartier divisors at infinity} on $\bar{X}$, that is, Cartier divisors with support contained in $\partial\bar{X}$. Such a divisor is \emph{strictly effective} if it is effective and its support equals $\partial\bar{X}$.

A morphism $\bar{X}'\to\bar{X}$ of compactifications with $\bar X$ integrally closed induces an injective pullback map $\CDiv_\infty(\bar{X})\to\CDiv_\infty(\bar{X}')$ which preserves the class of effective (resp.\ strictly effective, resp.\ semiample, resp.\ big) divisors. 
A \emph{Cartier b-divisor at infinity on $X$} is an element of the group 
\[
\CbDiv_\infty(X):=\varinjlim_{\bar{X}}\CDiv_\infty(\bar{X}).
\]
When $\bar X$ is integrally closed, we will often identify an element $\CDiv_\infty(\bar X)$ with its image in $\CbDiv_\infty(X)$; this image is said to be \emph{determined on $\bar{X}$} and \emph{represented} by the Cartier divisor at infinity on $\bar{X}$. 

By the remarks above, it makes sense to talk about the class of effective (resp.\ strictly effective, resp.\ semiample, resp.\ big) Cartier b-divisors at infinity on $X$. Any strictly effective Cartier b-divisor at infinity on $X$ is big.

A morphism $h\colon X\to X'$ of affine varieties induces a map $h^*\colon\CbDiv_\infty(X')\to\CbDiv_\infty(X)$, which preserves the class of effective (resp.\ semiample) divisors. When $h$ is finite (resp.\ finite and surjective), $h^*$ preserves the class of strictly effective (resp.\ big) divisors. 

An element of $\CbDiv_\infty(X)_\Q:=\CbDiv_\infty(X)\otimes_\Z\Q$ (resp.\ $\CbDiv_\infty(X)_\R\=\CbDiv_\infty(X)\otimes_\Z\R$) is called $\Q$-Cartier (resp.\ $\R$-Cartier) b-divisor at infinity.
\begin{cor}\label{cor:diffsa}
    Any divisor in $\CbDiv_\infty(X)$ can be written as a difference of two divisors that are semiample and strictly effective.
\end{cor}
\begin{proof}
    Assume $D$ is a divisor at infinity on a compactification $\bar{X}$ of $X$. By Lemma~\ref{lem:hasample} we may assume that $\bar{X}$ admits an ample and strictly effective divisor $H$ at infinity. It suffices to note that $mH+D$ and $mH$ are ample and strictly effective for $m\gg0$.
\end{proof}
\begin{lem}\label{lem:sabig}
    Any semiample and strictly effective Cartier divisor at infinity on $X$ is determined by a unique integrally closed polarized compactification $(\bar X, D)$ of $X$.
\end{lem}
\begin{proof}
    After passing to a multiple, there exists an integrally closed compactification $\bar{X}'$ such that the Cartier divisor is represented by a basepoint free and strictly effective Cartier divisor $D'$ at infinity on $\bar{X}'$. 
    Then the section ring $R(\bar{X}',D')$ is a finitely generated $k$-algebra, and we obtain a projective morphism $h\colon\bar{X}'\to\bar{X}''\coloneqq\Proj R(\bar{X}',D')$. Now     
    $D'$ is the pullback of an ample, strictly effective divisor $D''$ at infinity on $\bar{X}''$ and we can identify $X$ with the complement of $D''$ in $\bar{X}''$. Then $h\colon\bar{X}'\to\bar{X}''$ is a morphism of compactifications, whose Stein factorization yields an integrally closed compactification $\bar X$ of $X$. 
    Further, $D'$ is the pullback of an ample divisor $D$ on $\bar X$, which in turn is the pullback of $D''$ via a finite morphism. Suppose $(\bar X_1, D_1)$ and $(\bar X_2, D_2)$ are two integrally closed polarized compactifications representing the same     
    semiample strictly effective Cartier divisor at infinity. Then we may pick a compactification $\bar X_3$ dominating both $\bar X_1$ and $\bar X_2$ such that $D_3=h_1^*D_1=h_2^*D_2$, where $h_1\colon \bar X_3\to \bar X_1$ and $h_2\colon \bar X_3\to \bar X_2$. Since both $\bar X_1, \bar X_2$ are integrally closed, we infer from Lemma~\ref{lem:Zmain} that $R(\bar X_3, D_3)=R(\bar X_1, D_1)=R(\bar X_2, D_2)$, and hence $\bar X_1\cong \bar X_2$.
\end{proof}

\subsection{Ideals at infinity}\label{sec:idealinf}
An \emph{ideal at infinity} on a compactification $\bar{X}$ is a coherent ideal sheaf $\fa$ on $\bar{X}$ that is trivial on $X$. We say $\fa$ is \emph{strictly effective} if $\fa_\xi\ne\cO_\xi$ for all points $\xi\in\partial\bar{X}$.
If $\bar{X}'\to\bar{X}$ is a morphism of compactifications, then an ideal $\fa$ at infinity on $\bar{X}$ pulls back to an ideal $\fa'$ at infinity on $\bar{X}'$, strictly effective when $\fa$ is. We can always choose $\bar{X}'$ integrally closed such that $\fa'=\cO_{\bar{X}'}(-D')$ for some effective divisor $D'$ at infinity on $\bar X'$, and say that $\fa$ is \emph{semiample} if $D'$ is semiample. 
Similarly, a \emph{fractional ideal at infinity} on a compactification $\bar{X}$ of $X$ is a coherent fractional ideal sheaf that is trivial on $X$. 
The next result is proved in the same way as~\cite[Lemma~1.23]{nakstab1}.
\begin{lem}\label{lem:critnopole}
    Assume $\bar{X}$ is an integrally closed compactification of $X$, and let $\fa$ be a fractional ideal on $\bar{X}$ at infinity. Then $\fa\subset\cO_{\bar{X}}$ iff $\ord_E(\fa)\ge0$ for every irreducible component $E$ of $\bar{X}\setminus X$.
\end{lem}

\subsection{Weil divisors at infinity}\label{sec:divinfty}
For any compactification $\bar X$ of $X$, we write $\WDiv_\infty(\bar X)$ for the group of \emph{Weil divisors at infinity on $\bar X$}. More precisely, it is the free group generated by the set $\PDiv_\infty(\bar X)$ of \emph{prime divisors at infinity}, that is, irreducible components of $\partial\bar X$. In other words, $\WDiv_\infty(\bar X)$ is the set of $\Z$-valued functions on $\PDiv_\infty(\bar X)$.
Any morphism $\pi\colon\bar X'\to \bar X$ of compactifications induces a group homomorphism $\pi_*\colon\WDiv_\infty(\bar X')\to \WDiv_\infty(\bar X)$. When $\bar X$ is integrally closed, we also have an injective map $\PDiv_\infty(\bar X)\hto\PDiv_\infty(\bar X')$, and $\pi_*$ corresponds to restriction of functions.

The group of \emph{Weil b-divisors at infinity} is defined as the projective limit
\[
\WbDiv_\infty(X):=\varprojlim_{\bar X} \WDiv_\infty(\bar X).
\]
Concretely, an element $B\in\WbDiv_\infty(X)$ is a family $(B_{\bar X})_{\bar X}$ of Weil divisors at infinity compatible under pushforwards; we sometimes call $B_{\bar X}$ the \emph{incarnation} of $B$ on $\bar{X}$. It is equivalent to let $\bar{X}$ range over integrally closed compactifications, these being cofinal.

If $\bar X'\to \bar X$ is a morphism of integrally closed compactifications, strict transforms define an injective map $\PDiv_\infty(\bar X)\to \PDiv_\infty(\bar X')$.
The set of \emph{prime b-divisors at infinity} is defined as 
\[
\PbDiv_\infty(X):=\varinjlim_{\bar X} \PDiv_\infty(\bar X),
\]
where $\bar{X}$ ranges over integrally closed compactifications. It follows that $\WbDiv_\infty(X)$ is the set of $\Z$-valued functions on $\PbDiv_\infty(X)$.

Replacing $\Z$ by $\Q$ or $\R$, we obtain a similar description of the vector spaces $\WbDiv_\infty(X)_\Q\=\varprojlim_{\bar X} \WDiv_\infty(\bar X)\otimes_\Z\Q$ and $\WbDiv_\infty(X)_\R\=\varprojlim_{\bar X} \WDiv_\infty(\bar X)\otimes_\Z\R$; the latter space is equipped with the inverse-limit topology, equivalently the topology of pointwise convergence on prime b-divisors at infinity.

For any integrally closed compactification $\bar{X}$ of $X$, we have an injective cycle map $\CDiv_\infty(\bar{X})\to\WDiv_\infty(\bar{X})$, and this gives rise to an injective group homomorphism 
\[
\CbDiv_\infty(X)\hto\WbDiv_\infty(X),
\]
and similarly after tensoring by $\Q$ or $\R$.

A Cartier divisor $B\in\CbDiv_\infty(X)$ is effective (resp.\ strictly effective) iff the corresponding function $\p_B\colon\PbDiv_\infty(X)\to\Z$ is nonnegative (resp.\ strictly positive).

\subsection{General Weil b-divisors}\label{sec:Weilb}
Ignoring the role of ``infinity", we also have the more general notion of a Weil b-divisor in the sense of~\cite{bdff}. Namely, for any projective variety $Y$, let $\WDiv(Y)$ be the group of Weil divisors on $Y$, and set 
\[
\WbDiv(X):=\varprojlim_Y\WDiv(Y),
\]
where $Y$ ranges over projective varieties together with a birational map $Y\dashrightarrow X$. Here it suffices to let $Y$ range over projective models of $X$.
We may even assume $Y$ smooth. Let $\PDiv(Y)$ denote the set of prime divisors $E\subset Y$. If $Y,Y'$ are projective models of $X$ such that the birational map $Y'\dashrightarrow Y$ is a morphism, and $Y$ is normal, then we have an injective map $\PDiv(Y)\hto\PDiv(Y')$ given by strict transform. We can therefore form the limit
\[
\PbDiv(X):=\varinjlim_Y\PDiv(Y),
\]
where $Y$ ranges over normal, or even smooth, projective models of $X$. Thus we can identify $\WbDiv(X)$ with the group of functions $\p\colon\PbDiv(X)\to\Z$ \emph{of divisorial type}, that is for each projective model $Y$ we have $\p(E)=0$ for all but finitely many prime divisors $E\subset Y$.

Similarly, $\WbDiv(X)_\Q\=\varprojlim_Y\WDiv(Y)\otimes_\Z\Q$ (resp.\ $\WbDiv(X)_\R\=\varprojlim_Y\WDiv(Y)\otimes_\Z\R$) can be viewed as the space of $\Q$-valued (resp.\ $\R$-valued) functions on $\PbDiv(X)$ of divisorial type.

If $Y$ is a normal projective model of $X$, dominating some integrally closed compactification $\bar X$, then $\PDiv_\infty(\bar X)$ injects into $\PDiv(Y)$. This gives rise to an injection 
\[
\PbDiv_\infty(X)\hto\PbDiv(X),
\]
which induces both a restriction map $\WbDiv(X)\to\WbDiv_\infty(X)$ and an injective map
\[
\WbDiv_\infty(X)\hto\WbDiv(X),
\]
the latter given by extending a function $\p$ on $\PbDiv_\infty(X)$ to $\PbDiv(X)$ by setting it equal to zero on $\PbDiv(X)\setminus\PbDiv_\infty(X)$; such an extension is automatically of divisorial type.

Summarizing the above discussion up, we have the following refinement of~\cite[Lemma~1.1]{bdff}.
\begin{lem}\label{lem:divfcn}
    The map $B\mapsto \psi_B$ defines a vector space isomorphism between the space of $\R$-Weil b-divisors and the space of functions on $\PbDiv(X)$ of divisorial type. Further, it restricts to an isomorphism between the space of $\R$-Weil b-divisors at infinity and the space of functions on $\PbDiv_\infty(X)$, and an $\R$-Cartier divisor at infinity is strictly effective iff $\psi>0$ on $\PbDiv_\infty(X)$.
\end{lem}
\subsection{Numerical classes of divisors}\label{sec:numcld}
As in~\cite{bdff}, we define the space of \emph{Cartier b-divisor classes} of $X$ by
\[
\CNum(X)\=\varinjlim_Y \Num(Y),
\]
where $Y$ runs through all projective models of $X$ and the maps are given by pullbacks, and $\Num(Y)$ denotes the vector space of numerical classes of $\R$-Cartier divisors on $Y$. The space $\CNum(X)$ is endowed with the inductive-limit topology.
We also have the space of \emph{Weil b-divisor classes}, given by
\[
\WNum_{n-1}(X)\=\varprojlim_Y\WNum_{n-1}(Y),
\]
where $Y$ runs through all \emph{smooth} projective models of $X$ and the maps are given by pushforwards. We equip $\WNum_{n-1}(X)$  with the natural inverse limit topology.

There are natural maps $\CbDiv_\infty(X)_\R\to \CNum(X)$ and $\WbDiv_\infty(X)_\R\to \WNum_{n-1}(X)$. Indeed, if $D\in \CDiv_\infty(\bar X)$ for some compactification $\bar X$, then $D$  defines a Cartier b-divisor via pulling back to all higher projective models of $X$. Taking its numerical classes on each model gives a Cartier b-divisor class. Similarly, given a Weil b-divisor at infinity, the associated numerical class is given by its numerical class on each smooth model. 
\begin{lem}\label{lem:nepc}
    Suppose $(\bar X_1,D_1)$, $(\bar X_2,D_2)$ are two integrally closed polarized compactifications such that the images of $D_1$ and $D_2$ in $\CNum(X)$ coincide. Then $\bar X_1=\bar X_2=:\bar X$, and $D_2-D_1\in\CDiv_\infty(\bar{X})_\Q$ is numerically trivial.
\end{lem}
\begin{proof}
Let $\bar X_3$ be the closure of image of $X$ in the 
$\bar X_1\times\bar X_2$ under the diagonal map, and denote the projections by $p_i\colon\bar X_3\to\bar X_i$. By assumption,
$p_1^*D_1\equiv p_2^*D_2$, as can be checked after pullback
to a smooth projective model dominating $\bar X_3$.

If a curve $C\subset\bar X_3$ is contracted by $p_1$, then
$(p_2^*D_2\cdot C)=0$. Since $D_2$ is ample, $p_2$ also
contracts $C$, contradicting the inclusion
$\bar X_3\subset\bar X_1\times\bar X_2$. Thus $p_1$ has
no positive-dimensional fibers and is finite, and similarly
for $p_2$. Lemma~\ref{lem:Zmain} now implies that both $p_i$
are isomorphisms. Under the resulting identification
$\bar X_1=\bar X_2\eqqcolon\bar X$, we have $D_1\equiv D_2$,
so $D_2-D_1\in\CDiv_\infty(\bar X)_\Q$ is numerically trivial.
\end{proof}

The map  $\WbDiv(X)\to\WNum_{n-1}(X)$ is not continuous in general, see~\cite[Example~1.10]{bdff}. However, as noted in~\cite[Lemma~1.12]{bdff} and~\cite[Lemma~1.8]{BJGreen}, we have

\begin{lem}\label{lem:numclasscts}
    The map $\WbDiv_\infty(X)_\R\to\WNum_{n-1}(X)$ is continuous.
\end{lem}
\begin{proof}
    It suffices to prove that for every compactification $\bar{X}$ of $X$, and any smooth projective model $\pi\colon Y\to\bar X$, the composition 
    \[    \WbDiv_\infty(X)_\R\to\WbDiv(X)_\R\to\WDiv(Y)_\R\to\WNum_{n-1}(Y)
    \] 
    is continuous. Let $E_1,\dots,E_N$ be the prime divisors on $Y$ contained in $\pi^{-1}(\partial\bar X)$. The image of $B\in\WbDiv_\infty(X)_\R$ in $\WNum_{n-1}(Y)$ is of the form $\sum_{j=1}^Nb_j(B)[E_j]$, where  the coefficients $b_j(B)$ depend continuously on $B$. The result follows.
\end{proof}
As in~\cite[Lemma~1.9]{bdff}, there is a natural continuous injection $\CNum(X)\to \WNum_{n-1}(X)$ with dense image. For a divisor supported at infinity, its numerical class is determined by testing on curves at infinity. 
\begin{lem}\label{lem:numtrivinf}
    Assume $n\ge 2$, and let $D\in\CDiv_\infty(\bar X)$ for some compactification $\bar X$ of $X$. Then $D$ is numerically trivial iff $D\cdot C=0$ for all curves $C\subset \bar X\setminus X$.
\end{lem}
\begin{proof}
    We may assume $\bar X$ is polarizable. Let $H$ be a strictly effective very ample divisor at infinity on $\bar X$. If $n=2$, then $(D\cdot H)=0$ and $(D^2)=0$ by assumption. Thus $D\equiv 0$ by the Hodge index theorem. 
    When $n\ge 3$, pick very ample divisors $L_1, \dots, L_N$ such that $H, L_1, \dots, L_N$ generate $\NS(\bar X)_\Q$. Then, for general members $H_i\in |L_i|$, we have $(D\cdot H_i\cdot H^{n-2})=0$. By linearity this implies $(D^2\cdot H^{n-2})=0$. Since we also have $(D\cdot H^{n-1})=0$, we again have $D\equiv 0$ by the Hodge index theorem.
\end{proof}
We say that a Weil b-divisor $D\in\WbDiv_\infty(X)_\R$ at infinity is \emph{numerically trivial} if its image in $\WNum_{n-1}(X)$ vanishes.
\begin{lem}\label{lem:ntrivCDiv}
    If $B\in \WbDiv_\infty(X)_\R$ is numerically trivial, then $B$ is an $\R$-Cartier b-divisor in the sense that there exists a smooth projective model $Y$ of $X$ such that the trace $B_Y$ is numerically trivial, and such that $B_{Y'}$ is the pullback of $B_Y$ for any smooth projective model $Y'$ dominating $Y$. 
    If, in addition, $X$ is normal with isolated singularities, then $B\in \CbDiv_\infty(X)_\R$.
\end{lem}

\begin{proof}
    Let $\pi\colon Y'\to Y$ be two smooth models of $X$. Then $B_{Y'}-\pi^*\pi_*B_{Y'}$ is both $\pi$-exceptional and numerically trivial. By the negativity lemma~\cite[Lemma~3.39]{KollarMori}, we have $B_{Y'}=\pi^*B_Y$ and hence $B$ is $\R$-Cartier. If $X$ is normal with isolated singularities, then we may pick $\bar X$ such that $\bar X$ is smooth at infinity and apply the same argument above.
\end{proof}

As in~\cite{bdff}, one can define nef Weil b-divisor classes as weak limits of nef Cartier b-divisor classes. It is equivalent to the following characterization by~\cite[Lemma~2.10]{bdff}(see also~\cite[Theorem~5.2]{DangFavre}).
\begin{defi}
    A Weil b-divisor class is \emph{nef} (resp. \emph{pseudoeffective}) if its incarnation on each smooth projective model is movable (resp. pseudoeffective). 
\end{defi}
We write 
\[
\BPF_{n-1}(X)\subset\WNum_{n-1}(X)
\]
for the set of nef Weil b-divisor classes. If $\alpha\in\WNum_{n-1}(X)$ is pseudoeffective, we will also write $\alpha\ge 0$.
By~\cite[Theorem~A]{DangFavre}, any $\alpha\in \BPF_{n-1}(X)$ can be written as a decreasing limit of nef Cartier b-divisor classes. Let us remark that the regularization is obtained by multiplier ideal approximations on  smooth models and hence works in our setting when $X$ is affine and singular, since it only depends on the Riemann--Zariski space of $X$.

Using this regularization, it was shown in \loccit \ that the intersection pairing extends to $\BPF_{n-1}(X)$, and is $\R_{>0}$-linear, increasing in each argument, upper semicontinuous, and continuous along decreasing limits. As a consequence, the classical inequalities deriving from the Hodge Index Theorem extend from nef Cartier classes to $\BPF_{n-1}(X)$. 
For example, the \emph{Khovanskii-Teissier inequality} states that $j\mapsto\log(\a^j, \b^{n-j})\in\R\cup\{-\infty\}$ is concave for any $\a,\b\in\BPF_{n-1}(X)$. In particular,
\begin{equation}\label{equ:KTclass2}
    (\b,\a^{n-1})\ge(\b^n)^{\frac1n}(\a^n)^{\frac{n-1}{n}}
\end{equation}
for $\a,\b\in\BPF_{n-1}(X)$.

We also have 
\begin{equation}\label{equ:KTclass1}
    (\alpha_1, \dots, \alpha_n)\ge (\alpha_1^p, \alpha_{p+1}, \dots, \alpha_n)^{1/p}\dots (\alpha_p^p, \alpha_{p+1}, \dots, \alpha_n)^{1/p},
\end{equation}
for any $\alpha_1, \dots, \alpha_n\in \BPF_{n-1}(X)$ and any $1\le p\le n$.

\subsection{Numerical classes of curves}\label{sec:numclc}
Assume now $n\ge 2$.
The space of \emph{Weil b-curve classes} is defined as
\[
\Curv(X)\=\varprojlim_Y\Curv(Y),
\]
where the limit is taken over all smooth projective models $Y$ of $X$. Via the intersection pairing, every Weil b-curve class in $\Curv(X)$ can be identified with a unique linear form on $\CNum(X)$. Similarly, the weak topology on $\Curv(X)$ is the natural inverse limit topology. More concretely, a Weil b-curve class consists of a family $(\mu_Y)_Y$ of curve classes on smooth models $Y$, compatible under pushforwards. A net of curve classes $\mu_j\to \mu$ weakly if $\mu_{j, Y}\to \mu_Y$ on each smooth model $Y$.
We also recall the following notions of positivity of numerical classes, see Definition~1.2 and Theorem~5.2 in~\cite{DangFavre}.

\begin{defi}
    A Weil b-curve class is \emph{BPF} (resp. \emph{pseudoeffective}) if its incarnation on each smooth projective model is movable (resp. \emph{pseudoeffective}). We write $\BPF_1(X)$ for the set of BPF curve classes.
\end{defi}
If $\mu,\nu$ are Weil b-curve classes, we write $\mu\le\nu$ iff $\nu-\mu$ is pseudoeffective.
We also recall the following characterization of BPF and pseudoeffective curve classes from~\cite[Proposition~5.6]{DangFavre}.
\begin{prop}\label{prop:bpfposlin}
    A curve class $\mu\in\WNum_1(X)$ is BPF (resp. \emph{pseudoeffective}) iff $(\alpha, \mu)\ge 0$ for any pseudoeffective (resp. \emph{nef}) Cartier b-divisor class $\alpha\in \CNum(X)$.
\end{prop}
\begin{proof}
    If $\mu\in \WNum_1(X)$ is pseudoeffective, and $\alpha$ a nef Cartier b-divisor, then we can pick a smooth model $Y$ on which $\alpha=\alpha_Y$ is determined, $(\alpha, \mu)=(\alpha_Y, \mu_Y)\ge 0$ since $\mu_Y$ is pseudoeffective. Similarly if $\mu$ is BPF, and $\alpha$ is pseudoeffective, determined on a smooth model $Y$, then~\cite[Corollary~2.5]{BDPP} yields $(\alpha, \mu)=(\alpha_Y, \mu_Y)\ge 0$ since $\mu_Y$ is movable. The converse follows from the same duality on each smooth model.
\end{proof}
\begin{exam}\label{exam:Mink}
    As in~\cite{DangFavre}, any class $\a\in\BPF_{n-1}(X)$ defines a curve class $\mu:=\a^{n-1}\in\BPF_1(X)$.
    Indeed, if $\b\in\CNum(X)$, write $\b=\b_1-\b_2$ with $\b_i$ nef and Cartier, and define $(\b,\mu):=(\a^{n-1},\b_1)-(\a^{n-1},\b_2)$. By the properties of the intersection pairing, this is well defined, and $(\b,\mu)\ge0$ when $\b$ is psef; so $\mu\in\BPF_1(X)$.
\end{exam}
The discussion above holds for any quasi-projective variety. Let us now discuss the role of divisors at infinity, when $X$ is an affine variety.
In general, a curve class in $\WNum_1(X)$, viewed as a linear form on $\CNum(X)$, is not uniquely determined by its values on classes of Cartier divisors at infinity. For example, if $X=\A^2$, $\bar{X}=\P^2$, $H=\P^2\setminus\A^2$, and $\pi\colon Y\to\bar X$ is the blowup at a point in $\A^2$, with exceptional divisor $E$, then the curve classes defined by $2\pi^*H$ and $2\pi^*H-E$ are both BPF and define the same linear functional on $\CbDiv_\infty(X)$. 

However, the following lemma shows that certain curve classes \emph{are} uniquely defined by their values on Cartier divisors at infinity.
\begin{lem}\label{lem:curvmeascart}
    Assume $n\ge2$, and let $D,D'\in\CbDiv_\infty(X)_\R$ be big and nef such that $(D^{n-1}\cdot D'')=((D')^{n-1}\cdot D'')$ for all $D''\in\CbDiv_\infty(X)_\R$. Then $D$ and $D'$ have the same image in $\CNum(X)$.
\end{lem}
\begin{proof}
    Set $\la_j=(D^{n-j}\cdot (D')^j)$ for $0\le j\le n$. The assumption yields $\la_0=\la_{n-1}$ and $\la_1=\la_n$. 
    Pick a compactification $\bar{X}$ of $X$ on which $D$ and $D'$ are determined, and a smooth model $Y\to\bar{X}$. The Khovanskii--Teissier inequalities on $Y$ give $\la_1^n\ge\la_0^{n-1}\la_n$ and $\la_{n-1}^n\ge\la_0\la_n^{n-1}$. Thus $\la_0\ge\la_n$ and $\la_n\ge\la_0$, so that $\la_0=\la_1=\la_{n-1}=\la_n$. It now follows from the Diskant inequality (see~\cite[Theorems~D and~F]{BFJ09}) that the images of $D$ and $D'$ in $\Num(Y)$ are proportional; hence equal, and the result follows.
\end{proof}
\subsection{Strong topology for divisor classes}\label{sec:strongPnum}
We declare a nef b-Weil divisor $\a\in\BPF_{n-1}(X)$ to be \emph{big} if 
\[
\vol(\a):=(\a^n)>0,
\]
and we write
\[
\BBPF_{n-1}(X)\subset\BPF_{n-1}(X)
\]
for the set of big and nef b-Weil divisor classes; this is a convex cone by concavity of $\vol^{1/n}$.

The following result, that we refer to as \emph{Siu's inequality} shows that any two big and nef b-Weil divisor classes are comparable.
\begin{lem}\label{lem:Siuclass}
    For $\alpha, \beta\in \BPF_{n-1}(X)$ with $(\beta^n)>0$, we have
    \[
    \alpha\le n\frac{(\alpha, \beta^{n-1})}{(\beta^n)}\beta.
    \]
\end{lem}
This follows from a classical inequality of Siu when $\a$ and $\b$ are Cartier classes, see~\cite[Theorem~2.2.15]{pos1}. In general, it follows by approximation, see~\cite[Corollary~3.4]{DangFavre}.

Lemma~\ref{lem:Siuclass} implies that if $\a\in\BPF_{n-1}(X)$ then $\a\in\BBPF_{n-1}(X)$ iff $\a$ is big in the sense of~\cite{DangFavre}, that is, $\a$ dominates some multiple of any fixed big and nef Cartier divisor class.

Now consider $\a,\b\in\BBPF_{n-1}(X)$. As in~\cite[Definition~3.7]{BFJ09}, we define the \emph{pseudoeffective slope} 
\begin{equation}\label{equ:psefslope1}
   s(\alpha, \beta)\=\sup\{t>0\mid \alpha\ge t\beta\}\in(0,\infty),
\end{equation}
which in turn induces a natural metric on $\BBPF_{n-1}(X)$, given by
\begin{equation}\label{equ:dinfty1}
    \de_\infty(\alpha, \beta)\=\inf\{\delta\ge0\mid e^{-\delta}\alpha\le \beta\le e^{\delta}\alpha\}
    =\log\max\{\frac{1}{s(\alpha, \beta)}, \frac{1}{s(\beta, \alpha)}\}.
\end{equation}
It is clear that $\de_\infty$ is a well-defined metric, and hence defines a topology on $\BBPF_{n-1}(X)$ that we call the \emph{strong topology}. This is stronger than the \emph{weak topology}, which is the topology induced by $\BBPF_{n-1}(X)\subset\WNum_{n-1}(X)$. 

\begin{prop}\label{prop:divcomplete}
    The metric space $(\BBPF_{n-1},\de_\infty)$ is complete.
\end{prop}
The proof is an adaptation of the proof of Proposition~\ref{prop:curvcomplete} below; the details are omitted. Alternatively, Theorem~\ref{thm:metricequivnum} combined with Proposition~\ref{prop:curvcomplete} also gives completeness of $(\BBPF_{n-1}(X), \de_\infty)$.

\medskip
The pseudo-effective slope crucially enters in the \emph{Diskant inequality}, see~\cite[Theorem~3.6]{DangFavre}. If $\a,\b\in\BBPF_{n-1}(X)$, we have 
\begin{equation}\label{equ:Diskant}
    (\beta, \alpha^{n-1})^{\frac{n}{n-1}}-(\alpha^n)(\beta^n)^{\frac{1}{n-1}}
    \ge \left((\beta, \alpha^{n-1})^{\frac{1}{n-1}}-s(\a,\b)(\beta^n)^\frac{1}{n-1}\right)^n.
\end{equation}
Note that both sides of this inequality are non-negative, as a consequence of the Khovanskii--Teissier inequality~\eqref{equ:KTclass2}. Moreover, the Diskant inequality implies that equality holds in~\eqref{equ:KTclass2} iff $\a$ and $\b$ are proportional. 

Our next goal is to make the last observation quantitative.
As ``homogeneous" versions of functionals considered in the projective setting, we define
\begin{align*}
    \jj_\beta(\alpha)
    &\= \log \frac{(\alpha, \beta^{n-1})}{\vol(\alpha)^{1/n}\vol(\beta)^{1-1/n}}\\
    \ii(\alpha, \beta)
    &=\log \frac{(\alpha, \beta^{n-1})(\beta, \alpha^{n-1})}{\vol(\alpha)\vol(\beta)}
    =\jj_\beta(\alpha)+\jj_\alpha(\beta)
\end{align*}
for $\alpha, \beta\in \BBPF_{n-1}(X)$.
These functionals are non-negative by the Khovanskii-Teissier inequality. 
\begin{lem}\label{lem:ijequiv}
    For any $\alpha, \beta\in \BBPF_{n-1}(X)$, we have
    $\jj_\beta(\alpha)\le \ii(\alpha, \beta)\le n\jj_\beta(\alpha)$. Moreover, these three quantities are zero iff $\a$ and $\b$ are proportional.
\end{lem}
\begin{proof}
    The first inequality is clear. To see the other inequality, we use the Khovanskii--Teissier inequality to get
    \[
    (\alpha, \beta^{n-1})\ge (\alpha^{n-1}, \beta)^{\frac{1}{n-1}}(\beta^n)^{\frac{n-2}{n-1}}. 
    \]
    Thus,
    \[
    \exp\left((n-1)\jj_\beta(\alpha)\right)\ge \frac{(\alpha^{n-1}, \beta)(\beta^n)^{n-2}}{(\alpha^n)^{\frac{n-1}{n}}(\beta^n)^{\frac{(n-1)^2}{n}}} = \exp\left(\jj_\alpha(\beta)\right),
    \]
    and so
    \[
    \ii(\alpha, \beta)=\jj_\beta(\alpha)+\jj_\alpha(\beta)\le \jj_\beta(\alpha)+(n-1)\jj_\beta(\alpha)=n\jj_\beta(\alpha).
    \]
    The last statement is a consequence of the Diskant inequality. Indeed, the latter can be rewritten as 
    \begin{equation}\label{equ:Diskant2}
        s(\alpha, \beta)
        \ge\exp(\jj_\alpha(\beta))^{\frac{1}{n-1}}(\a^n)^{\frac{1}{n}}(\b^n)^{-\frac1n}-(\exp(\jj_\alpha(\beta))^{\frac{n}{n-1}}-1)^{\frac 1n}(\a^n)^{\frac{1}{n}}(\b^n)^{-\frac1n}
    \end{equation}
    for any $\a,\b\in\BBPF_{n-1}(X)$. Now suppose $J_\b(\a)=0$, which by what precedes is equivalent to $J_\a(\b)=0$. By homogeneity we may assume $(\a^n)=(\b^n)=1$. Then $s(\a,\b)\ge1$ and $s(\b,\a)\ge1$, and hence $\a=\b$.
\end{proof}

While we will not use it, let us for completeness prove that the $\ii$-functional satisfies a generalized triangle inequality, \cf~\cite[Theorem~1.8]{BBEGZ} and~\cite[Proposition~7.27]{BJpp}.
\begin{lem}\label{lem:iquasitriangle}
    For any $\a,\b,\g\in\BBPF_{n-1}(X)$, we have
    \[
    \ii(\a,\g)\le 2^{n+1}
    \bigl(\ii(\a,\b)+\ii(\b,\g)\bigr).
    \]
\end{lem}
\begin{proof}
    By scaling invariance, we may assume
    $(\a^n)=(\b^n)=(\g^n)=1$. Set
    \[
    m_k\=(\a^k,\g^{n-k}),\qquad
    V\=\vol\left(\frac{\a+\g}{2}\right)
    =2^{-n}\sum_{k=0}^n\binom nk m_k.
    \]
    By the Khovanskii--Teissier inequality, we have $m_k\ge1$.
    Thus Jensen's inequality gives
    \[
    \log V
    \ge 2^{-n}\sum_{k=0}^n\binom nk\log m_k
    \ge \frac n{2^n}\max\{\log m_1,\log m_{n-1}\}
    \]
    It follows that
    \[
    \ii(\a,\g)=\log m_1+\log m_{n-1}
    \le \frac{2^{n+1}}n\log V.
    \]
    On the other hand, if we set
    \[
    A\=(\a,\b^{n-1}),\qquad C\=(\g,\b^{n-1}),
    \]
    then the Khovanskii--Teissier inequalities again give $A,C\ge1$, and
    \[
    V^{1/n}
    \le \left(\frac{\a+\g}{2},\b^{n-1}\right)
    =\frac{A+C}{2}\le AC.
    \]
    We conclude that 
    \[
    \ii(\a,\g)
    \le 2^{n+1}(\log A+\log C)
    \le 2^{n+1}\bigl(\ii(\a,\b)+\ii(\b,\g)\bigr),
    \]
    completing the proof.
\end{proof}

Given $\alpha, \beta\in \BBPF_{n-1}(X)$, we also define
\begin{equation}\label{defi:iibar}
\bar \ii(\alpha, \beta)\=\ii(\alpha, \beta)+|\log (\alpha^n)-\log (\beta^n)|
=\log \frac{(\alpha, \beta^{n-1})(\beta, \alpha^{n-1})}{\min\{(\alpha^n),(\beta^n)\}^2}.
\end{equation}

By Lemma~\ref{lem:ijequiv}, we have $\bar\ii(\a,\b)=0$ iff $\a=\b$, so it follows from Lemma~\ref{lem:iquasitriangle}, that $\bar\ii$ is a quasimetric. We now show this quasimetric defines the strong topology on $\BBPF_{n-1}(X)$.
\begin{prop}\label{prop:metrictop}
    We have 
    \[
    \bar \ii(\alpha, \beta)\le n\de_\infty(\alpha, \beta),
    \]
    and there is a continuous, strictly increasing function $F\colon\R_+\to \R_+$ with $F(0)=0$ and $F\in C^{0,1/n}_{\mathrm{loc}}([0,\infty))$,
    such that 
    \[
    \de_\infty(\alpha, \beta)\le F(\bar \ii(\alpha, \beta)).
    \]
    In particular,
    $\bar\ii(\alpha_i, \alpha)\to 0$ iff $\de_\infty(\alpha_i, \alpha)\to 0$.
\end{prop}

\begin{proof}
    By symmetry, we may assume $(\a^n)\le(\b^n)$. If we set $\d:=\de_\infty(\a,\b)$, then $(\a,\b^{n-1})\le e^{(n-1)\d}(\a^n)$ and $(\a^{n-1},\b)\le e^\d(\a^n)$. This yields $\bar\ii(\a,\b)\le n\d$.
    For the second inequality, note that $(\beta, \alpha^{n-1})=\exp(\jj_\alpha(\beta))(\a^n)^{\frac{n-1}{n}}(\b^n)^{\frac1n}$. The Diskant inequality as in~\eqref{equ:Diskant2} gives
\[        s(\alpha, \beta)
        \ge\exp(\jj_\alpha(\beta))^{\frac{1}{n-1}}(\a^n)^{\frac{1}{n}}(\b^n)^{-\frac1n}-(\exp(\jj_\alpha(\beta))^{\frac{n}{n-1}}-1)^{\frac 1n}(\a^n)^{\frac{1}{n}}(\b^n)^{-\frac1n}.
  \] 
    If we define $G:\R_+\to \R_+$ by $G(x)=e^{\frac{x}{n-1}}-(e^{\frac{n}{n-1}x}-1)^{\frac 1n}$, then $G$ is strictly decreasing, $G(0)=1$ and we see that
    \[
    -\log s(\alpha,\beta)\le \frac1n\log\frac{(\b^n)}{(\a^n)}-\log G(\jj_{\alpha}(\beta))
    \le\frac1n\bar\ii(\a,\b)-\log G(\bar\ii(\a,\b)),
    \]
    since $J_\a(\b)\le\ii(\a,\b)\le\bar\ii(\a,\b)$ and $\log\frac{(\b^n)}{(\a^n)}\le\bar\ii(\a,\b)$.
    By symmetry, the same bound holds for $s(\b,\a)$, so in view of~\eqref{equ:dinfty1} we get $\de_\infty(\alpha,\beta)\le F(\bar\ii(\a,\b))$, with 
    $F(x)=\frac{x}n-\log G(x)$. It is straightforward to check that $F\colon\R_+\to\R_+$ is continuous and strictly increasing, and H\"older continuous in the sense that there exists $C>0$ such that $|F(x)-F(y)|\le C|x-y|^{1/n}$ for all $x,y$ with $|x-y|\le1$.
\end{proof}
\begin{cor}\label{cor:strongdense}
    If $(\alpha_i)_i$ is a decreasing net in $\BBPF_{n-1}(X)$ converging weakly to $\alpha\in \BBPF_{n-1}(X)$, then $\alpha_i\to \alpha$ strongly. As a consequence, nef Cartier b-divisor classes are strongly dense in $\BBPF_{n-1}(X)$.
\end{cor}
\begin{proof}
  From the continuity of intersection numbers under decreasing nets, it follows that
$\bar \ii(\alpha_i,\alpha)\to0$. The last statement follows by regularization.
\end{proof}

\subsection{Strong topology for curve classes}\label{sec:strongcurve}
Recall that $\BPF_1(X)\subset \Curv(X)$ is the space of BPF Weil b-curve classes. Following~\cite[\S5.2]{DangFavre} (itself following~\cite{Xiao17,LX}), we define:
\begin{defi}\label{def:volhat}
    The \emph{volume} of a Weil b-curve class $\mu\in \BPF_1(X)$ is
    \begin{equation}\label{equ:volhat}
    \vol^\vee(\mu)\=\inf_{\alpha}\frac{(\alpha, \mu)}{\vol(\alpha)^{1/n}}\in[0,+\infty),
    \end{equation}
    where the infimum is taken over all big and nef divisor classes $\alpha\in\CNum(X)$.

    We say that $\mu\in\BPF_1(X)$ is \emph{big} if $\vol^\vee(\mu)>0$, and write
    \[
    \BBPF_1(X)\subset\BPF_1(X)
    \]
    for the subset of big curve classes.
\end{defi}
\begin{rmk}
    Our definition of volume differs from the one in~\cite{DangFavre} by a power of $\frac{n}{n-1}$.
\end{rmk}
\begin{lem}\label{lem:bpfdual2}
    The volume function on $\BPF_1(X)$ is concave, usc, and homogeneous. As a consequence, $\BBPF_1(X)$ is a convex subcone of $\BPF_1(X)$.
\end{lem}
\begin{proof}
    This follows from $\mu\mapsto(\alpha, \mu)$ being linear and continuous for all $\alpha\in \CNum(X)$.
\end{proof}
If $\mu\in\BBPF_1(X)$, then by definition we have $(\a,\mu)>0$ for all Cartier classes $\a\in\BBPF_{n-1}(X)$. Moreover, $\a\mapsto\log(\a,\mu)$ is 1-Lipschitz in the $\de_\infty$-metric, so since Cartier classes are strongly dense in $\BBPF_{n-1}(X)$, see Corollary~\ref{cor:strongdense}, $\a\mapsto(\a,\mu)$ extends uniquely to a continuous positive function on $\BBPF_{n-1}(X)$, and $\a\mapsto\log(\a,\mu)$ remains 1-Lipschitz.
\begin{lem}\label{lem:bpfdual}
    The volume function on $\BBPF_1(X)$ satisfies
    \begin{equation}\label{equ:volhat2}
    \vol^\vee(\mu)=\inf_{\alpha\in \BBPF_{n-1}(X)}\frac{(\alpha, \mu)}{\vol(\alpha)^{1/n}}.
    \end{equation}
    Moreover, if $\a\in\BBPF_{n-1}(X)$, then $\mu\=\alpha^{n-1}\in\BBPF_1(X)$, and $\vol^\vee(\mu)=\vol(\alpha)^{\frac{n-1}{n}}$ with $\alpha$ achieving the infimum in~\eqref{equ:volhat2}.
\end{lem}
Here $\a^{n-1}$ is defined as in~\cite{DangFavre} (see Example~\ref{exam:Mink}).

\begin{proof}
Fix $\mu\in\BBPF_1(X)$. It is clear that $\vol^\vee(\mu)\ge \inf_{\alpha\in \BBPF_{n-1}(X)}\frac{(\alpha, \mu)}{\vol(\alpha)^{1/n}}$. 
The reverse inequality follows from $\alpha\mapsto\log(\a,\mu)$ and $\a\mapsto\log\vol(\a)^{1/n}$ being strongly continuous, and Cartier classes being strongly dense in $\BBPF_{n-1}(X)$.

    Now suppose $\a\in\BBPF_{n-1}(X)$. The Khovanskii--Teissier inequalities imply 
    \[
    (\beta, \alpha^{n-1})\ge \vol(\beta)^{1/n}\vol(\alpha)^{\frac{n-1}{n}},
    \]
    for any $\beta\in \BBPF_{n-1}(X)$, and so 
    \[
    \vol(\alpha)^{\frac{n-1}{n}}\le \vol^\vee(\mu)\le \frac{(\alpha, \mu)}{\vol(\alpha)^{1/n}}=\vol(\alpha)^{\frac{n-1}{n}},
    \]
    where $\mu=\a^{n-1}$, proving the last assertion.
\end{proof}

We shall use the following version of Siu's inequality for curve classes.

\begin{lem}\label{lem:curvbig}
    For any $\mu\in \BBPF_1(X)$ and any $\a\in\BBPF_{n-1}(X)$, we have
    \begin{equation}\label{equ:Siucurve}
    \frac1e\frac{\vol^\vee(\mu)^n}{(\a,\mu)^{n-1}}\a^{n-1}
    \le\mu
    \le n\frac{(\a,\mu)}{\vol(\a)}\a^{n-1}
    \end{equation} 
    as curve classes.
\end{lem}
The first inequality together with monotonicity of $\vol^\vee$ implies that if $\mu\in\BPF_1(X)$, then $\vol^\vee(\mu)>0$ iff $\mu$ is big in the sense of~\cite{DangFavre}, that is for some/any big and nef Cartier class $\a$, we have $\mu\ge\e\a^{n-1}$ for some $\e>0$.
\begin{proof}
    A version of the first inequality was established in the proof of Theorem~5.11 in~\cite{DangFavre}. Here is an alternative, elementary argument. By homogeneity we may assume $\vol^\vee(\mu)=1$ and $(\a,\mu)=1$.
    Pick a nef Cartier class $\b$. For $t>0$, we have $\b+t\a\in\BBPF_{n-1}(X)$, and $\vol(\b+t\a)\ge nt^{n-1}(\b,\a^{n-1})$. 
    \[
    (\b,\mu)
    =(\b+t\a,\mu)-t
    \ge\vol(\b+t\a)^{\frac1{n}}-t
    \ge(n(\b,\a^{n-1}))^{\frac1{n}}t^{\frac{n-1}{n}}-t
\]
    where the equality follows from Cartier approximation of $\a$, the first inequality from the definition of $\vol^\vee(\mu)$, and the second inequality from the estimate above.
    By taking the supremum over $t>0$, we obtain
    \[
    (\b,\mu)
    \ge\left(\frac{n-1}{n}\right)^{n-1}(\b,\a^{n-1})
    \ge\frac1e(\b,\a^{n-1}).
    \]
    As this holds for all nef Cartier classes $\beta$, the first inequality in~\eqref{equ:Siucurve} follows.

    To prove the second inequality, it similarly suffices to prove $(\b,\mu)\le n\frac{(\a,\mu)}{\vol(\a)}(\b,\a^{n-1})$, and this follows from Siu's inequality for divisor classes, see Lemma~\ref{lem:Siuclass}.
\end{proof}

We now introduce the strong topology on $\BBPF_1(X)$.
We start with defining the pseudoeffective slope, following the procedure in~\S\ref{sec:strongPnum}.
For $\mu, \nu\in \BBPF_1(X)$, we set 
\[
s(\mu, \nu)\=\sup\{s>0\mid \mu\ge s\nu\}.
\]
It follows from Lemma~\ref{lem:curvbig} that $0<s(\mu,\nu)<\infty$. 
Then we set 
\[
\de_\infty(\mu, \nu)\=\inf\{\delta>0\mid e^{-\delta}\mu\le\nu\le e^\delta\mu\}
=\log\max\left\{\tfrac1{s(\mu,\nu)},\tfrac1{s(\nu,\mu)}\right\}.
\]
From the definition of $\vol^\vee$ we obtain the Lipschitz estimate
\begin{equation}\label{equ:volhatLip}
|\log\vol^\vee(\mu)-\log\vol^\vee(\nu)|\le\de_\infty(\mu,\nu).
\end{equation}
\begin{defi}
    The \emph{strong topology} on $\BBPF_1(X)$ is the metric topology induced by $\de_\infty$.
\end{defi}
It is clear that if a net $(\mu_i)$ in $\BBPF_1(X)$ converges strongly to $\mu\in \BBPF_1(X)$, then $\mu_i\to \mu$ weakly and $\vol^\vee(\mu_i)\to \vol^\vee(\mu)$. 
\begin{prop}\label{prop:curvcomplete}
    The metric space $(\BBPF_1(X), \de_\infty)$ is complete. 
\end{prop}
\begin{proof}
    Let $(\mu_i)$ be a Cauchy net in $\BBPF_1(X)$. By~\eqref{equ:volhatLip}, $(\log\vol^\vee(\mu_i))$ is a Cauchy net,  and so $\lim_j\vol^\vee(\mu_j)>0$. Moreover, we have $\lim_{i,j}s(\mu_i, \mu_j)=1$.
    Let $\alpha\in \CNum(X)$ be any big and nef class. Then
    \[
    0<\liminf_j\vol(\alpha)^{1/n}\vol^\vee(\mu_j)\le \liminf_j(\alpha, \mu_j),
    \]
    and further, for any $i, j$, we have
    \begin{equation}\label{equ:unifconv}
        s(\mu_i, \mu_j)\le \frac{(\alpha, \mu_i)}{(\alpha, \mu_j)}\le s(\mu_j, \mu_i)^{-1}.
    \end{equation}
    Thus $\lim_j(\alpha, \mu_j)$ exists and is positive for any $\alpha$ big and nef. Extending by linearity, this yields a weak limit $\mu\in \BPF_1(X)$, and by upper semicontinuity, we have $\vol^\vee(\mu)\ge\varlimsup_j\vol^\vee(\mu_j)>0$, so $\mu\in\BBPF_1(X)$.
    
    We claim that $\de_\infty(\mu_i,\mu)\to0$. This amounts to $\lim_js(\mu_j, \mu)=\lim_js(\mu, \mu_j)=1$, or equivalently, that 
    \[
    \lim_j\frac{(\alpha, \mu_j)}{(\alpha, \mu)}=\lim_j\frac{(\alpha, \mu)}{(\alpha, \mu_j)}=1
    \]
    uniformly for all big and nef classes $\alpha\in \CNum(X)$. To see this, fix $c>1$. Then by~\eqref{equ:unifconv}, for any $i, j\gg 1$, we have $c^{-1}\le \frac{(\alpha, \mu_i)}{(\alpha, \mu_j)}\le c$ for all $\alpha$ big and nef. Since $\mu_j\to \mu$ weakly, we have
    \[
    c^{-1}\le \frac{(\alpha, \mu_i)}{(\alpha, \mu)}\le c, \quad  c^{-1}\le \frac{(\alpha, \mu)}{(\alpha, \mu_i)}\le c
    \]
    for all $i\gg 1$. Letting $c\to 1$ we conclude.
\end{proof}

\subsection{The Minkowski problem}\label{sec:Mink}
We now prove Theorem~A in the introduction. More precisely we prove
\begin{thm}\label{thm:metricequivnum}
    The map $(\BBPF_{n-1}(X), \de_\infty)\to (\BBPF_1(X), \de_\infty)$ given by $\a\mapsto\a^{n-1}$ is a Lipschitz continuous homeomorphism, with locally uniformly H\"older continuous inverse. More precisely, for any $\alpha, \beta\in \BBPF_{n-1}(X)$, we have
    \begin{equation}\label{equ:Mincont1}
    \de_\infty(\alpha^{n-1},\beta^{n-1})\le(n-1)\de_\infty(\alpha, \beta),
    \end{equation}
    and 
    \begin{equation}\label{equ:Mincont2}
    \de_\infty(\alpha, \beta)\le F\left(\tfrac{3n-2}{n-1}\de_\infty(\alpha^{n-1},\beta^{n-1})\right),
    \end{equation}
    for the function $F$ in Proposition~\ref{prop:metrictop}.
\end{thm}
\begin{proof}
    That the map is a bijection is Theorem~C in~\cite{DangFavre}. It only remains to prove~\eqref{equ:Mincont1} and~\eqref{equ:Mincont2}.
    
    If $e^{-\d}\alpha\le\beta\le e^\d\alpha$, then $e^{-(n-1)\d}\alpha^{n-1}\le\beta^{n-1}\le e^{(n-1)\d}\alpha^{n-1}$; this immediately gives~\eqref{equ:Mincont1}.
    Now suppose $e^{-\d}\alpha^{n-1}\le \beta^{n-1}\le e^{\d}\alpha^{n-1}$. By Lemma~\ref{lem:bpfdual}, this first yields $e^{-\frac{n}{n-1}\d}\vol(\a)\le \vol(\b)\le e^{\frac{n}{n-1}\d}\vol(\a)$. Further, if $(\a_i)_i$ is a decreasing net in $\BBPF_{n-1}(X)$ converging to $\a$, then 
    $(\a_i,\b^{n-1})\le e^\d(\a_i,\a^{n-1})$, for all $i$, which by continuity of intersection numbers along decreasing nets yields $(\a,\b^{n-1})\le e^\d \vol(\a)$. By symmetry, we have $(\b,\a^{n-1})\le e^\d\vol(\b)$, which leads to $\bar\ii(\a,\b)\le (2+\frac{n}{n-1})\d$, and we get~\eqref{equ:Mincont2} using Proposition~\ref{prop:metrictop}. 
\end{proof}

\section{The Berkovich analytification and Fubini--Study functions}\label{sec:FSfcn}

Next we recall the Berkovich analytification $X^\na$ of $X$, and introduce certain natural classes of functions on it.

\subsection{Valuations and Berkovich analytification}\label{sec:vals}
A \emph{semivaluation} on $X$ (or on $R$) is a function $v\colon R\to\R\cup\{+\infty\}$ satisfying $v(0)=+\infty$, $v|_{k^\times}\equiv0$ and $v(fg)=v(f)+v(g)$, $v(f+g)\ge\min\{v(f),v(g)\}$ for $f,g\in R$. It is a \emph{valuation} if $v(f)<+\infty$ when $f\ne0$. Any valuation on $R$ extends uniquely to a valuation on the fraction field $k(X)$ of $X$. On the other hand, a semivaluation $v$ induces a valuation on $R/\{v=+\infty\}$. 

The \emph{Berkovich analytification} $X^\na$ of $X$ is defined as the space of semivaluations on $X$. If $X^\val\subset X^\na$ is the set of valuations on $X$, we see that $X^\na=\bigcup_ZZ^\val$, where $Z$ ranges over subvarieties of $X$. A semivaluation $v\in Z^\val\subset X^\na$ is said to have \emph{support} equal to $Z$.

We equip $X^\na$ with the topology of pointwise convergence on $R$; then $X^\na$ is locally compact. More precisely, if $f_1,\dots,f_N\in R$ are elements that generate $R$ as a $k$-algebra, then the set $\{v\in X^\na\mid\min_iv(f_i)\ge-t\}$ is compact for all $t>0$.

A semivaluation $v\in X^\na$ is \emph{centered on $X$} if $v\ge0$ on $R$: in this case the \emph{center} $c_X(v)$ is the point on $X$ defined by the prime ideal $\{v>0\}\subset R$. If $v$ is not centered on $X$, we say that $v$ is \emph{centered at infinity}. We write $X^\beth$ (resp.\ $X^\na_\infty$) for the closed (resp.\ open) set of semivaluations centered on $X$ (resp.\ at infinity). 

If $\bar X$ is a compactification of $X$, then it follows from the valuative condition of properness that any valuation $v\in X^\val$ admits a center $c_{\bar X}(v)\in\bar{X}$, the unique point such that $v$ is nonnegative on the local ring $\cO_{\bar X,c_{\bar X}(v)}$ and strictly positive on its maximal ideal. The center is defined also for $v\in X^\na$. We have $c_{\bar X}(v)\in X$ iff $v\in X^\beth$, and in this case $c_{\bar X}(v)=c_X(v)$ as defined above.

We have a natural multiplicative action of $\R_{>0}$ on $X^\na$. This preserves the support and center; in particular $X^\beth$ and $X^\na_\infty$ are invariant under this action.

The \emph{trivial valuation} $\triv_X$ is defined by $\triv_X\equiv0$ on $R\setminus\{0\}$ (hence also on $k(X)^\times$). More generally, we have a trivial semivaluation $\triv_Z$ for every subvariety $Z\subset X$. These semivaluations can be characterized by taking values in $\{0,+\infty\}$, as fixed points for the multiplicative action of $\R_{>0}$, or as semivaluations in $X^\beth$ for which the support is equal to the center.

A valuation $v\in X^\val$ is \emph{divisorial} if $v(k(X)^\times)\subset\Q$ and $\trdeg(\kappa(v)/k)=n-1$, where $\kappa(v)=\{v\ge0\}/\{v>0\}$ is the residue field of $v$. We write $X^\div$ for the set of divisorial valuations, and $X^\div_\infty$ for the subset of divisorial valuations centered at infinity. These sets are invariant under the scaling action by $\Q_{>0}$.

Valuations in $X^\div_\infty$ can be geometrically described as follows: 
there exists an integrally closed compactification $\bar{X}$ of $X$, an irreducible component $E\subset\partial\bar X$, and $c\in\Q_{>0}$ such that $v=c\ord_E$. This follows from a result by Zariski, see~\cite[Lemma~2.45]{KollarMori}, which says that after blowing up the center finitely many times, the center becomes a prime divisor. The same argument shows that if $v\in X^\beth\cap X^\div$, then $v=c\ord_E$, where now $E$ is a prime divisor on a normal projective variety $Y$ birational to $X$; we may assume $Y$ is a smooth projective model of $X$, that is, $Y$ is a smooth projective variety admitting a birational morphism $Y\to\bar X$ onto some integrally closed compactification $\bar X$ of $X$.
Moreover, we have 
\begin{prop}\label{prop:compact}
    The space $X^\na_\infty/\R_+^\times$ is compact.
\end{prop}
\begin{proof}
    Pick a compactification $\bar X$ of $X$ together with a strictly  effective Cartier divisor $D$ at infinity. The function $\f_D\colon\bar{X}^\na\to[0,+\infty]$ defined by $\f_D(v):=v(D)$ is then continuous, equivariant with respect to action of $\R_+^\times$, and we have $\{\f_D=0\}=X^\beth$, $\{0<\f_D<\infty\}=X^\na_\infty$. Thus $X^\na_D:=\{\f_D=1\}$ is a compact subset of $X^\na_\infty$. The restriction to $X^\na_D$ of the canonical map $X^\na_\infty\to X^\na_\infty/\R_+^\times$ is now a homeomorphism, whose inverse given by $[v]\mapsto v/\f_D(v)$. The result follows.
\end{proof}

A morphism $h\colon X\to X'$ between affine varieties induces a continuous map $h^\na\colon X^\na\to X^{\prime\na}$, which satisfies $(h^\na)^{-1}(X^{\prime\na}_\infty)\subset X^\na_\infty$. When $h$ is a closed immersion, or more generally a finite morphism, we have $h^\na(X^\na_\infty)\subset X^{\prime\na}_\infty$.

\subsection{Piecewise linear functions}\label{sec:PL}
Given a compactification $\bar{X}$ of $X$ and an effective Cartier divisor $D$ at infinity on $\bar{X}$, we define a continuous function $\psi_{\bar{X},D}\colon\bar{X}^\na\to\R_+\cup\{+\infty\}$ by 
\[
\p_{\bar{X},D}(v)\coloneqq v(D).
\]
More precisely, if $U\subset\bar{X}$ is an open affine subset containing the center of $v$ on $\bar{X}$, and such that $D$ is defined by $h\in\cO_{\bar{X}}(U)$, then $v\in U^\na$ and $\p_{\bar{X},D}(v)\coloneqq v(h)$. For example, if $\bar{X}$ is integrally closed and $D=\sum_ib_iE_i$, then $\p_{\bar{X},D}(\ord_{E_i})=b_i$.

The function $\p_D$ is \emph{homogeneous} in the sense that $\f(tv)=t\f(v)$ for any $v\in X^\na$ and any $t>0$. It is finite-valued on $X^\na$ and vanishes on $X^\na\setminus X^\na_\infty$. Further, $D$ is strictly effective iff $\p_D>0$ on $X^\na_\infty$. 

If $\mu\colon\bar{X}'\to\bar{X}$ is a morphism of compactifications, then $\p_{\bar{X},D}=\p_{\bar{X}',\mu^*D}$ on $X^\na$. We therefore obtain a well-defined map $D\mapsto\p_D\in C^0(X^\na)$ from the space of effective Cartier divisors at infinity on $X$. This map is $\Z_{>0}$-linear, and therefore extends uniquely to a $\Q$-linear map on the space $\CbDiv_\infty(X)_\Q\to C^0(X^\na)$.
This map is injective. Indeed, if $D\in\CbDiv_\infty(X)_\Q$, then we can realize $D$ as a $\Q$-Cartier $\Q$-divisor $D=\sum_ir_iE_i$ at infinity on an integrally closed compactification $\bar{X}$ of $X$, and we have $\p_D(\ord_{E_i})=r_i$, so if $\p_D=0$, then $D=0$
\begin{defi}
    The space $\PL(X)\subset C^0(X^\na)$ of \emph{piecewise linear functions} is the image of the map $\CbDiv_\infty(X)_\Q\to C^0(X^\na)$. We write $\PL^+(X)\subset\PL(X)$ for the set of functions $\p_D$ with $D\in\CbDiv_\infty(X)_\Q$ semiample and strictly effective.
\end{defi}
Any function in $\PL(X)$ is homogeneous and vanishes on $X^\na\setminus X^\na_\infty$. By Lemma~\ref{lem:sabig}, 
we have an identification 
\[
\PL^+(X)=\cP\cC^{\mathrm{int}}
\]
by sending the divisor to the unique integrally closed polarized compactification representing it. 
It follows that $\p$ is strictly positive on $X^\na_\infty$, and extends as a continuous homogeneous function $\p\colon\bar{X}^{\prime\an}\to\R_+\cup\{\infty\}$ for any  compactification of $X$ dominating $\bar X$.

\begin{lem}\label{lem:PLloggr}
Any function in $\PL(X)$ is a difference of functions in $\PL^+(X)$. Moreover, given $\p_1,\p_2\in\PL^+(X)$ there exists $C\ge1$ such that $C^{-1}\p_1\le\p_2\le C\p_1$ on $X^\na$.
\end{lem}
\begin{proof}
    The first statement is a consequence of Corollary~\ref{cor:diffsa}.
    The second statement follows from the fact that if $D_1$, $D_2$ are strictly effective divisors at infinity on some compactification, then there exists $C\ge1$ such that $CD_1-D_2$ and $CD_2-D_1$ are effective.
\end{proof}
\begin{lem}\label{lem:PLfin}
    Let $\tau\colon X\to X'$ be a finite morphism of affine varieties. Then $(\tau^\na)^*\PL(X')\subset\PL(X)$ and $(\tau^\na)^*\PL^+(X')\subset\PL^+(X)$.
\end{lem}
\begin{proof}
    Suppose $\p\in\PL(X')$, so that $\p=\p_{D'}$ for some $\Q$-Cartier $\Q$-divisor $D'$ at infinity on some compactification $\bar{X}'$ of $X'$. Pick a compactification $\bar{X}$ of $X$ such that $\tau$ extends to a morphism $\tau\colon\bar{X}\to\bar{X}'$. Then $\p\circ\tau^\na=\p_{\tau^*D'}$, so $\p\circ\tau^\na\in\PL(X)$. Moreover, if $\p\in\PL^+(X')$, then $D'$ is semiample and strictly effective. It follows that $\tau^*D'$ has the same properties since $\tau$ is finite, so $\p\circ\tau^\na\in\PL^+(X)$.
\end{proof}

\subsection{Functions of logarithmic growth}\label{sec:loggr}
Given $\p\in\PL^+(X)$ we set 
\[
X^\na_\p:=\{v\in X^\na_\infty\mid \p(v)=1\}.
\]
This is a compact subset of $X^\na_\infty$, and the canonical map $X^\na_\infty\to X^\na_\infty/\R_+^\times$ restricts to a homeomorphism $X^\na_\p\simto X^\na_\infty/\R_+^\times$; see the proof of Proposition~\ref{prop:compact}. Any function $\psi_D$ defined by a strictly effective divisor $D\in \CbDiv_\infty(X)_\Q$ is in particular positive on $X^\na_\psi$ for any $\psi\in\PL^+(X)$ and hence there exists $C\ge 1$ with
$C^{-1}\psi\le \psi_D\le C\psi$. More generally, we introduce
\begin{defi}
    A function $\f\colon X^\na\to\R$ is of \emph{logarithmic growth} if it is homogeneous, and for some (or, equivalently, any) $\p\in\PL^+(X)$, 
    there exists $C\ge1$ such that 
    \[
    C^{-1}\p\le\f\le C\p.
    \]
\end{defi}
Denote by $\cL^\infty=\cL^\infty(X)$ the space of functions of logarithmic growth.
A function of logarithmic growth is thus nonnegative, and strictly positive exactly on $X^\na_\infty$. 
As an immediate consequence of Lemma~\ref{lem:PLfin}, we obtain
\begin{lem}\label{lem:loggrimm}
    Let $\tau\colon X\hto X'$ be a finite morphism of affine varieties. If $\f'\colon X^{\prime\na}\to\R_+$ is of logarithmic growth, so is $\f\coloneqq\f'\circ\tau^\na\colon X^\na\to\R_+$.
\end{lem}
The space $\cL^\infty$ comes with a $\de_\infty$ metric, that is,
given two functions $\varphi, \psi\in \cL^\infty$, we set
\[
\de_\infty(\varphi, \psi)\coloneqq \log\inf\{C\geq 1\mid C^{-1}\psi\leq \varphi\leq C\psi\};
\]
equivalently, $\de_\infty(\f,\p)=\sup_{X^\na_\infty}\|\log\f-\log\p\|$.

\subsection{Fubini--Study functions}
We now introduce a class of functions defined by regular functions on $X$. They give explicit descriptions of functions in $\PL^+(X)$.
\begin{defi}
    For any nonzero subgroup $\Lambda\subset\R$ we define $\cH_\Lambda=\cH_\Lambda(X)$ as the set of functions 
    $\f\colon X^\na\to\R_+$ of the form 
    \[
    \f=\max_{1\le i\le N}\la_i^{-1}\log^+|f_i|,
    \]
    where $0<\la_i\in\Lambda$ and where $f_i\in R$ generate $R$ as a $k$-algebra.
    We call the elements of $\cH\coloneqq\cH_\Q$ (resp.\ $\cH_\R$) rational (resp.\ real) \emph{Fubini--Study functions}.
\end{defi}

Any real Fubini-Study function is nonnegative, homogeneous, and strictly positive exactly on $X^\na_\infty$. 
Note that $\cH_\Lambda$ is stable under finite maxima. The set $\cH$ of rational Fubini--Study functions is also stable under addition (as follows from Lemma~\ref{lem:FSPL} below), but this is not true for $\cH_\Lambda$ for general $\Lambda$.

The following characterization will be convenient.
\begin{lem}\label{lem:charFS}
    Given a nonzero subgroup $\Lambda\subset\R$, let $\f\colon X^\na\to\R_+$ be a function of the form
    \[
    \f=\max_{1\le i\le M}\la_i^{-1}\log^+|f_i|,
    \]
    where $0<\la_i\in\La$, $f_i\in R$, and $R$ is an integral $k[f_1,\dots,f_M]$-module. Then $\f\in\cH_\La(X)$.
\end{lem}
\begin{proof}
    Pick any $\la_0\in\Lambda\cap\R_{>0}$, and pick elements $g_1,\dots,g_N\in R$ that generate $R$ as a $k$-algebra. For $m\in\Z_{>0}$, the function $\p:=\max\{\f,(m\la_0)^{-1}\max_j\log^+|g_j|\}$ lies in $\cH_\La(X)$, so it suffices to prove that if $m\gg1$, then $(m\la_0)^{-1}\log^+|g_j|\le\f$ for all $j$. 

    As $g_j$ is integral over $k[f_1,\dots,f_M]$, it satisfies a monic equation
    \[
    g_j^{n_j}=\sum_{l=0}^{n_j-1}p_{j,l}(f_1,\dots,f_M)g_j^l
    \]
    for some polynomials $p_{j,l}\in k[z_1,\dots,z_M]$, where $n_j\ge1$.
    For large $C$, we have
    \[
    \max_l\log^+|p_{j,l}(f_1,\dots,f_M)|\le C\f
    \]
    for all $j,l$. This yields
    \[
    n_j\log^+|g_j|
    \le\max_l\log^+|p_{j,l}(f_1,\dots,f_M)g_j^l|
    \le\max_l l\log^+|g_j|+C\f,
    \]
    so that $\log^+|g_j|\le C\f$. Thus it suffices to take $m>C/\la_0$.
\end{proof}
\begin{cor}
    If $\Lambda\subset\R$ is a nonzero subgroup, then $\cH_\Lambda(X)=\cH_{\Q\Lambda}(X)$.
\end{cor}
\begin{proof}
    We trivially have $\cH_\Lambda(X)\subset\cH_{\Q\Lambda}(X)$. Conversely, any $\f\in\cH_{\Q\Lambda}(X)$ can be written 
    \[
    \f=m\max_{1\le i\le M}\la_i^{-1}\log^+|f_i|=\max_{1\le i\le M}\la_i^{-1}\log^+|f_i^m|,
    \]
    where $m\in\Z_{>0}$, $0<\la_i\in\La$ for all $i$, and where $f_1,\dots,f_M$ generate $R$ as a $k$-algebra. As $R$ is finite (hence integral) over $k[f_1^m,\dots,f_M^m]$, it follows from Lemma~\ref{lem:charFS} that $\f\in\cH_\Lambda(X)$.
\end{proof}
\begin{cor}\label{cor:FSfin}
    If $\tau\colon X\to X'$ is a finite morphism, then $(\tau^\na)^*\cH_\Lambda(X')\subset\cH_\Lambda(X)$ for every nonzero subgroup $\Lambda\subset\R$.
\end{cor}
\begin{lem}\label{lem:FSPL}
    The space $\cH(X)$ of rational Fubini--Study functions equals $\PL^+(X)$. 
\end{lem}
\begin{proof}
    First consider $\p\in\PL^+(X)$, and write $\p=\p_D$ for an ample divisor $D\in\CDiv_\infty(\bar X)_\Q$ at infinity on a compactification $\bar{X}$ of $X$. For $m$ sufficiently divisible, pick a basis $s_0,\dots,s_N$ of $\Hnot(\bar{X},mD)$, with $s_0=s_{mD}$, and let $\tau\colon\bar{X}\to\P^N$ be the embedding defined by $\tau=[s_0:\dots:s_N]$. Then $\tau$ restricts to an embedding $X\to\A^N$ given by $\tau=(f_1,\dots,f_N)$, where $f_i=s_i/s_0$. Moreover, we have $\tau^*H=mD$, where $H=\P^N\setminus\A^N$. Hence $\p_D=m^{-1}(\tau^\na)^*\p_H$. Now $\p_H=\max_i\log^+|z_i|\in\cH(\A^N)$, so by Corollary~\ref{cor:FSfin} we have $\p_D\in\cH(X)$.

    Conversely, suppose $\p\in\cH(X)$, and write $\p=\max_{1\le i\le N}\la_i^{-1}\log^+|f_i|$, where $\la_i\in\Q_{>0}$ and where the $f_i$ generate $R$ as a $k$-algebra. After replacing $\p$ by an integer multiple, we may assume that $\la_i=1/d_i$, with $d_i\in\Z_{>0}$. If $\tau\colon X\to\A^N$ is the finite map given by $\tau=(f_1^{d_1},\dots,f_N^{d_N})$, then $\p=(\tau^\na)^*\p'$, where $\p'=\max_i\log^+|z_i|$. Now $\p'=\p_H$, where $H=\P^N\setminus\A^N$, so $\p'\in\PL^+(\A^N)$, and we conclude using Lemma~\ref{lem:PLfin} that $\p\in\PL^+(X)$.
\end{proof}
\begin{cor}
    Any function in $\PL(X)$ is a difference of rational Fubini--Study functions.
\end{cor}
\begin{cor}
    Any real Fubini--Study function is of logarithmic growth.
\end{cor}

\begin{cor}
    The space $\PL^+(X)$ is closed under finite maxima. As a consequence, the $\Q$-vector space $\PL(X)$ is closed under finite maxima and minima.
\end{cor}

\begin{prop}\label{prop:SW}
    For any $\p\in\PL^+(X)$, the set $\{\f|_{X^\na_\p}\mid \f\in\PL(X)\}$, is dense in $C^0(X^\na_\p)$.
\end{prop}
\begin{proof}
It is clear that the set contains rational constants. Let $\psi=\psi_{\bar X, D}$ for some polarized compactification $(\bar X, D)$. Assume $v, v'\in X^\na_\psi$ are two different semivaluations centered at infinity. Then there exists $f\in R$ with $v(f)\ne v'(f)$. Wlog we may assume $v(f)<v'(f)$ and that $v(f)<0$. Indeed, if $g\in R$ is such that $v(g)<0$ and $v'(g)\ne v(g)$, then we may replace $f$ by $g$; if $v(g)=v'(g)<0$, then we may replace $f$ by $fg^m$ for $m\gg0$. Now the PL function $\max\{\psi, \log|f^N|\}$ separates $v$ and $v'$ for $N\gg 1$.
The assertion now follows from the lattice version of the Stone--Weierstrass theorem.
\end{proof}

Using the same method as in the proof of Lemma~\ref{lem:loggrimm}, we obtain
\begin{prop}\label{prop:FSpsh}
    Any real Fubini--Study function $\f\in\cH_\R(X)$ is a cpsh (continuous plurisubharmonic) function on $X^\na$ in the sense of Chambert-Loir and Ducros~\cite{CLD25}.
\end{prop}
\begin{proof}
    By construction there is a closed immersion $\iota\colon X\to\A^N=\Spec k[z_1,\dots,z_N]$ and a function $\p\colon\A^{N,\na}\to\R_+$ of the form $\p=\max_i\la_i^{-1}\log^+|z_i|$, 
    $\la_i\in\R_+^\times$, such that $\f=\p\circ\iota^\na$. Arguing as in the proof of~\cite[Proposition~17.8.6]{CLD25}, one shows that $\p$ and $\f$ are cpsh.
\end{proof}
We also note the following connection with homogeneous psh functions in~\cite{BJpp}. 
\begin{prop}\label{prop:fshom}
    For any $\psi=\psi_{\bar X,D}\in \cH(X)$ with $\bar X$ integrally closed, there is a natural inclusion
    \[
    \iota\colon \cH(X, \psi)\=\{\f\in\cH(X)\mid \f\le\p\}\hookrightarrow \cH_\hom(\bar X, D), \f\mapsto \f-\p.
    \]
\end{prop}
\begin{proof}
    Pick $\f\in\cH(X,\p)$. We can write $\f=\max_i\la_i^{-1}\log^+|f_i|$, with $\la_i\in \Z$ and $f_i\in R\setminus\{0\}$. After passing to a common multiple, we may and will assume $\la_i=m\in \Z_{>0}$ for all $i$, and that $mD$ is Cartier. Then 
    \[
    \f-\p=m^{-1}\max_i \max\{\log|s_{mD}|, \log|f_is_{mD}|\}.
    \]
    Since $\f-\p\le 0$, we have $\log|f_is_{mD}|\le 0$ on $X^\val=\bar X^\val$, and hence $f_is_{mD}\in \Hnot(\bar X, mD)$; see Lemma~\ref{lem:critnopole}. Thus $\f-\p\in \cH_\hom(\bar X, D)$.
\end{proof}
\section{Plurisubharmonic functions}\label{sec:psh}
In this section, we introduce and study the space of general psh functions on $X^\na$, obtained as suitable limits of Fubini--Study functions. Any psh function can be identified with a homogeneous psh function in the sense of~\cite{BJpp} on some polarized compactification of $X$, and also induces a nef Weil b-divisor at infinity.
\subsection{General plurisubharmonic functions}
We start with the definition and basic properties of general psh functions.
\begin{defi}
    A function on $X^\na$ is \emph{plurisubharmonic} (psh) if it is the pointwise limit of a decreasing net of Fubini--Study functions. We write $\PSH(X)$ for the space of psh functions, and equip it with the topology of pointwise convergence on $X^\div$.
\end{defi}
Any psh function is usc, nonnegative, homogeneous, and vanishes on $X^\beth=X^\na\setminus X^\na_\infty$. We also denote by $\CPSH(X)\= C^0(X^\na)\cap \PSH(X)$ the space of continuous psh functions.
\begin{rmk}\label{rmk:notCLDpsh}
If $f$ is a nonconstant invertible function on $X$, then the function $\log|f|$ is continuous and psh (even pluriharmonic) in the sense of~\cite{CLD25}, but it is not nonnegative, and hence not psh in our sense. However, the function $\log^+|f|$ is psh, and such functions generate $\PSH(X)$ in a suitable sense, see Corollary~\ref{cor:pshchar}.
\end{rmk}
\begin{thm}\label{thm:pshprop}
The space $\PSH(X)$ is closed under addition, the multiplicative action by $\R_{>0}$, finite maxima, decreasing limits, and locally uniform limits.
\end{thm}
\begin{proof}
    Let $\varphi, \varphi'\in \PSH(X)$, and pick decreasing nets $(\varphi_i), (\varphi_i')$ in $\cH(X)$ converging to $\varphi, \varphi'$ respectively. Also pick a decreasing net $(a_i)$ in $\Q_{>0}$ converging to $a\in\R_{>0}$. Then $\varphi_i+\varphi_i', a_i\varphi_i, \max\{\varphi_i, \varphi_i'\}\in \cH(X)$ decrease to $\varphi+\varphi', a\varphi, \max\{\varphi, \varphi'\}$ respectively. That $\PSH(X)$ is closed under decreasing limits follows from~\cite[Lemma~4.6]{BJpp} with $K=X^\na_\rho$ for some $\rho\in \cH(X)$, $S=T$ is singleton, and $\cF=\{\log\varphi\mid \varphi\in\cH(X)\}$. 
    (Strictly speaking, this lemma requires $\cF$ to be closed under addition of arbitrary real constants, but it is easy to see that being closed under addition of constants in $\log(\Q_{>0})$ is enough.)
    Finally let $(\varphi_i)$ be a sequence in $\PSH(X)$ converging locally uniformly to $\varphi$. Pick $\rho\in \cH(X)$. Then $\varphi_i\to \varphi$ uniformly on $X^\na_\rho$.
    After passing to a subsequence, we may choose $0<t_i\to 0$ such that $\varphi_i+t_i$ decreases to $\varphi$ on $X^\na_\rho$. This shows $\varphi_i+t_i\rho\in \PSH(X)$ decreases to $\varphi$ on $X^\na$, and hence $\varphi\in \PSH(X)$.
\end{proof}
\begin{cor}\label{cor:cpshclos}
    The space $\CPSH(X)$ is the closure of $\cH(X)\subset C^0(X^\na)$ under local uniform convergence.
\end{cor}
\begin{proof}
    Theorem~\ref{thm:pshprop} shows that $\CPSH(X)$ is closed under local uniform limits. On the other hand, any function $\f\in\CPSH(X)$ is a limit of a decreasing net in $\cH(X)$. By Dini's lemma, this is uniform convergence on $X^\na_\rho$ for any $\rho\in\cH(X)$, and the result follows.
\end{proof}
\begin{cor}
    The metric completion of $(\cH(X), \de_\infty)$ is $(\CPSH(X)\cap \cL^\infty, \de_\infty)$.
\end{cor}
\begin{proof}
    Note first that since every $\f\in \CPSH(X)\cap \cL^\infty$ is a local uniform limit of functions in $\cH(X)$, it is also a limit of those functions in $\de_\infty$ metric.
    Suppose $(\f_i)$ is a Cauchy net in $\cH(X)$, and fix $\rho\in \cH(X)$. Then for each $v\in X^\na_\infty$, we have 
    \[
    -\de_\infty(\f_i,\f_j)\le \log\f_i(v)-\log\f_j(v) \le\de_\infty(\f_i,\f_j)
    \]
    for all $i,j$. Thus $(\log\f_i(v))$ is a Cauchy net for each $v\in X^\na_\infty$ and hence there exists $\f\colon X^\na\to \R_+$ such that $\f>0$ on $X^\na_\infty$ and $\f_i\to \f$ pointwise. This in fact shows that $\f_i\to \f$ locally uniformly since for any $c>1$, $c^{-1}\le \frac{\f_i}{\f_j}\le c$ for all $i, j\gg0$. Thus by Corollary~\ref{cor:cpshclos}, $\f\in \CPSH(X)$. Note that $X^\na_\rho$ is compact, we have $\f$ is bounded on $X^\na_\rho$, and so $\f\in \cL^\infty$.
\end{proof}

\begin{exam}\label{exam:PLPSH}
    Any function in $\PL(X)\otimes_\Q\R$ of the form $\f=\f_D$, with $D\in\CbDiv_\infty(X)_\R$ nef and effective, belongs to $\CPSH(X)$. The space of such functions is closed under addition and the scaling action by $\R_{>0}$, but not under finite maxima, in general.
\end{exam}
\begin{exam}
    The space $\cH_\R(X)$ of real Fubini--Study functions is contained in $\CPSH(X)$.
\end{exam}
This also gives an alternative characterization of $\PSH(X)$.
\begin{cor}\label{cor:pshchar}
    The space $\PSH(X)$ is the smallest class of non-negative homogeneous functions $\varphi: X^\na\to \R_+$ that contains all functions of the form $\log^+|f|$ for $f\in R$, and is stable under scaling by a positive real constant, addition, finite maxima and decreasing limits.
\end{cor}
Recall from Corollary~\ref{cor:FSfin} that pullbacks of Fubini--Study functions under finite morphisms remain Fubini--Study. For general psh functions, we have
\begin{prop}\label{prop:pshfunct}
    Let $\tau\colon X'\to X$ be a morphism of affine algebraic varieties. For any $\f\in\PSH(X)$, we then have $\f\circ\tau^\na\in\PSH(X')$.
\end{prop}
\begin{proof}
    By Theorem~\ref{thm:pshprop}, it suffices to show that if $\varphi\in \cH(X)$, then $\varphi\circ \tau^\na\in \PSH(X')$. Write $\varphi=\max_i\lambda_i^{-1}\log^+|f_i|$, where $\{f_i\}$ generate $R\=\cO(X)$. Then $\varphi\circ\tau^\na=\max_i\lambda_i^{-1}\log^+|\tau^*f_i|$, which lies in $\PSH(X')$, again by Theorem~\ref{thm:pshprop}.
\end{proof}

\subsection{The topology of \texorpdfstring{$\PSH(X)$}{PSH(X)}}
As in~\cite{BJpp}, we now show that the topology on $\PSH(X)$ is Hausdorff.
\begin{thm}\label{thm:weaktop}
    Let $\varphi\in \PSH(X)$ and $\psi: X^\na\to \R_+$ be a homogeneous usc function. Then $\varphi\le \psi$ on $X^\div$ iff $\varphi\le \psi$ on $X^\na$.
\end{thm}
\begin{proof}
    We first assume $\psi\in \cH(X)$, and pick an integrally closed compactification $\bar X$ of $X$ such that $\psi=\varphi_D$ for some $D\in \CDiv_\infty(\bar X)$ strictly effective and semiample. If $\varphi\in \cH(X)$, then there is an integrally closed compactification $\pi\colon\bar X'\to \bar X$ such that $\varphi$ is defined by $D'\in \CDiv_\infty(\bar X')$. Up to replacing $\varphi, \psi$ by a positive multiple, we may also assume $D'$ is base point free.
    Set $\fa=\pi_*\cO_{\bar X'}(D')$, and $c\= \max_E \frac{\varphi(\ord_E)}{\psi(\ord_E)}\le1$, where the max is taken over all irreducible components of $\bar X\setminus X$. 
    Then on $\bar X$ we have 
    \[
    -\ord_E(\fa)= \varphi(\ord_E)\le c\psi(\ord_E)=c\ord_E(D),
    \]
    for all $E$. This implies $\cO_{\bar X}(-crD) \fa^r\subset\cO_{\bar X}$ for $r$ sufficiently divisible by Lemma~\ref{lem:critnopole}, and hence $\varphi\le c\psi\le\psi$ on $X^\na$. In general, $\varphi$ is a decreasing limit of $\varphi_j\in \cH(X)$. Assume $\varphi\le \psi$ on $X^\div$. Then in particular $\varphi(\ord_E)\le \psi(\ord_E)$ for any $E\subset \bar X\setminus X$, and $c_j\=\max_E\frac{\varphi_j(\ord_E)}{\psi(\ord_E)}$ satisfies $\lim_jc_j\le 1$. By the previous argument, we have $\varphi_j\le c_j\psi$ on $X^\na$ and letting $j\to \infty$ we get $\varphi\le \psi$ on $X^\na$.

    Next if $\psi\in \PL(X)$, then we may write $\psi=\psi_1-\psi_2$ with $\psi_1, \psi_2\in \cH(X)$. Then $\varphi\le \psi$ on $X^\div$ implies $\varphi+\psi_2\le \psi_1$ on $X^\div$. Since $\varphi+\psi_2\in \PSH(X)$, by the previous case we have $\varphi\le \psi$ on $X^\na$.

    Now suppose $\psi\in C^0(X^\na)$. By Proposition~\ref{prop:SW}, $\psi$ is a locally uniform limit of $\psi_i\in \PL(X)$. We may pick $\rho\in \cH(X)$ and $t_i\to 0+$ such that $\varphi\le \psi_i+t_i\rho$, and the result follows by taking $i\to \infty$.

    Finally, any homogeneous usc function $\psi\colon X^\na\to \R_+$ is a decreasing limit of homogeneous functions $\psi_i\in C^0(X^\na)$, and if $\varphi\le \psi\le \psi_i$ on $X^\div$, then $\varphi\le \psi_i$ on $X^\na$ and we are done by taking $i\to \infty$.
\end{proof}
\begin{cor}
    Every psh function $\varphi\in\PSH(X)$ is uniquely determined by its values on $X^\div$, and the topology on $\PSH(X)$ is Hausdorff. 
\end{cor}
The above proof also yields
\begin{cor}\label{cor:supcts}
    For every $\psi\in \cH(X)$, the map $\varphi\mapsto \sup_{X^\na_\psi}\varphi$ is continuous on $\PSH(X)$.
\end{cor}
From Proposition~\ref{prop:fshom} we can also deduce
\begin{cor}\label{cor:pshhomemb}
    Pick $\p\in\cH(X)$, and write $\p=\p_{\bar X, D}$ with $\bar{X}$ integrally closed and $D\in\CDiv_\infty(\bar X)_\Q$ ample and strictly effective. Then, for any  $t\in\Q_{>1}$, there is a closed embedding
    \[
    \iota\colon \PSH(X, \psi)\=\{\f\in \PSH(X)\mid \f\le \psi\}\hookrightarrow \PSH_\hom(\bar X, tD), \f\mapsto \f-t\p.
    \]
\end{cor}
\begin{proof}
    If $\f\in \cH(X,\psi)$, then $\f\le \p\le t\p$ and $\f-t\p\in \cH_\hom(\bar X, tD)$ by Proposition~\ref{prop:fshom}. In general, pick a decreasing net $(\f_i)$ in $\cH(X)$ converging to $\f$. Then, by Corollary~\ref{cor:supcts}, we have $\f_i\le t\p$ for $i\gg 0$. Thus $\f-t\p$ is a limit of the decreasing net $\f_i-t\p\in \cH_\hom(\bar X, tD)$, and so is a function in $\PSH_\hom(\bar X, tD)$. 
    
    The map $\f\mapsto\f-t\p$ is continuous and injective since $X^\div=\bar X^\div$. To see that the image is closed, $(\f_j)_j$ is a net in $\PSH(X,\p)$ such that $\f_j-t\p$ converges pointwise on $X^\div$ to $\f'\in\PSH_\hom(\bar X, tD)$, and set $\f\=\f'+t\psi$. By~\cite[Corollary~6.17]{BJpp}, $\f'$ can be written as a pointwise limit of a decreasing net $(\tau'_i)$ in $\cH_\hom(\bar X, tD)$. We claim that $\tau_i\=\tau_i'+t\psi\in \PSH(X)$. Indeed, write $\tau'_i\=m^{-1}\max_i\log|s_i|$ for a finite set of global sections $(s_i)$ of $mtD$ for $m$ divisible enough. Then 
    \[
    \tau_i= \tau'_i+t\p=m^{-1}\max_i\log|s_i|-m^{-1}\log|s_{mtD}|=m^{-1}\max_i\log|f_i|,
    \]
    for $f_i=s_i/s_{mtD}\in R$. Note also that $\f_j\ge 0$ implies $\f\ge0$; hence $\tau_i\ge0$. Thus $\tau_i=m^{-1}\max_i\log^+|f_i|$, so $\tau_i\in \PSH(X)$ by Corollary~\ref{cor:pshchar}. This shows $\f=\lim_i\tau_i\in \PSH(X)$ and one further has $\f\le \psi$ by pointwise convergence. So $\f\in \PSH(X, \psi)$, as desired.
\end{proof}

\subsection{The envelope property}
We say that a family $\{\varphi_i\}_i$  of homogeneous functions on $X^\na$ has \emph{sublogarithmic growth} if there exists $\psi\in \cH(X)$ such that $\varphi_i\le \psi$ for all $i$.

\begin{defi}
    We say $X$ has the \emph{envelope property} if for any family $(\varphi_i)$ in $\PSH(X)$ of sublogarithmic growth, $\sup^\star_i\varphi_i\in \PSH(X)$.
\end{defi}
Here $\sup^\star_i\f_i$ denotes the usc regularization of $\sup_i\f_i$.
\begin{lem}\label{lem:divneg}
    For any family $(\varphi_i)$ in $\PSH(X)$ of sublogarithmic growth, we have
    \[
    \sup_i\varphi_i={\sup_i}^\star \varphi_i \text{ on } X^\div.
    \]
\end{lem}
\begin{proof}
    Pick $\psi=\psi_{\bar X, D}\in \cH(X)$ with $\varphi_i\in \PSH(X, \psi)$ for all $i$, and write $\p=\p_{\bar X,D}$ for an integrally closed compactification $\bar{X}$ of $X$, and $D\in\CDiv(\bar X)_\Q$ ample and strictly effective. Pick any $t\in\Q_{>1}$. By Corollary~\ref{cor:pshhomemb} we have $\varphi_i-t\p\in \PSH_\hom(\bar X, tD)$. It then follows from~\cite[Theorem~5.6]{BJpp} that 
    \[
    \sup_i(\f_i-t\p)={\sup_i}^\star (\f_i -t\p)
    \]
    on $\Xdiv$ and hence the same holds for $(\f_i)$.
\end{proof}

\begin{thm}\label{thm:envpropcrit}
    The following properties are equivalent:
    \begin{itemize}
        \item [(i)] the envelope property holds;
        \item [(ii)] the space $\PSH(X, \psi)\=\{\varphi\in \PSH(X)\mid \varphi\le \psi\}$ is compact, for some (or equivalently any) $\psi\in \cH(X)$;
        \item [(iii)] any increasing net $(\varphi_i)$ in $\PSH(X)$ of sublogarithmic growth converges in $\PSH(X)$.
    \end{itemize}
\end{thm}
\begin{proof}
    First assume~(i), and let $\psi\in \PL^+(X)$. Pick a net $(\varphi_i)$ in $\PSH(X, \psi)$. For each $v\in X^\div_\psi$, $(\varphi_i(v))$ defines a bounded net in $\R$. By Tychonoff's theorem, we may assume, after passing to a subnet, that $\varphi_i$ converges pointwise to a function $\varphi\colon X^\div_\psi\to [0,1]$. Extending by homogeneity, this defines a non-negative homogeneous function on $X^\div$, which we still denote by $\varphi$. We need to show that $\varphi$ extends to a psh function on $X$, which will then necessarily be unique and in $\PSH(X, \psi)$, by Theorem~\ref{thm:weaktop}. By~(i), for each $i$, $\rho_i\= \sup_{j\ge i}^\star\varphi_j\in \PSH(X)$ and by construction $(\rho_i)$ is a decreasing net. Thus, the pointwise limit $\rho\=\lim_i\rho_i\in \PSH(X)$, and by Lemma~\ref{lem:divneg}, $\rho=\varphi$ on $X^\div$. This proves~(ii).

    Next assume~(ii) is true for some $\psi$, and let $(\varphi_i)$ be an increasing net in $\PSH(X)$ with $\varphi_i\le\rho$ for some $\rho\in \PL^+(X)$. By Lemma~\ref{lem:PLloggr} we may find $C\ge 1$ with $\rho\le C\psi$, and thus $(C^{-1}\varphi_i)$ defines an increasing net in $\PSH(X, \psi)$. By~(ii) this yields a limit $\varphi'\in \PSH(X, \psi)\subset \PSH(X)$, and $\varphi_i\to \varphi\=C\varphi'$ in $\PSH(X)$.

    Finally assume~(iii) and consider a family $(\varphi_i)_{i\in I}$ in $\PSH(X)$ with $\varphi_i\le \rho$ for some $\rho\in \PL^+(X)$. The increasing net defined by $\varphi_J\=\max_{i\in J} \varphi_i$ for $J\subset I$ finite subsets has a limit $\varphi\in \PSH(X)$. It remains to show $\varphi=\sup_i^\star\varphi_i$. On $X^\div$, by Lemma~\ref{lem:divneg}, we have
    \[
    \varphi=\sup_J\max_{i\in J}\varphi_i=\sup_i\varphi_i={\sup_i}^\star\varphi_i.
    \]
    We conclude that $\varphi\le\sup_i^\star\varphi_i$ on $X^\na$ by Theorem~\ref{thm:weaktop}, whereas the reverse inequality is clear from $\f_i\le\f$ and $\f$ being usc.
\end{proof}
\begin{cor}
    If $X$ is smooth, then the envelope property holds on $X$.
\end{cor}
\begin{proof}
    Pick any polarized smooth compactification $\bar{X}$ of $X$. By~\cite[Theorem~5.20]{BJpp}, the envelope property holds for all ample classes on $\bar X$, so by Corollary~6.22 in~\loccit, $\PSH_\hom(\bar{X},D)$ is compact for all ample strictly effective divisors $D\in\CDiv_\infty(\bar X)$. Pick any such $D$, and write $\p:=\p_{\bar X,D}\in\cH(X)$. Corollary~\ref{cor:pshhomemb} with $t=2$ now gives a closed embedding $\PSH(X,\p)\hto\PSH_\hom(\bar{X},2D)$ given by $\f\mapsto \f-2\p$. Thus $\PSH(X,\p)$ is compact, which completes the proof in view of Theorem~\ref{thm:envpropcrit}.
\end{proof}
\subsection{Envelopes}\label{sec:FSenv}
For later use, we define and study two natural envelopes associated to non-negative functions on $X^\na$.
\begin{defi}
    Given any function $\varphi\colon X^\na\to \R_+\cup \{\infty\}$, the \emph{psh envelope} of $\varphi$ is the function $P(\varphi)\colon X^\na\to \R_+\cup\{+\infty\}$ defined by
    \[
    P(\varphi)\=\sup\{\psi\in \PSH(X)\mid \psi\le \varphi\},
    \]
    and the \emph{Fubini--Study envelope} of $\varphi$ is
    \[
    Q(\varphi)\= \sup\{\psi\in \cH(X)\mid \psi\le \varphi\}.
    \]
\end{defi}
The psh envelopes satisfy the following properties, and similarly for the Fubini--Study envelope:
\begin{itemize}
    \item $P(\varphi)\ge0$ and is homogeneous;
    \item $P(t\varphi)=tP(\varphi)$;
    \item $\varphi\mapsto P(\varphi)$ is increasing;
    \item $P(\varphi)=P(\hat\varphi)$, where $\hat\varphi$ is the homogenization of $\varphi$: $\hat\varphi(v)\=\inf_{t>0}t^{-1}\varphi(tv)$.
\end{itemize}
An immediate observation is that $Q(\varphi)\equiv 0$ unless $\varphi$ dominates some function in $\cH(X)$. In the case when $Q(\varphi)\not\equiv 0$, by density of $\cH(X)\subset \CPSH(X)$, it is not hard to see that 
\[
Q(\varphi)=\sup\{\psi\in\CPSH(X)\mid \psi\le \varphi\}.
\]

\begin{lem}\label{lem:FSpshenv}
    Let $\varphi\colon X^\na\to \R_+\cup \{\infty\}$ be a homogeneous function dominating some $\rho\in \PL^+(X)$, and let $\varphi_\star$ be its lsc regularization. Then
    \begin{itemize}
        \item [(i)] $Q(\varphi)=P(\varphi_\star)$;
        \item[(ii)] $Q(\varphi)\in C^0(X^\na)$ iff $Q(\varphi)\in \PSH(X)$;
        \item[(iii)] If $\varphi$ is further a pointwise limit of an increasing net $(\varphi_j)$ of non-negative homogeneous lsc functions, then $P(\varphi_j)\nearrow P(\varphi)$.
    \end{itemize}
\end{lem}
\begin{proof}
    First note that if $\psi\in \cH(X)$ and $\psi\le \varphi$, then by continuity of $\psi$, we have $\psi\le \varphi_\star$ and it follows that $Q(\varphi)\le P(\varphi_\star)$. Conversely, suppose $\psi\le \varphi_\star$ for a psh function $\psi$, and pick $\psi_i\in \cH(X)$ a decreasing net converging pointwise to $\psi$. Then for any $t>1$ and $v\in X^\na_\rho$, $\{\psi_i<t\varphi_\star\}$ defines an open neighborhood of $v$ in $X^\na_\rho$ for $i\gg 0$, by lower semicontinuity of $\varphi_\star$. Since $X^\na_\rho$ is compact, we have $\psi_i<t\varphi_\star$ on $X^\na_\rho$ for all $i\gg 0$. In particular, $\p\le\p_i\le t\f_\star$, which yields $P(\varphi_\star)\le Q(t\varphi_\star)=tQ(\varphi_\star)=tQ(\varphi)$, where the last equality follows since $Q(\varphi)$ is lsc. This proves the reverse inequality and hence~(i).

    For~(ii), if $Q(\varphi)\in \PSH(X)$, then it is usc. Since $Q(\varphi)$ is lsc by construction, this shows $Q(\varphi)\in C^0(X^\na)$. Conversely, suppose $Q(\f)$ is continuous. As $\cH(X)$ is closed under finite maxima, we may write $Q(\varphi)$ as the pointwise limit of an increasing net $(\varphi_j)$ in $\cH(X)$. By Dini's lemma, $\varphi_j\to Q(\varphi)$ uniformly on $X^\na_\rho$, and hence locally uniformly on $X^\na$. This shows $Q(\varphi)\in \PSH(X)$ by Theorem~\ref{thm:pshprop}.

    Finally we prove~(iii). It is clear that $\lim_jP(\varphi_j)\le P(\varphi)$. Further, $\f$ is lsc, so by~(i) we have $P(\f)=Q(\f)$. Now pick $\p\in\cH(X)$ with $\psi\le\varphi$. Then, for any $t>1$, arguing as before we have $\psi\le t\varphi_j$ for $j\gg 0$. This yields
    \[
    P(\f)=Q(\f)\le \sup_jQ(t\f_j)=\sup_jtQ(\f_j)\le\sup_jtP(\f_j),
    \]
    which completes the proof.
\end{proof}
\begin{rmk}
    In analogy with~\cite[Lemma~5.17]{BJpp}, we can characterize the envelope property using envelopes: the envelope property holds iff $P(\f)\in \CPSH(X)$ for any continuous homogeneous $\f\colon X^\na_\infty\to \R_{>0}$. 
\end{rmk}

\subsection{Associated Weil divisors at infinity}\label{sec:pshWeil}

As we saw in~\S\ref{sec:Weilb}, the set $\WbDiv_\infty(X)_\R$ of real Weil b-divisors at infinity is in bijection with the set of real-valued functions on the set $\PbDiv_\infty(X)$ of prime divisors at infinity, and therefore in bijection with the set of homogeneous real-valued functions on $X^\div_\infty$. 

If $\p\in\PSH(X)$, then $\p$ restricts to a non-negative homogeneous real-valued function on $X^\div_\infty$, and we write $D_\p$ for the corresponding effective Weil b-divisor at infinity.
\begin{lem}\label{lem:irregzero}
    Assume $X$ is smooth and admits a compactification $\bar X$ with $q(\bar X)=0$ and rational singularities. Then $\f, \p\in \PSH(X)$ define numerically equivalent Weil b-divisors $D_\f, D_\p$ iff $\f-\p=\sum_ic_i\log|f_i|$ for some $c_i>0$ and $f_i\in R^\times$.
\end{lem}
\begin{proof}
    It suffices to show the forward direction. First note that the assumption implies all smooth compactifications of $X$ has irregularity zero. Indeed, if $\bar X_1$ is a smooth compactification dominating $\bar X$, then $\mathrm{H}^1(\bar X_1, \cO_{\bar X_1})\simeq \mathrm{H}^1(\bar X, \cO_{\bar X})=0$, since $\bar X$ has rational singularities and $q(\bar X)=0$.
    Let $\bar X'$ be any normal compactification of $X$, and pick a smooth compactification $\bar X''$ dominating $\bar X'$. Then the map $\mathrm{H}^1(\bar X', \cO_{\bar X'})\to \mathrm{H}^1(\bar X'', \cO_{\bar X''})$ is injective, and so $q(\bar X')=0$. This implies all incarnations of $D_\f, D_\p$ on normal compactifications of $X$ are $\R$-linearly equivalent, and that $D_\f-D_\p\in \CbDiv_\infty(X)_\R$ by Lemma~\ref{lem:ntrivCDiv}. It follows that $\f-\p=\sum_i c_i\log|f_i|$ for some $c_i>0$ and $f_i\in R^\times$.
\end{proof}
The associated numerical class $[D_\p]\in\WNum_{n-1}(X)$, is by definition of $\PSH(X)$ the limit of nef Cartier classes, and hence a nef class in the sense of~\cite{DangFavre}. 

As a converse, we have
\begin{prop}\label{prop:nefweilpsh}
    Assume $X$ is smooth, and let $D\in\WbDiv_\infty(X)_\R$ be an effective Weil b-divisor at infinity whose numerical class $[D]\in\WNum_{n-1}(X)$ is nef. Then the associated homogeneous function $\p_D\colon X^\div\to\R_+$ is the restriction of a (unique) function in $\PSH(X)$.
\end{prop}
\begin{proof}
    Pick a polarizable smooth compactification $\bar{X}$ of $X$. The trace $D_{\bar{X}}$ of $D$ on $\bar{X}$ is effective. Pick an ample effective divisor $H\in\CDiv_\infty(\bar{X})_\R$ such that $L:=D_{\bar{X}}+H$ is an ample $\Q$-divisor. For example, given an ample strictly effective divisor $A$ at infinity, we can pick $L=mA$ and $H=mA-D_{\bar{X}}$ for $m$ sufficiently divisible.
    
    Applying~\cite[Theorem~6.40]{BJpp} to $[D]$ shows that the function $\p_D-\p_{D_{\bar{X}}}$ on $X^\div={\bar{X}}^\div$ is the restriction of a function $u\in\PSH_\hom(\bar{X},[D_{\bar X}])$ on $\bar{X}^\val=X^\val$. As $H$ is effective, $-\p_H$ is $[H]$-psh, and hence $\f:=u-\p_H=\p_D-\p_L\in\PSH_\hom(\bar{X},L)$.
    
    By Corollary~6.17 and Remark~6.18 of~\cite{BJpp}, we can find a decreasing sequence $(\f_l)_l$ of functions in $\cH_\hom(\bar{X},L)$ converging to $\f$ pointwise on $\bar{X}^\val=X^\val$. Thus each function $\f_l$ is of the form $\f_l=\frac1{m_l}\max_j\log|s_{l,j}|$, where $m_l\in\Z_{>0}$ is sufficiently divisible so that $m_lL$ is a $\Z$-divisor, and $\{s_{l,j}\}_j$ is a finite set of nonzero global sections of $m_lL$. As $\f_l\ge\f=0$ on $X^\beth\cap X^\val$, these sections have no common zero on $X$. If $s_{m_lL}$ is the canonical section of $m_lL$, then $f_{l,j}:=s_{l,j}/s_{m_lL}\in R$, and the functions $\{f_{l,j}\}_j$ have no common zero on $X$. 
    It follows that $\p_l:=\max\{0,\f_l+\p_L\}=\frac1{m_l}\max_j\log^+|f_{l,j}|\in\PSH(X)$. 
    Hence $(\p_l)_l$ forms a decreasing sequence in $\PSH(X)$ that converges to $\p_D$ pointwise on $X^\div$, so $\p_D$ extends uniquely to a function in $\PSH(X)$ by Theorem~\ref{thm:pshprop}.
\end{proof}

\section{Curvature and Monge--Amp\`ere operator}\label{sec:numint}
In this section, $X$ is an affine variety of dimension $n\ge2$. We define the curvature operator on psh functions by sending a function to the numerical class of its associated Weil b-divisor. By taking decreasing limits, we extend the intersection pairing on PL functions to arbitrary psh functions and define the Monge--Amp\`ere operator. We also study the metric space structure on the space of numerical classes defined by psh functions of positive volume.
\subsection{Curvature and intersection pairing on PL functions}\label{sec:curvPL}
Following~\cite{BFJ16} we define the \emph{curvature operator}
\begin{equation}\label{equ:curvop}
\ddc\colon \PL(X)\to \CNum(X) 
\end{equation}
as follows: if $\varphi\in\PL(X)$ is given by a $\Q$-Cartier divisor $D\in \CDiv_\infty(\bar X)_\Q$ for some compactification $\bar X$, then $\ddc\varphi\in\CNum(X)$ is the Cartier b-divisor class defined by $D$. Its image under the map $\CNum(X)\to\WNum_{n-1}(X)$ equals the numerical class induced by $\f|_{X^\div}$, see~\S\ref{sec:numcld}.

\begin{exam}
    If $f\in R$ is invertible, then $\log|f|\in\PL(X)$, and $\ddc\log|f|=0$. Indeed, pick any compactification $\bar{X}$ of $X$, and set $D=-\div_{\bar{X}}(f)$. Then $D$ is numerically trivial and $\log|f|=\p_D$. In the terminology of~\cite{CLD25}, this is compatible with the fact that $\log|f|$ is pluriharmonic, see~\cite[Lemme~16.7.4]{CLD25}.  
\end{exam}
\begin{exam}
    Let $E$ be an elliptic curve, and $p, q\in E$ be two distinct points. Then $X\=(E\setminus\{p, q\})\times\A^1$ is affine, and denote by $D_p=\{p\}\times\P^1, D_q\=\{q\}\times \P^1$ divisors at infinity on the compactification $E\times \P^1$. Then $D=D_p-D_q$ is a numerically trivial (but not trivial, if $p-q$ is not torsion) divisor at infinity, so $\ddc\p_D=0$.
\end{exam}
Via the curvature operator, the intersection pairing $\CNum(X)^n\to\R$  induces an intersection pairing $\PL(X)^n\to\Q$, which we write as
\[
    (\f_1,\dots,\f_n)\in\Q
\]
for $\f_1,\dots\f_n\in\PL(X)$, and which can be more concretely computed as follows: given $\f_1,\dots,\f_n\in\PL(X)$, pick an integrally closed compactification $\bar X$ of $X$ such that each $\f_i$ is defined by a $\Q$-Cartier divisor $D_i$ on $\bar X$. Then
\[
(\f_1,\dots,\f_n)=(D_1\cdot\ldots\cdot D_n).
\]
We also write $\vol(\f)\=(\f^n)=(\f,\dots,\f)$.
Note that $(\f_1,\dots,\f_n)>0$ when the $\f_i$ are all Fubini--Study functions. Indeed, the classes $\ddc\f_i\in\CNum(X)$ are then big and nef.
We recall the following classical inequalities, most of which follow from the Hodge Index Theorem, see~\cite{pos1,rKT}.
\begin{prop}\label{prop:Hineq}
    The intersection pairing on $\cH(X)$ satisfies the following:
    \begin{itemize}
        \item[(i)] for any $1\le p\le n$ and $\varphi_1, \dots, \varphi_p, \psi_1, \dots, \psi_{n-p}\in \cH(X)$, we have 
        \[
        (\varphi_1\dots\varphi_p, \psi_1, \dots, \psi_{n-p})\ge (\varphi_1^p, \psi_1, \dots, \psi_{n-p})^{1/p}\dots (\varphi_p^p, \psi_1, \dots, \psi_{n-p})^{1/p};
        \]
        \item [(ii)] for any $\varphi\in \cH(X)$, $\varphi\mapsto (\varphi^n)^{1/n}$ is concave;
        \item [(iii)] for any $\varphi, \psi\in \cH(X)$, the map $k\mapsto \log(\varphi^k, \psi^{n-k})$ is concave.
        \item[(iv)] for any $\varphi, \psi, \rho\in \cH(X)$, and $1\le k\le n-1$, we have
        \[
        (\varphi^k, \rho^{n-k})(\rho^k, \psi^{n-k})\ge \frac{k!(n-k)!}{n!}(\rho^n)(\varphi^k, \psi^{n-k}).
        \]
        \item[(v)] for $\f, \p\in\cH(X)$, $\ddc\f\le n\frac{(\f, \p^{n-1})}{(\p^n)}\ddc\p$;
    \end{itemize}
\end{prop}
\subsection{Monge--Amp\`ere operator on PL functions}\label{sec:MAPL}
The curvature operator and the multilinear map $\PL(X)^{n-1}\to\WNum_1(X)$ also induce a \emph{Monge--Amp\`ere operator} that we write as  
\[
\MA(\f_1,\dots,\f_{n-1}):=\ddc\f_1\wedge\dots\wedge \ddc\f_{n-1}\in\WNum_1(X)
\]
for $\f_1,\dots,\f_{n-1}\in\PL(X)$. We also write $\MA(\f)=(\ddc\f)^{n-1}$.

Note that the Monge--Amp\`ere operator takes values in curve classes rather than measures. If $\rho\in\cH(X)$ is a Fubini--Study function, associated to a big and semiample divisor $D\in\CbDiv_\infty(X)_\Q$, then the curve class $\MA(\f_1,\dots,\f_{n-1})$ also defines a signed Radon measure $\MA_\rho(\f_1,\dots,\f_{n-1})$ on the compact level set $X_\rho^\na:=\{\rho=1\}\subset X^\na$, uniquely determined by 
\[
\int_{X_\rho}\p\MA_\rho(\f_1,\dots,\f_{n-1})=(\f_1,\dots,\f_{n-1},\p)
\]
for any $\p\in\PL(X)$. This measure is well-defined since the right-hand side defines a bounded linear functional on the dense linear subspace $\PL(X)_\R|_{X_\rho^\na}\subset C^0(X_\rho^\na)$, see Proposition~\ref{prop:SW}.

More concretely, suppose that $\bar{X}$ is an integrally closed compactification on which $\rho$ and the $\f_i$ are determined, the latter by divisors $D_i\in\CDiv_\infty(\bar{X})_\Q$. Then 
\begin{equation}\label{equ:PLmeas}
\MA_\rho(\f_1,\dots,\f_{n-1})=\sum_E b_E(D_1|_E\cdot \dots\cdot D_{n-1}|_E)\delta_{b_E^{-1}\ord_E},
\end{equation}
where the sum runs through all irreducible components $E\subset \bar X\setminus X$ and $b_E\=\rho(\ord_E)$. 

If $X=\A^n$ or more generally when $X$ is normal and $\mathrm{Cl}(X)$ is torsion, the above measure determines the curve class.
The following example shows that in general the curve class $\MA(\f_1, \dots, \f_{n-1})$ cannot be determined by the measure on $X^\na_\rho$. However, recall from Lemma~\ref{lem:curvmeascart} that if $\f\in\cH(X)$, then $\MA(\f)$ is uniquely determined by its induced measure $\MA_\rho(\f)$ on $X^\na_\rho$.
\begin{exam}\label{exam:curvmeasneq}
Take $\bar{X}=\P^1\times\P^1\times\P^1$, with fiber divisor classes $h_1,h_2,h_3$. Set $A:=(x_1=x_2)\sim h_1+h_2$ and $B:=(x_2=x_3)\sim h_2+h_3$. Then $A+B\sim h_1+2h_2+h_3$ is ample, so $X:=\bar X\setminus(A\cup B)$ is affine. Now set $D_1=A+B$, $D_2=3A+5B$, $D'_1=A+4B$, $D'_2=2A+B$. These are ample  Cartier divisors at infinity on $\bar{X}$ and correspond to functions $\f_1,\f_2,\f'_1,\f'_2\in\cH(X)$. We claim that $\MA(\f_1,\f_2)\ne\MA(\f'_1,\f'_2)$ but $\MA_\rho(\f_1,\f_2)=\MA_\rho(\f'_1,\f'_2)$ for any $\rho\in\cH(X)$. To see this, note that $\Delta:=D_1\cdot D_2-D'_1\cdot D'_2=h_1h_2-h_1h_3+h_2h_3$. This yields $\Delta\cdot A=\Delta\cdot B=0$, but $\Delta\cdot h_1=1$.
\end{exam}
When the $\f_i$ are Fubini--Study functions, we have $\MA(\f_1,\dots,\f_{n-1})\in\BPF_1(X)$, and $\MA_\rho(\f_1,\dots,\f_{n-1})$ is a positive measure on $X_\rho^\na$ of mass $(\f_1,\dots,\f_{n-1},\rho)$.
\begin{rmk}\label{rmk:MAatinfty}
    When $X=\A^2=\Spec k[z_1,z_2]$, and $\rho=\max_i\log^+|z_i|$, $X_\rho^\na$ is the space introduced in~\cite{FJEigenvals}, where it was called the \emph{valuative tree at infinity}, and the Monge--Amp\`ere operator $\MA_\rho$ corresponds to the tree Laplacian defined there; see also~\cite{FJDynComp,JonDynBerko}.
\end{rmk}
\begin{rmk}\label{rmk:MACLD}
    Any Fubini--Study function is a continuous psh function on $X^\na$ in the sense of~\cite{CLD25}. If $\f_1,\dots,\f_n\in\cH(X)$, we expect the Monge--Amp\`ere measure $d'd''\f_1\wedge\dots\wedge d'd''\f_n$ in the sense of~\cite{CLD25} to be a Dirac mass at the trivial valuation on $X$, with mass $(\f_1,\dots,\f_n)$. Similarly, if $\rho\in\cH(X)$, then $\max\{\rho,1\}$ is a continuous psh function in the sense of~\cite{CLD25}, and we expect $\MA_\rho(\f_1,\dots,\f_{n-1})$ to equal $d'd''\f_1\wedge\dots\wedge d'd''\f_{n-1}\wedge d'd''\max\{\rho,1\}$.
\end{rmk}

\subsection{Curvature and intersection pairing on psh functions}
The constructions in~\S\ref{sec:curvPL} and~\S\ref{sec:MAPL} apply in particular to Fubini--Study functions. We now extend them to psh functions.

Let us start with the curvature operator. Recall  from~\S\ref{sec:pshWeil} that each psh function induces a nef and effective Weil b-divisor at infinity. Partly extending~\eqref{equ:curvop}, we define 
\[
\ddc\colon\PSH(X)\to\WNum_{n-1}(X), \quad \varphi\mapsto \ddc\varphi,
\]
taking a function $\varphi$ to the numerical class represented by the associated Weil b-divisor. The class $\ddc\f$ is nef in the sense of~\cite{DangFavre}, as it is the (decreasing) limit of nef Cartier classes. Two functions $\f, \p\in \PSH(X)$ satisfy $\ddc\f=\ddc\p$ iff the restriction of $\tau\=\f-\p$ to $X^\div$ defines a numerically trivial Weil b-divisor. If in addition $X$ is normal with isolated singularities, then $\tau\in \PL(X)\otimes_\Q\R$ by Lemma~\ref{lem:ntrivCDiv}.

\begin{thm}\label{thm:pshint}
    For $p=1, n-1, n$, and fixed $\gamma\in \BPF_p(X)$ if $p=1, n-1$,
    there exists a unique extension of the intersection pairing on $\cH(X)$ to a symmetric map
    \[
    \PSH(X)^p\to \R_+,
    \quad (\varphi_i)_{i=1}^p\mapsto(\varphi_1, \varphi_2, \dots, \varphi_p, \gamma),
    \]
    that is upper semicontinuous, increasing and $\R_{>0}$-linear in each variable. This pairing is furthermore continuous along decreasing nets. 
\end{thm}
\begin{proof}
    We follow the proof of~\cite[Theorem~7.1]{BJpp}. Uniqueness is clear since every element of $\PSH(X)$ is the limit of a decreasing net in $\cH(X)$. As for existence, we set
     \[
    (\varphi_1, \varphi_2, \dots, \varphi_p, \gamma)\=\inf_{\psi_i\in\cH(X), \psi_i\ge \varphi_i}(\psi_1, \psi_2, \dots, \psi_p, \gamma),
    \]
    for any $\f_1,\dots,\f_p\in\PSH(X)$. By monotonicity,  this extends the pairing on $\cH(X)$, and is symmetric and increasing.
    
    It only remains to prove upper semicontinuity; continuity along decreasing nets then follows formally, and implies $\R_{>0}$-multilinearity. Thus     
    suppose $\varphi_i\in\PSH(X)$, $i=1,\dots,p$, are such that 
    \[
    (\varphi_1, \dots, \varphi_p, \gamma)<t
    \]
    for some $t\in \R_{>0}$. (Here $\gamma$ is omitted when $p=n$.) Then for $0<\varepsilon\ll1$ we can find $\psi_i\in \cH(X), i=1, \dots, p$ with $\psi_i\ge \varphi_i$ for each $i$ and 
    \[
    (\psi_1, \dots, \psi_p, \gamma)<(1-\varepsilon) t.
    \]
    Note that by Corollary~\ref{cor:supcts}, for any fixed constant $c\in\Q_{>1}$ and any fixed $\rho\in\cH(X)$, we can find open neighborhoods $U_i$ of $\varphi_i$ such that for any $\varphi_i'\in U_i$, one has $\varphi_i'< c\psi_i$ on $X_\rho$, and hence $\f'_i\le c\p_i$ on all of $X^\na$. This implies
    \[
    (\varphi_1', \dots, \varphi_p', \gamma)\le (c\psi_1, \dots, c\psi_p, \gamma)<c^p(1-\varepsilon)t<t,
    \]
    by choosing $c$ such that $c^p(1-\varepsilon)<1$, proving upper semicontinuity.
\end{proof} 

Next we show that the intersection pairing defined above is numerical in nature, and agrees with the one considered in~\cite{DangFavre} as recalled in~\S\ref{sec:numcld}.
\begin{prop}\label{prop:intnum}
    The intersection pairing depends only on the numerical class. In particular, for $p=0, 1, n-1$ and $\gamma=\ddc\f_{p+1}\wedge\dots\wedge\ddc\f_n$ for $\f_{p+1}, \dots, \f_n\in\PSH(X)$, we have
    \[
    (\f_1, \dots, \f_p, \gamma)=(\f_1, \dots, \f_n)
    \]
    for $\f_1, \dots, \f_p\in \PSH(X)$.
\end{prop}
\begin{proof}
    Let $\varphi_0, \dots, \varphi_{n-1}, \varphi_0'\in \PSH(X)$ with $\ddc\varphi_0=\ddc\varphi_0'$. We are going to show
    \[
    (\varphi_0, \varphi_1, \dots, \varphi_{n-1})=(\varphi_0', \varphi_1, \dots, \varphi_{n-1}).
    \]
    The same argument will show that the intersection pairing only depends on the numerical classes.
    Let $(\varphi_{0j}')_j, (\varphi_{ij})_j, i=0, \dots, n-1$ be decreasing nets in $\cH(X)$ converging to $\varphi_0', \varphi_0, \dots, \varphi_{n-1}$ respectively. We also write $\alpha_i\=\ddc\varphi_i, i=0, \dots, n-1$ for the b-divisor classes, and similarly set $\alpha_{ij}\=\ddc\varphi_{ij}, i=0, \dots, n-1$ and $\alpha_{0j}'\=\ddc\varphi_{0j}'$. Then $(\alpha_{ij})$ in particular form decreasing nets of Cartier b-divisor classes converging to $\alpha_i$, by Lemma~\ref{lem:numclasscts}. 
    By continuity along decreasing nets, see~\cite[Theorem~3.2]{DangFavre}, we then have
    \begin{align*}
        (\varphi_0, \varphi_1, \dots, \varphi_{n-1})
        &=\lim_j (\varphi_{0j}, \varphi_{1j}, \dots, \varphi_{n-1,j})
        =\lim_j (\alpha_{0j}, \alpha_{1j}, \dots, \alpha_{n-1,j})\\
        &=(\alpha_{0}, \alpha_{1}, \dots, \alpha_{n-1})
        =\lim_j (\alpha_{0j}', \alpha_{1j}, \dots, \alpha_{n-1,j})\\
        &=\lim_j (\varphi_{0j}', \varphi_{1j}, \dots, \varphi_{n-1,j})
        =(\varphi_0', \varphi_1, \dots, \varphi_{n-1}).
    \end{align*}
    We are done by symmetry and induction. The second assertion follows from the fact that both intersection pairings in Theorem~\ref{thm:pshint} and~\cite{DangFavre} are continuous along decreasing nets.
\end{proof}

We will write $(\f^p, \f_1, \dots, \f_{n-p})\=(\f, \dots, \f, \f_1, \dots, \f_{n-p})$, and also
$\vol(\f):=(\f^n)$.
The following inequalities are direct consequences of monotone approximation and Proposition~\ref{prop:Hineq}:
\begin{prop}\label{prop:intineq}
    The intersection pairing on $\PSH(X)$ satisfies the following:
    \begin{itemize}
        \item[(i)] for any $1\le p\le n$ and $\varphi_1, \dots, \varphi_p, \psi_1, \dots, \psi_{n-p}\in \PSH(X)$, we have 
        \[
        (\varphi_1\dots\varphi_p, \psi_1, \dots, \psi_{n-p})\ge (\varphi_1^p, \psi_1, \dots, \psi_{n-p})^{1/p}\dots (\varphi_p^p, \psi_1, \dots, \psi_{n-p})^{1/p};
        \]
        \item [(ii)] for any $\varphi\in \PSH(X)$, $\varphi\mapsto\vol(\f)^{1/n}$ is concave;
        \item [(iii)] for any $\varphi, \psi\in \PSH(X)$, the map $k\mapsto \log(\varphi^k, \psi^{n-k})$ is concave.
        \item[(iv)] for any $\varphi, \psi, \rho\in \PSH(X)$, and $1\le k\le n-1$, we have
        \[
        (\varphi^k, \rho^{n-k})(\rho^k, \psi^{n-k})\ge \frac{k!(n-k)!}{n!}(\rho^n)(\varphi^k, \psi^{n-k}).
        \]
    \end{itemize}
\end{prop}
We also record the following variant of Siu's inequality, see~\cite[Proposition~1.1, Corollary~3.4]{DangFavre}.

\begin{cor}\label{cor:siu}
    For any $\varphi, \psi\in \PSH(X)$ with $(\psi^n)>0$, we have
    \[
    \ddc\varphi\le n\frac{(\varphi, \psi^{n-1})}{(\psi^n)}\ddc\psi.
    \]
\end{cor}

\subsection{Monge--Amp\`ere operator on psh functions}\label{sec:MAPSH}
Next we extend the Monge--Amp\`ere operator in~\S\ref{sec:MAPL} to psh functions. Recall from~\S\ref{sec:numclc} that $\BPF_1(X)$ consists of curve classes that define a positive linear functional on $\CNum(X)$. In particular, by density of $\PL(X)\subset C^0(X^\na_\rho)$ for any $\rho\in\cH(X)$, any curve class $\mu\in \BPF_1(X)$ induces a positive measure of mass $(\rho, \mu)$ on $X^\na_\rho$. 
\begin{prop}\label{prop:MAop}
    There exists a unique extension of the Monge--Amp\`ere operator on $\cH(X)$ to a symmetric, $\R_{>0}$-multilinear map
    \[
    \PSH(X)^{n-1}\to\BPF_1(X),
    \quad (\f_i)_{i=1}^{n-1}\mapsto\MA(\f_1,\dots,\f_{n-1})=\ddc\f_1\wedge\dots\wedge\ddc\f_{n-1},
    \]
    that is continuous along decreasing nets.
    When $(\f_i^n)>0$, $1\le i\le n-1$, the curve class $\MA(\f_1,\dots,\f_{n-1})$ is nonzero.
\end{prop}
For any $\rho\in\cH(X)$, the BPF curve class $\MA(\f_1,\dots,\f_{n-1})$ induces a unique positive measure $\MA_\rho(\f_1,\dots,\f_{n-1})$ on $X_\rho$. As before we write $\MA(\f)=\MA(\f,\dots,\f)$ and $\MA_\rho(\f)=\MA_\rho(\f,\dots,\f)$ for $\f\in\PSH(X)$.

\begin{proof}
    The first assertion follows from Theorem~\ref{thm:pshint} and Proposition~\ref{prop:intnum}. If $(\f_i^n)>0$, then by Corollary~\ref{cor:siu}, we have $\ddc\f_i\ge c_i\ddc\rho$ for $c_i>0$ and $\rho\in \cH(X)$. It follows that $(\rho, \MA(\f_1, \dots,\f_{n-1}))\ge \prod c_i(\rho^n)>0$, and so $\MA(\f_1, \dots, \f_{n-1})$ is nonzero.
    
\end{proof}

\subsection{Numerical classes of psh functions of positive volume}\label{sec:strongP}
We now restrict our discussion on the strong topology on $\BBPF_{n-1}(X)$ to infinity. We first set up some notation.
For $\varphi\in \PSH(X)$ we write $[\varphi]$ for the numerical class represented by $\f$. Denote by 
\[
\cE(X)\=\{\varphi\in \PSH(X)\mid \vol(\varphi)\=(\f^n)>0\}
\]
the set of psh functions of positive volume. By what precedes, any function $\f\in \cE(X)$ defines a nonzero curve class $\MA(\f)$ via the Monge--Amp\`ere operator.
\begin{exam}\label{exam:posvolloggr}
    By monotonicity of the intersection pairing, every psh function of logarithmic growth is in $\cE(X)$. But in general, not every function in $\cE(X)$ has logarithmic growth. For example, when $X=\A^1\times \G_m$ and let $D_{(0,1)}, D_{(-1,0)}, D_{(0,-1)}$ be the toric divisors on $\P^1\times \P^1$ at infinity. The function $\psi_{D_{(0,1)}+D_{(-1,0)}}\in \cE(X)\setminus\cL^\infty$.
\end{exam}
For functions in $\cE(X)$, we set
\[
\bar\cE(X)\=\{[\varphi]\mid \varphi\in \cE(X)\}\subseteq \BBPF_{n-1}(X).
\]

The space $\bar\cE(X)$ has a natural weak topology induced by the weak topology on $\WNum_{n-1}(X)$ and the $\de_\infty$ metric on $\BBPF_{n-1}(X)$ restricts to a metric on $\bar \cE(X)$, as in~\eqref{equ:dinfty1}:
    \[
    \de_\infty([\varphi], [\psi])\=\inf\{\delta\ge0\mid e^{-\delta}\ddc\varphi\le \ddc\psi\le e^\delta\ddc\varphi\}=\log\max\{\frac{1}{s(\varphi, \psi)}, \frac{1}{s(\psi, \varphi)}\},
    \]
    for all $[\varphi], [\psi]\in \bar\cE(X)$, where the \emph{pseudoeffective slope} of $\psi$ with respect to $\varphi$ is
    \[
    s(\varphi, \psi)\=\sup\{t>0\mid \ddc\varphi\ge t\ddc\psi\}>0.
    \]
\begin{defi}
    The \emph{strong topology} on $\bar\cE(X)$ is the metric topology defined by $\de_\infty$. 
\end{defi}

More quantitatively, we have
\begin{prop}\label{prop:currentconv}
    Assume $([\f_i])$ is a net in $\bar \cE(X)$ converging strongly to $[\f]\in \bar\cE(X)$.
    Then for any $k=1, \dots, n-1$, and any $f_1, \dots, f_{n-k-1}\in \cH(X)$, we have
    \[
    (\ddc\varphi_i)^k\wedge \ddc f_1\wedge\dots\wedge \ddc f_{n-k-1} \to (\ddc\varphi)^k\wedge \ddc f_1\wedge\dots\wedge \ddc f_{n-k-1}
    \]
    weakly as BPF curve classes. In particular, they also converge weakly as measures on $X^\na_\rho$ for some (or equivalently any) $\rho\in \cH(X)$. 
    Further, 
    \[
    \frac{(\psi, \varphi_i^k, f_1, \dots, f_{n-k-1})}{(\psi, \varphi^k, f_1, \dots, f_{n-k-1})}\to 1
    \]
    uniformly with respect to $\psi\in \cE(X)$.
\end{prop}
\begin{proof}
    Set $\delta_i\=\de_\infty([\varphi_i],[ \varphi])$. We have $\delta_i\to 0$ and for nef classes $\alpha\in \CNum(X)$, 
    \[
    e^{-k\delta_i}(\alpha, \varphi^k,  f_1, \dots, f_{n-k-1})
    \le (\alpha, \varphi_i^k,  f_1, \dots, f_{n-k-1})
    \le e^{k\delta_i}(\alpha, \varphi^k,  f_1, \dots, f_{n-k-1}).
    \]
    Taking $i\to\infty$, we get the desired convergence since $\CNum(X)$ is generated by nef classes. The same argument proves the remaining assertions, by noting that $\PL(X)$ is dense in $C^0(X^\na_\rho)$.
\end{proof}

Given $\varphi, \psi\in \cE(X)$, we also define
    \[
    \jj_\psi(\varphi)\= \log \frac{(\varphi, \psi^{n-1})}{\vol(\varphi)^{1/n}\vol(\psi)^{1-1/n}},
    \]
    \[
    \ii(\varphi, \psi)=\log \frac{(\varphi, \psi^{n-1})(\psi, \varphi^{n-1})}{\vol(\varphi)\vol(\psi)}
    =\jj_\psi(\varphi)+\jj_\varphi(\psi),
    \]
    and
\begin{equation}
\bar \ii(\varphi, \psi)\=\ii(\varphi, \psi)+|\log (\varphi^n)-\log (\psi^n)|
=\log \frac{(\varphi, \psi^{n-1})(\psi, \varphi^{n-1})}{\min\{(\varphi^n),(\psi^n)\}^2}.
\end{equation}
By Proposition~\ref{prop:intnum}, the functionals defined above only depends on the numerical classes. 
As an immediate corollary of Proposition~\ref{prop:metrictop}, we get an analog of~\cite[Corollary~10.4]{BJpp}.
\begin{cor}\label{cor:numequiv}
    Given $\varphi, \psi\in \cE(X)$, the following are equivalent:
    \begin{itemize}
        \item [(i)] $\bar \ii(\varphi, \psi)=\de_\infty([\f],[\p])=0$;
        \item[(ii)] $\ddc\varphi=\ddc\psi$;
        \item[(iii)] $\MA(\varphi)=\MA(\psi)$;
        \item[(iv)] $\MA_\rho(\varphi)=\MA_\rho(\psi)$ on $X^\na_\rho$ for any $\rho\in \cH(X)$.
    \end{itemize}
\end{cor}
\begin{proof}
    Obviously we have~(i)$\Rightarrow$(ii)$\Rightarrow$(iii)$\Rightarrow$(iv). Assume~(iv) holds. Then we have $(\f^n)=(\f, \psi^{n-1})$ and $(\psi^n)=(\psi, \f^{n-1})$, showing $\ii(\f,\p)=0$. In particular, $\jj_\psi(\f)=\frac{n-1}{n}(\log(\varphi^n)-\log(\psi^n))=0$. This shows $\bar \ii(\f, \psi)=0$, which forces $\de_\infty([\f],[\psi])=0$ by Proposition~\ref{prop:metrictop}.
\end{proof}
\begin{exam}\label{exam:decstrong}
    If $(\varphi_i)$ is a decreasing net in $\cE(X)$ converging to $\varphi\in \cE(X)$, then $[\varphi_i]\to [\varphi]$ strongly in $\bar \cE(X)$.
\end{exam}
\begin{cor}\label{cor:weakstrongconv}
    Let $(\varphi_i)$ be a net in $\cE(X)$ and assume $\varphi_i\to \varphi$ weakly for some $\varphi\in \cE(X)$. Then $[\varphi_i]\to [\varphi]$ strongly iff $\vol(\varphi_i)\to \vol(\varphi)$.
\end{cor}
\begin{proof}
    The forward direction is obvious. Assume now $\vol(\varphi_i)\to \vol(\varphi)$. Then by upper semicontinuity of the intersection product that $0\le \limsup_i\jj_\varphi(\varphi_i)\le \jj_\varphi(\varphi)=0$, and hence $\ii(\varphi_i, \varphi)\approx \jj_\varphi(\varphi_i)\to 0$, which, together with the convergence of volumes, yields $\bar \ii(\varphi_i, \varphi)\to 0$.
\end{proof}
\begin{cor}\label{cor:monovolconv}
    Let $(\varphi_i)$ be a net in $\cE(X)$ and assume $\varphi_i\le \varphi$ for some $\varphi\in \cE(X)$. Then $[\varphi_i]\to [\varphi]$ strongly iff $\vol(\varphi_i)\to \vol(\varphi)$.
\end{cor}
\begin{proof}
    Assume $\vol(\varphi_i)\to \vol(\varphi)$. Then by monotonicity we have $0\le \limsup_i\jj_\varphi(\varphi_i)\le \jj_\varphi(\varphi)=0$, which shows $\bar\ii(\varphi_i, \varphi)\to 0$.
\end{proof}

\begin{cor}\label{cor:fcnstrtop}
    Let $[\varphi_i], [\varphi]\in \bar\cE(X)$. We have $[\varphi_i]\to  [\varphi]$ strongly iff $\MA(\varphi_i)\to \MA(\varphi)$ weakly and $\vol(\varphi_i)\to \vol(\varphi)$.
\end{cor}
\begin{proof}
    Suppose $\bar\ii(\varphi_i, \varphi)\to 0$, then it follows from Proposition~\ref{prop:currentconv} that $\MA(\varphi_i)\to \MA(\varphi)$ weakly, and by definition $\vol(\varphi_i)\to \vol(\varphi)$. 

    Conversely, pick $\psi_j\in \cH(X)$ a decreasing net converging to $\varphi$. Then for any $j$, we have
    \[
    \limsup_i (\varphi, \varphi_i^{n-1})\le \limsup_i (\psi_j, \varphi_i^{n-1})=(\psi_j, \varphi^{n-1}).
    \]
    Taking $j\to \infty$ we get $\limsup_i(\varphi, \varphi_i^{n-1})\le (\varphi^n)$, and so $\jj_{\varphi_i}(\varphi)\to 0$. By Lemma~\ref{lem:ijequiv}, this yields $\ii(\varphi_i, \varphi)\to 0$, and therefore $\bar\ii(\varphi_i, \varphi)\to 0$.
\end{proof} 
For later use, we also record the convergence of volumes along increasing nets:
\begin{cor}\label{cor:volincreas}
    Let $(\varphi_i)$ be an increasing net in $\cE(X)$ converging weakly to $\varphi\in \cE(X)$. Then $\vol(\varphi_i)\to \vol(\varphi)$.
\end{cor}
\begin{proof}
    By assumption we have $0<\lim_i\vol(\varphi_i)=\sup_i\vol(\f_i)\le \vol(\varphi)$. Thus $(\log\vol(\f_i))$ is a Cauchy net. Since $(\f_i)$ is increasing, a direct computation shows for $i\le j$, $\bar\ii(\f_i, \f_j)\le 2(\log\vol(\f_j)-\log\vol(\f_i))\to 0$. Thus $([\f_i])_i$ is a Cauchy net in the strong topology, and for any $\varepsilon>0$, Proposition~\ref{prop:metrictop} implies that for $i,j\gg0$, $e^{-\varepsilon}[\f_j]\le [\f_i]\le e^\varepsilon[\f_j]$. Letting $j\to \infty$, we have $e^{-\varepsilon}[\f]\le [\f_i]\le e^\varepsilon[\f]$ since $[\f_j]\to [\f]$ weakly. Thus $\de_\infty([\f_i], [\f])\to 0$ and in particular $\vol(\f_i)\to \vol(\f)$.    
\end{proof}

In general, strong convergence in $\bar\cE(X)$ does not imply weak convergence on $\cE(X)$. However, we have
\begin{lem}\label{lem:volweakconv}
    If $\varphi_i\le \psi_i$ in $\cE(X)$ with $\psi_i\in \cH(X)$ a decreasing net converging to some $\varphi\in \cE(X)$, and $\vol(\varphi_i)\to \vol(\varphi)$, then $\varphi_i\to \varphi$ weakly.
\end{lem}
\begin{proof}
    Since $\varphi_i\le \psi_i$, by monotonicity we have $\limsup_i(\varphi_i, \varphi^{n-1})\le \lim_i(\psi_i, \varphi^{n-1})= (\varphi^n)$ and $\limsup_i(\varphi, \varphi_i^{n-1})\le \lim_i(\varphi, \psi_i^{n-1})=(\varphi^n)$. Thus if $\vol(\varphi_i)\to \vol(\varphi)$, then we must have $\bar\ii(\varphi_i, \varphi)\to 0$. 
    By Proposition~\ref{prop:currentconv}, for any $\rho\in \cH(X)$, we have
    \[
    \ddc\varphi_i\wedge (\ddc \rho)^{n-2}\to \ddc\varphi\wedge (\ddc \rho)^{n-2},
    \]
    and in particular $(\varphi_i, \rho^{n-1})\to (\varphi, \rho^{n-1})$. Now let $v\in \Xdiv$ be any divisorial valuation. If $v\in X^\beth$, then $\f_i(v)=\f(v)=0$, so suppose $v\in X^\na_\infty$. Pick a large enough integrally closed compactification $\bar X$ of $X$ such that $\bar X\setminus X$ supports an ample divisor $D$ and $v=c\ord_E$ for some $c>0$ and prime divisor $E\subset \bar X\setminus X$. Put $\rho=\varphi_D\in \cH(X)$. After scaling, we may assume $v=v_E$ for $v_E=\rho(\ord_E)^{-1}\ord_E$. Computed on $X^\na_\rho$, $\MA_\rho(\rho)=\sum_F c_F\delta_{v_F}$, where the sum runs through all irreducible components in $\bar X\setminus X$, and $c_F>0$. Thus,
    \begin{align*}
    0\le c_E( \psi_i(v)-\varphi_i(v))&\le \int_{X^\na_\rho} (\psi_i-\varphi_i)\MA_\rho(\rho)
    =\int_{X^\na_\rho} (\psi_i-\varphi)\MA_\rho(\rho)+\int_{X^\na_\rho} (\varphi-\varphi_i)\MA_\rho(\rho)\\
    &=(\psi_i, \rho^{n-1})-(\varphi, \rho^{n-1})+(\varphi, \rho^{n-1})-(\varphi_i, \rho^{n-1})\to 0.
    \end{align*}
    This implies $\varphi_i(v)\to \varphi(v)$, and we are done by Theorem~\ref{thm:weaktop}.

\end{proof}

\section{Big curve classes at infinity}\label{sec:curvinf}
In this section, $n\ge2$. We introduce the space $\cC_\infty(X)\subset \BPF_1(X)$ of big and BPF curve classes at infinity, together with a natural complete metric. The main goal is to show that the Monge--Amp\`ere operator induces a dense Lipschitz embedding of $\bar\cE(X)$ into $\cC_\infty(X)$, and that this is surjective when $X$ is smooth.

\subsection{Volumes of curve classes}
Recall that $\BPF_1(X)\subset \Curv(X)$ is the space of BPF Weil b-curve classes, and that any element $\mu\in \BPF_1(X)$ defines a positive Radon measure on $X^\na_\rho$ for any $\rho\in \cH(X)$.
Hence we get a canonical map $\BPF_1(X)\to C^0(X^\na_\rho)^\vee$ which is continuous with respect to the weak topology. The image is a convex cone, and a locally compact subset of $C^0(X^\na_\rho)^\vee$. As observed in~\S\ref{sec:numclc}, this map is not injective. However, we will see that it is injective on a subset of big curve classes at infinity, see Corollary~\ref{cor:curvinj}. 
Motivated by Definition~\ref{def:volhat}, we define:
\begin{defi}\label{def:volinfhat}
    The \emph{volume at infinity} of $\mu\in \BPF_1(X)$ is  
    \[
    \vol^\vee_\infty(\mu)\=\inf_{\varphi\in \cH(X)}\frac{(\varphi, \mu)}{\vol(\varphi)^{1/n}}\in [0, +\infty).
    \]
\end{defi}
By definition, $\vol^\vee(\mu)\le \vol^\vee_\infty(\mu)$ for any $\mu\in \BPF_1(X)$, and the following example shows that we have strict inequality in general.
\begin{exam}
    Recall the example in~\S\ref{sec:numclc} where $\mu=2\pi^*H-E$. Then $\vol^\vee(\mu)=\vol(2\pi^*H-E)^{1/2}=\sqrt{3}$, whereas $\vol^\vee_\infty(\mu)=\vol(2\pi^*H)^{1/2}=2.$
\end{exam}
\begin{defi}
    We say $\mu\in \BPF_1(X)$ is a \emph{big curve class at infinity} if 
    \[
    \vol^\vee_\infty(\mu)=\vol^\vee(\mu)>0.
    \]
\end{defi}

We denote by $\cC_\infty(X)\subset\BBPF_1(X)$ the space of big b-curve classes at infinity. 
We have the following refinement of Lemma~\ref{lem:bpfdual}.
\begin{prop}\label{prop:pshlegdual}
    The volume at infinity of a b-curve class satisfies the following:
    \begin{itemize}
        \item [(i)] for $\mu\in\BPF_1(X)$, we have
        \[
        \vol^\vee_\infty(\mu)=\inf_{\varphi\in \cE(X)}\frac{(\varphi, \mu)}{\vol(\varphi)^{1/n}} 
        \]
        \item [(ii)] $\vol^\vee_\infty$ is concave, usc and homogeneous;
        \item[(iii)] for $\varphi\in\cE(X)$, we have $\MA(\varphi)\in \cC_\infty(X)$, and $\vol^\vee_\infty(\MA(\varphi))=\vol(\varphi)^{\frac{n-1}{n}}$ with $\varphi$ achieving the infimum;
        \item[(iv)] for $\varphi\in\cE(X)$, we have 
        \[
        \vol(\varphi)^{1/n}=\inf_{\mu\in \cC_\infty(X)} \frac{(\varphi, \mu)}{\vol^\vee_\infty(\mu)},
        \]
        with $\mu=\MA(\varphi)$ achieving the infimum.
    \end{itemize}
\end{prop}
\begin{proof}
    It is clear that $\vol^\vee_\infty(\mu)\ge \inf_{\varphi\in \cE(X)}\frac{(\varphi, \mu)}{\vol(\varphi)^{1/n}}$. On the other hand, for any $\varphi\in\cE(X)$, let $(\varphi_j)_j$ be a decreasing net in $\cH(X)$ converging to $\varphi$. Then
    \[
    \frac{(\varphi, \mu)}{\vol(\varphi)^{1/n}}\ge \frac{(\varphi_j, \mu)}{\vol(\varphi_j)^{1/n}}+\frac{(\varphi-\varphi_j, \mu)}{\vol(\varphi)^{1/n}}\ge \inf_{\varphi\in \cH(X)}\frac{(\varphi, \mu)}{\vol(\varphi)^{1/n}}+\frac{(\varphi-\varphi_j, \mu)}{\vol(\varphi)^{1/n}}.
    \]
    We obtain the reverse inequality by taking $j\to \infty$ and applying Theorem~\ref{thm:pshint}. This proves the first equality. 
    
    The assertions in~(ii) is clear from the definition, since $\mu\mapsto \frac{(\varphi, \mu)}{\vol(\varphi)^{1/n}}$ is affine and continuous for all $\f\in\cH(X)$.
    The assertion~(iii) is a direct consequence of the Khovanskii--Teissier inequalities. Indeed, for any big and nef $\alpha\in \CNum(X)$, we have
    \[
    (\alpha, \varphi^{n-1})\ge \vol(\alpha)^{1/n}\vol(\varphi)^{\frac{n-1}{n}},
    \]
    and so
    \[
    \vol(\varphi)^{\frac{n-1}{n}}\le \vol^\vee(\MA(\varphi))\le \vol^\vee_\infty(\MA(\f))\le \frac{(\varphi, \varphi^{n-1})}{\vol(\varphi)^{1/n}}=\vol(\varphi)^{\frac{n-1}{n}}.
    \]
    Finally, for any $\varphi\in \cE(X)$ and $\mu\in\cC_\infty(X)$, we have
    $\vol(\varphi)^{1/n}\le \frac{(\varphi, \mu)}{\vol^\vee_\infty(\mu)},$ with equality if $\mu=\MA(\varphi)$. Taking inf over $\mu\in \cC_\infty(X)$, we get~(iv).
\end{proof}

\subsection{Strong topology}
We use the restriction of the $\de_\infty$ metric from \S\ref{sec:strongcurve}.

\begin{defi}
    The \emph{strong topology} on $\cC_\infty(X)$ is the metric topology induced by $\de_\infty$.
\end{defi}
It is clear that if a net $(\mu_i)$ in $\cC_\infty(X)$ converges strongly to $\mu\in \cC_\infty(X)$, then $\mu_i\to \mu$ weakly and $\vol^\vee_\infty(\mu_i)\to \vol^\vee_\infty(\mu)$. Moreover, we have
\begin{lem}\label{lem:Cinfclosed}
    The space $\cC_\infty(X)$ is closed in $\BBPF_1(X)$ in the strong topology.
\end{lem}
\begin{proof}
    Assume $\mu_i\to \mu$ strongly in $\BBPF_1(X)$ with $\mu_i\in \cC_\infty(X)$, we need to show $\mu\in \cC_\infty(X)$. Since $\de_\infty(\mu_i, \mu)\to 0$, we have
    \[
    \vol^\vee_\infty(\mu)\le \liminf_is(\mu_i, \mu)^{-1}\vol^\vee(\mu_i)\le \limsup_i\vol^\vee_\infty(\mu_i)\le\vol^\vee_\infty(\mu),
    \]
    where the last inequality follows from upper semicontinuity of $\vol^\vee_\infty$. This shows $\mu\in \cC_\infty(X)$ since $\vol^\vee_\infty(\mu)=\lim_i\vol^\vee_\infty(\mu_i)=\lim_i\vol^\vee(\mu_i)=\vol^\vee(\mu)$, where the last equality follows from~\eqref{equ:volhatLip}.
\end{proof}
As a direct consequence of Proposition~\ref{prop:curvcomplete}, we get the first part of Theorem~A'.
\begin{cor}\label{cor:curvinfcomplete}
    The space $(\cC_\infty(X), \de_\infty)$ is complete.
\end{cor}
As a restriction of Theorem~\ref{thm:metricequivnum}, we get
\begin{prop}\label{prop:metricequiv}
    The map $\bar\cE(X)\ni[\f]\mapsto\MA(\f)\in\cC_\infty(X)$ is a Lipschitz continuous homeomorphism onto its image, with uniformly locally H\"older continuous inverse. More precisely, for any 
    $\varphi, \psi\in\cE(X)$, we have 
    \begin{equation}\label{equ:MAcont1}
    \de_\infty(\MA(\f),\MA(\p))\le(n-1)\de_\infty([\varphi], [\psi]),
    \end{equation}
    and 
    \begin{equation}\label{equ:MAcont2}
    \de_\infty([\f], [\p])\le F\left(\tfrac{3n-2}{n-1}\de_\infty(\MA(\f),\MA(\p))\right),
    \end{equation}
    for the function $F$ as in Proposition~\ref{prop:metrictop}.
\end{prop}
    
In~\S\ref{sec:MAeqn} we will show that the Monge--Amp\`ere operator has dense image, and in fact is surjective when $X$ is smooth.

\subsection{Uniform differentiability}
Through Legendre duality, for $[\varphi]\in \bar\cE(X)\cup\{0\}$ and $\alpha\in \CNum(X)$,
we define the \emph{volume at infinity} as
\[
\vol_\infty([\varphi]+\alpha)^{1/n}\=\inf_{\mu\in\cC_\infty(X)}\frac{((\varphi, \mu)+(\alpha, \mu))^+}{\vol^\vee_\infty(\mu)},
\]
where $(\cdot)^+\=\max\{\cdot, 0\}$.
\begin{lem}\label{lem:volinfty}
    The volume $\vol_\infty$ satisfies the following properties:
    \begin{itemize}
        \item [(i)] for $\alpha\in \CNum(X)$, $\vol_\infty(\alpha)\ge \vol(\alpha)$;
        \item [(ii)] for fixed $\alpha\in\CNum(X)$, $\varphi\mapsto \vol_\infty([\varphi]+\alpha)$ is usc, and hence continuous along decreasing nets in $\cE(X)$.
    \end{itemize}
\end{lem}
\begin{proof}
    Assume first that $\alpha$ is a big and nef Cartier b-divisor class. By definition of $\vol^\vee$ we have
    \[
    \vol(\alpha)^{1/n}\le \inf_{\mu\in \cC_\infty(X)}\frac{(\alpha, \mu)}{\vol^\vee(\mu)}=\inf_{\mu\in \cC_\infty(X)}\frac{(\a, \mu)}{\vol^\vee_\infty(\mu)}=\vol_\infty(\alpha)^{1/n},
    \]
    and thus 
    \[
    \vol_\infty(\alpha)\ge \vol(\alpha),
    \]
    for any $\a\in \CNum(X)$, by Fujita approximation. This proves~(i).
    For~(ii), note that $\varphi\mapsto (\varphi, \mu)$ is usc by Theorem~\ref{thm:pshint}, and hence so is $\varphi\mapsto((\varphi, \mu)+(\alpha, \mu))^+$. Taking infimum preserves upper-semicontinuity, and~(ii) follows since $\varphi\mapsto\vol_\infty([\varphi]+\alpha)$ is also increasing.
\end{proof}

With this at hand, we are ready to state the following uniform differentiability of the above extended volume. 
\begin{thm}\label{thm:unifderiv}
    Let $\rho\in \cH(X)$, and $\alpha\in\CNum(X)$. For any $\varphi\in \cE(X)$ with $[\varphi]\le[\rho]$, we have 
    \[
    \vol_\infty([\varphi]+t\alpha)\ge\vol(\varphi)+tn(\alpha, \varphi^{n-1})+O(t^2)
    \]
    for any $0\le t\le1$ 
    and the constant in $O(t^2)$ only depends on $\rho$ and $\alpha$.
\end{thm}
\begin{proof}
    First assume $\varphi\in \cH(X)$. In this case, by the previous discussion we have $\vol_\infty([\varphi]+t\a)\ge \vol([\varphi]+t\a)$ and the result follows from~\cite[Corollary~3.4]{BFJ09}. More precisely, pick a sufficiently ample class $\omega\in \CNum(X)$ such that $2[\rho]\le \omega$ and $\omega\pm\alpha$ is nef. Then 
    \[
    \vol([\f]+t\alpha)\ge \vol(\f)+tn(\alpha, \f^{n-1})+O(t^2),
    \]
    with the constant in $O(t^2)$ depends only on $(\omega^n)$, and hence on $\rho$ and $\alpha$.

    In general, pick a decreasing net $(\varphi_i)_i$ in $\cH(X)$ converging to $\varphi$. By Theorem~\ref{thm:pshint}, and Proposition~\ref{prop:MAop}, we have $\vol(\varphi_i)\to \vol(\varphi), \MA(\varphi_i)\to \MA(\varphi)$. Further, Lemma~\ref{lem:volinfty}~(ii) yields $\vol_\infty([\varphi_i]+t\a)\to \vol_\infty([\varphi]+t\a)$, and for large $i$, we have $[\varphi_i]\le2[\rho]$, see Example~\ref{exam:decstrong}. We are done by the previous case.
\end{proof}
\subsection{Minimizing nets} 
In order to prove that the Monge--Amp\`ere operator has dense image, we introduce the following concept.
\begin{defi}
    A net $(\varphi_i)$ in $\cE(X)$ is called a \emph{minimizing net} for $\mu\in \cC_\infty(X)$ if 
    \[\frac{(\varphi_i, \mu)}{\vol(\varphi_i)^{1/n}}\to \vol^\vee_\infty(\mu)\]
    and $\vol(\varphi_i)^{\frac{n-1}{n}}\to \vol^\vee_\infty(\mu)$.
\end{defi}
Such minimizing nets always exist, since $\varphi\mapsto\frac{(\varphi, \mu)}{\vol(\varphi)^{1/n}}$ is homogeneous of degree zero.
We have the following variational characterization of the Monge--Amp\`ere operator, see also~\cite[Theorem~C]{DangFavre}.

\begin{cor}\label{cor:MAsol}
    Let $\mu\in\cC_\infty(X)$ and $(\varphi_i)$ be a minimizing net for $\mu$ in $\cE(X)$.
    Then
    \[
    \MA(\varphi_i)\to \mu
    \]
    weakly as $i\to \infty$.
\end{cor}
\begin{proof}
    We will prove
    \[
    \frac{\MA(\varphi_i)}{\vol(\varphi_i)^{\frac{n-1}{n}}}\to \frac{\mu}{\vol^\vee_\infty(\mu)},
    \]
    from which the result follows since $\vol(\varphi_i)^{\frac{n-1}{n}}\to \vol^\vee_\infty(\mu)$.
    Let $\alpha\in\CNum(X)$. Since both sides are homogeneous of degree zero, we may
    pick $\rho\in \cH(X)$ and up to scaling assume $(\varphi_i, \rho^{n-1})=1$ for all $\varphi_i$.
    First note that by Corollary~\ref{cor:siu}, we have $\frac{(\f_i^n)}{n(\rho, \varphi_i^{n-1})}[\rho]\le [\varphi_i]\le \frac{n}{(\rho^n)}[\rho]$. Moreover, Lemma~\ref{lem:curvbig} implies that $(\varphi_i, \mu)\ge c(\varphi_i, \rho^{n-1})=c$ for some $c>0$, and
    therefore $\vol(\varphi_i)$ is uniformly bounded below, which in turn implies that $\jj_{\varphi_i}(\rho)\approx \jj_\rho(\varphi_i)$ is uniformly bounded. In particular, we may assume there are constants $c_1, c_2>0$ such that $c_1[\rho]\le [\varphi_i]\le c_2[\rho]$.   
    Thus we can choose $0< t\ll1$ such that $[\varphi_i]+t\alpha\ge 0$ for all $i$, and apply Theorem~\ref{thm:unifderiv} to get
    \[
    n(\alpha, \varphi_i^{n-1})+O(t)\le \frac 1t\left(\vol_\infty([\varphi_i]+t\alpha)-\vol(\varphi_i)\right),
    \]
    where the constant in $O(t)$ only depends on $\rho$ and $\alpha$. On the other hand, set $\delta_i\=\frac{(\varphi_i, \mu)}{\vol(\varphi_i)^{1/n}\vol^\vee_\infty(\mu)}$. By assumption, we have $1\le \delta_i\to 1$ as $i\to \infty$. We also have
    \[
    \vol_\infty([\varphi_i]+t\alpha)\le \frac{((\varphi_i, \mu)+t(\alpha, \mu))^n}{\vol^\vee_\infty(\mu)^n}
    \le \delta_i^n\vol(\varphi_i)+\frac{nt(\alpha,\mu)(\varphi_i, \mu)^{n-1}}{\vol^\vee_\infty(\mu)^n}+O(t^2),
    \]
    where $O(t^2)$ is controlled by $(\alpha, \mu)^k(\varphi_i, \mu)^{n-k}$ for $k=2,\dots, n$, which is in turn eventually bounded by a constant only depending on $\mu$, $\rho$, and $\alpha$ via Siu's inequality. Thus for each $i$, putting the above two inequalities together we have
    \[
    \frac{ n(\alpha, \varphi_i^{n-1})}{\vol(\varphi_i)^{\frac{n-1}{n}}}+O(t)\le t^{-1}(\delta_i^n-1)\vol(\varphi_i)^{1/n}+n\delta_i^{n-1}\frac{(\alpha, \mu)}{\vol^\vee_\infty(\mu)}+O(t).
    \]
    Letting $i\to \infty$ and then $t\to 0$ we get $\limsup_i\frac{(\alpha, \MA(\varphi_i))}{\vol(\varphi_i)^{\frac{n-1}{n}}}\le \frac{(\alpha, \mu)}{\vol^\vee_\infty(\mu)}.$ Replacing $\alpha$ by $-\a$ we can conclude.
\end{proof}

\begin{cor}\label{cor:MAsol2}
    Let $\mu\in \cC_\infty(X)$ and suppose $\varphi\in \cE(X)$ attains the infimum defining $\vol^\vee_\infty(\mu)$. Then $\MA(\varphi)=c\mu$ for $c=\frac{\vol(\varphi)^{\frac{n-1}{n}}}{\vol^\vee_\infty(\mu)}$.
\end{cor}
As a direct consequence, we have
\begin{cor}\label{cor:curvinj}
    The map $\cC_\infty(X)\to C^0(X^\na_\rho)^\vee$ is injective, for any $\rho\in \cH(X)$.
\end{cor}
\begin{proof}
    If $\mu=\mu'$ as measures on $C^0(X^\na_\rho)$, then $(\varphi, \mu)=(\varphi, \mu')$ for all $\varphi\in \cH(X)$ and in particular $\vol^\vee_\infty(\mu)=\vol^\vee_\infty(\mu')$. Thus any minimizing net $(\varphi_i)$ of $\mu$ is also a minimizing net of $\mu'$ and it follows that $\mu=\mu'$ by Corollary~\ref{cor:MAsol}.
\end{proof}

\begin{lem}\label{lem:minnetbdd}
    Let $(\varphi_i)$ be a minimizing net in $\cE(X)$ for $\mu\in \cC_\infty(X)$, and let $\rho\in \cH(X)$. Then $(\varphi_i, \rho^{n-1})$ is eventually uniformly bounded.
\end{lem}
\begin{proof}
    Lemma~\ref{lem:curvbig} yields $c>0$ such that $\mu\ge c\MA(\rho)$. This gives $0\le(\f_i,\rho^{n-1})\le c^{-1}(\f_i,\mu)$.
    As $(\f_i)_i$ is a minimizing net for $\mu$, we also have $(\f_i,\mu)\to\vol^\vee_\infty(\mu)^{\frac{n}{n-1}}$. The result follows.
\end{proof}
\begin{lem}\label{lem:MAsoluniq}
    Let $(\varphi_i), (\psi_i)$ be two minimizing nets in $\cE(X)$ for $\mu\in \cC_\infty(X)$. Then we have $\lim_{i,j}\bar\ii(\varphi_i, \psi_j)=0$.
\end{lem}
\begin{proof}
    By definition of minimizing nets we have $\lim_i\vol(\varphi_i)=\lim_j\vol(\psi_j)$. Thus it suffices to show $\ii(\varphi_i, \psi_j)\to 0$, or equivalently, by Lemma~\ref{lem:ijequiv}, that $\jj_{\varphi_i}(\psi_j)\to 0$ as $i,j\to \infty$. To this end, fix $0<t<1$ and let $\varphi_{t,ij}\=(1-t)\varphi_i+t\psi_j\in\cE(X)$. By concavity of $\f\mapsto\vol(\f)^{1/n}$, we have
    \[
    \frac{(\varphi_{t,ij}, \mu)}{\vol(\varphi_{t,ij})^{1/n}}\le \frac{(1-t)(\varphi_i, \mu)+t(\psi_j, \mu)}{(1-t)\vol(\varphi_i)^{1/n}+t\vol(\psi_j)^{1/n}}\to \vol^\vee(\mu).
    \]
    This implies $(\varphi_{t,ij})$ is also a minimizing net, and that 
    \[
    1\le\frac{\vol(\varphi_{t,ij})^{1/n}}{(1-t)\vol(\varphi_i)^{1/n}+t\vol(\psi_j)^{1/n}}\to1.
    \]
    By concavity of $k\mapsto \log(\varphi_i^k, \psi_j^{n-k})$, we must have in particular $\jj_{\varphi_i}(\psi_j)\to 0$.
\end{proof}

\begin{prop}\label{prop:numtrivlim}
    Let $(\varphi_i), (\psi_i)$ be two minimizing nets in $\cE(X)$ for some $\mu\in \cC_\infty(X)$. Then, after passing to a subnet, $\varphi_i-\psi_i$ converges pointwise to some function $\tau$ on $\Xdiv$ which defines a numerically trivial Weil b-divisor at infinity.
\end{prop}
\begin{proof}
    By Tychonoff's theorem, after passing to a subnet, we may assume $\varphi_i-\psi_i$ converges pointwise to a homogeneous function $\tau\colon X^\div_\infty\to[-\infty,+\infty]$. 
    Now, for each $E\in \PbDiv_\infty(X)$, Lemma~\ref{lem:minnetbdd} implies that $(\varphi_i-\psi_i)(\ord_E)$ is eventually uniformly bounded, so $\tau(\ord_E)\in\R$.
    We can equivalently view $\tau$ as a real-valued function on the set $\PbDiv_\infty(X)$, and hence as a real Weil b-divisor $D_\tau\in\WbDiv_\infty(X)$ at infinity, which then induces a numerical Weil class $[\tau]\in\WNum_{n-1}(X)$.

    We must show that $[\tau]=0$. To this end, let $s_i\=s(\varphi_i, \psi_i)$. Then $\vol(\varphi_i)\ge s_i^n\vol(\psi_i)$ and so $\limsup_is_i\le 1$ by assumption. On the other hand, Lemma~\ref{lem:MAsoluniq} together with the Diskant inequality yields $\liminf_is_i\ge 1$. By symmetry we may find $\varepsilon_i>0$ with $\lim_i\varepsilon_i=0$ such that
    \[
    -\varepsilon_i[\varphi_i]\le [\varphi_i]-[\psi_i]\le \varepsilon_i[\varphi_i].
    \]
    It only remains to show that $[\f_i]$ is eventually uniformly bounded; this will show $[\tau]=\lim_i[\varphi_i-\psi_i]=0$, see Lemma~\ref{lem:numclasscts}. But if we pick any $\rho\in\cH(X)$, then Siu's inequality gives $0\le[\f_i]\le\frac{n(\f_i,\rho^{n-1})}{\vol(\rho)}[\rho]$, and by Lemma~\ref{lem:minnetbdd}, $(\f_i,\rho^{n-1})$ is eventually uniformly bounded.
\end{proof}

\subsection{The Monge--Amp\`ere equation}\label{sec:MAeqn}  
We can now improve upon Proposition~\ref{prop:metricequiv} and finish the proof of Theorem~A'.
\begin{thm}\label{thm:metricMAsol}
    The Monge--Amp\`ere operator defines a Lipschitz continuous embedding
    \[
    \MA\colon (\bar\cE(X), \de_\infty)\hookrightarrow (\cC_\infty(X), \de_\infty)
    \]
    with dense image. If $X$ is smooth, then it is further surjective.
\end{thm}
\begin{proof}
    That $\MA$ is a Lipschitz continuous embedding follows from Proposition~\ref{prop:metricequiv}. It remains to prove density. For this, let $\mu\in \cC_\infty(X)$ and pick a minimizing net $(\varphi_i)$ in $\cE(X)$ for $\mu$.
    By Lemma~\ref{lem:MAsoluniq}, $([\varphi_i])_i$ is a Cauchy net for $(\bar\cE(X), \de_\infty)$ and hence $(\mu_i\=\MA(\varphi_i))$ is also Cauchy in $\cC_\infty(X)$. Now Corollary~\ref{cor:curvinfcomplete} yields a limit $\nu\in \cC_\infty(X)$ with $\de_\infty(\mu_i, \nu)\to 0$, and in particular $\mu_i\to \nu$ weakly. We also have $\mu_i\to \mu$ weakly by Corollary~\ref{cor:MAsol}, and thus $\mu=\nu$ and $\de_\infty(\mu_i, \mu)\to 0$. 
    
    Assume now $X$ is smooth, and let $(\varphi_i)$ be a minimizing net in $\cE(X)$ for $\mu$. For each $E\in \PbDiv_\infty(X)$, $(\varphi_i(\ord_E))$ is eventually uniformly bounded, since $\int_{X^\na_\rho} \varphi_i\MA_\rho(\rho)$ is eventually uniformly bounded for any $\rho$ represented by an ample divisor on a large enough integrally closed compactification realizing $E$, by Lemma~\ref{lem:minnetbdd}. 
    
    By Tychonoff's theorem applied to $[0,+\infty]^{\PbDiv_\infty(X)}$, after passing to a subnet, we may assume that $\varphi_i(\ord_E)$ converges in $[0,+\infty]$ for every $E\in\PbDiv_\infty(X)$. For each fixed $E$, the eventual uniform bound established above implies that this limit is finite. Extending these limits by homogeneity gives a function $\f\colon\Xdiv_\infty\to\R_+$, and hence an effective Weil b-divisor $D_\varphi$ at infinity; see~\S\ref{sec:Weilb}.

    Further, weak convergence yields that the numerical class $[D_\varphi]\in\WNum_{n-1}(X)$ is nef. By Proposition~\ref{prop:nefweilpsh}, $\varphi$ extends to a unique function $\varphi\in\PSH(X)$, and we have $\varphi\in\cE(X)$ since $\vol(\varphi)\ge \lim_i\vol(\varphi_i)>0$. In particular, $\varphi_i\to \varphi$ weakly, and so $[\varphi_i]\to [\varphi]$ weakly by Lemma~\ref{lem:numclasscts}. We claim that in fact $[\f_i]\to [\f]$ strongly. Indeed, since $([\f_i])$ is a Cauchy net in the strong topology, for any $\delta>0$, $e^{-\d}[\f_j]\le [\f_i]\le e^\d[\f_j]$, for $i, j\gg0$. Taking $j\to \infty$ we have $\de_\infty([\f_i], [\f])\le \d$ and hence $\de_\infty([\f_i], [\f])\to 0$. This gives $\MA(\f_i)\to \MA(\f)$ by Proposition~\ref{prop:currentconv}.
    Thus $\MA(\varphi)=\mu$ since $\MA(\varphi_i)\to \mu$ weakly.
    This completes the proof.
\end{proof}

\section{Norms}\label{sec:norm}
In this section, we study negative norms on $X$ and their relation to polarized compactifications, and prove the main part of Theorem~B.
\subsection{Negative norms and filtrations}
By a \emph{norm} (resp.\ \emph{seminorm}) on $X$ (or $R$) we mean a function $\chi\colon R\to\R\cup\{+\infty\}$ such that $\exp(-\chi)$ is a submultiplicative norm (resp.\ seminorm) on the $k$-algebra $R$ in the usual multiplicative sense. Concretely, a function $\chi\colon R\to\R\cup\{+\infty\}$ is a seminorm iff $\chi|_{k^\times}\equiv0$, $\chi(0)=\infty$, $\chi(f+g)\ge\min\{\chi(f),\chi(g)\}$, and $\chi(fg)\ge\chi(f)+\chi(g)$. Further, $\chi$ is a norm iff $\chi^{-1}(+\infty)=\{0\}$.

We will mainly consider (semi)norms that are \emph{nonpositive}, that is, $\chi(f)\le0$ for all $f\ne0$. Setting $\cF^\la R\coloneqq\{\chi\ge-\la\}$ shows that nonpositive norms are in 1-1 correspondence with \emph{filtrations} $(\cF^\la R)_{\la\ge0}$ of $R$ with $k\subset \cF^0R$ and $\cF^\la R=0$ for $\la<0$, 
that are
\begin{itemize}
    \item increasing: $\cF^\mu R\subset\cF^\la R$ for $0\le\mu\le\la$;
    \item right-continuous: $\cF^\la R=\cF^{\la+}R\coloneqq\bigcap_{\mu>\la}\cF^\mu R$ for $\la\ge0$;
    \item multiplicative: $f\in\cF^\la R$, $g\in\cF^\mu R$ $\Rightarrow$ $fg\in\cF^{\la+\mu}R$;
    \item exhaustive: $\bigcup_\la\cF^\la R=R$.
\end{itemize}
The norm $\chi$ associated to such a filtration is given by $\chi(f)\coloneqq-\inf\{\la\ge0\mid f\in\cF^\la R\}$ for $f\ne 0$ and $\chi(0)=\infty$.
\begin{rmk}
    The convention differs from~\cite{nakstab1} where decreasing filtrations were used.
\end{rmk}
We say that a norm $\chi$ is \emph{negative} if $\chi(f)<0$ for all $f\in R\setminus k$. Negative norms always exist, as shown shortly.
\begin{defi}
    We write $\cN(R)\coloneqq\cN_{<0}(R)$ and $\cN_{\le0}(R)$ for the spaces of negative (resp.\ nonpositive) norms on $R$.
\end{defi}
These spaces have the following algebraic structure. First, if $\chi_1,\dots,\chi_N$ are nonpositive (resp.\ negative) norms, then so is $\chi:=\chi_1\wedge\dots\wedge\chi_N$, defined by $\chi(f):=\min_i\chi_i(f)$. Second, the multiplicative group $\R_+^\times$ acts on $\cN(R)$ and $\cN_{\le0}(R)$ by multiplication: if $\chi$ is a negative (resp.\ nonpositive) norm, so is $t\chi$ for any $t>0$.

The \emph{reduction} of a negative norm $\chi$ is the $\R_+$-graded $k$-algebra 
\[
\gr_\chi R\coloneqq\bigoplus_{\la\in\R_+}\cF^\la R/\cF^{\la-}R,
\]
where $\cF^{\la-}_\chi R\coloneqq\bigcup_{\mu<\la}\cF^\mu_\chi R$.

If $X\=\Spec R\to X'\=\Spec R'$ is a closed immersion, induced by a surjective $k$-algebra homomorphism
$\tau\colon R'\to R$, then any nonpositive norm $\chi'$ on $X'$ induces a nonpositive \emph{residue norm} $\chi$ on $X$, defined by $\chi(f)\coloneqq\sup\{\chi'(f')\mid \tau(f')=f\}$ for any $f\in R$.

\medskip

\begin{defi}
    A (semi)valuation $v\in X^\na$ is a \emph{negative valuation}
    if $v(f)<0$ for \emph{all} $f\in R\setminus k$.
\end{defi}
\begin{rmk}
    In general, there may not be any negative valuations. This is the case for $X=\G_m^n$, for example. 
    However, we will see shortly that $X$ admits negative norms of the form $\chi=\chi_{v_1}\wedge\dots\wedge\chi_{v_N}$, where $v_i\in X^\val_\infty$. 
\end{rmk}

\subsection{Geometric norms}
Consider a strictly effective Cartier divisor $D$ at infinity on a compactification $\bar{X}$ of $X$.
If $f\in R\setminus\{0\}$ and $m\ge0$, then $fs_D^m$ is a rational section of $\cO_{\bar{X}}(mD)$, and a regular section for $m\gg1$. Here $s_D$ is the canonical section of $\cO_{\bar{X}}(D)$. We can therefore define a negative norm $\chi_{\bar{X},D}$ on $R$ by
\[
\chi_{\bar{X},D}(f):=-\min\{m\in\Z_{\ge0}\mid fs_D^m\in\Hnot(\bar{X},mD)\}.
\]
We set $\chi_{\bar X, D}(0)=\infty$.
It is clear that if $D_1,D_2$ are two strictly effective divisors at infinity on the same compactification $\bar{X}$ of $X$, with $D_1\le D_2$, then $\chi_{\bar X, D_1}\le\chi_{\bar X, D_2}$. 
More generally, for a strictly effective $\Q$-Cartier $\Q$-divisor $D\in \CDiv_\infty(\bar X)_\Q$, pick $m\gg0$ minimal so that $mD$ is a Cartier divisor, and we write $\chi_{\bar X, D}\=m\chi_{\bar X, mD}$.
\begin{defi}
    A \emph{geometric norm} is a norm of the form $\chi_{\bar X, D}$ for a compactification $\bar X$ and a strictly effective semiample divisor $D\in\CDiv_\infty(\bar X)$.
\end{defi}
\begin{lem}\label{lem:gnpull}
    Let $D$ be a strictly effective Cartier divisor at infinity on a compactification $\bar{X}$ of $X$. Let $\rho\colon\bar{X}'\to\bar{X}$ be a morphism of compactifications, and set $D'\coloneqq\rho^*D$. Then $\chi_{\bar{X},D}\le\chi_{\bar{X}',D'}$, with equality if $\bar{X}$ is integrally closed.
\end{lem}
\begin{proof}
    If $f\in R\setminus k$ and $\chi_{\bar{X},D}(f)=-m$, then $fs_D^m\in\Hnot(\bar{X},mD)$. This implies $fs_{D'}^m\in\Hnot(\bar{X'},mD')$, so that $\chi_{\bar{X}',D'}(f)\ge-m$.

    Now suppose $\bar{X}$ is integrally closed. By Lemma~\ref{lem:Zmain} we have $\rho_*\cO_{\bar{X'}}=\cO_{\bar{X}}$. Together with the projection formula, this yields $\Hnot(\bar{X},mD)=\Hnot(\bar{X}',mD')$ for all $m\in\Z_{>0}$, which implies $\chi_{\bar{X},D}=\chi_{\bar{X}',D'}$.
\end{proof}
This lemma allows us to associate a norm $\chi_D$ to any strictly effective divisor $D\in\CbDiv_\infty(X)$: just set $\chi_D\coloneqq\chi_{\bar{X},D}$ for any integrally closed compactification $\bar{X}$ of $X$ where $D$ ``lives".

\medskip

\subsection{Bounded negative norms}
Next we make the notion of a negative norm more quantitative. First, let us say that a negative norm is \emph{weakly bounded} (away from zero) if $\dim_k\cF^\la_\chi R<\infty$ for all $\la>0$. For example, any geometric norm is weakly bounded. We will work with a stronger notion, motivated by the following result.
\begin{lem}\label{lem:geomnormbd}
    If $\chi'\colon R\to\R\cup\{+\infty\}$ is any seminorm and $\chi$ is a geometric norm, then $\chi'\ge C\chi$ for some $C>0$.
\end{lem}
\begin{proof}
    Write $\chi=\chi_{\bar{X},D}$ as above. Using Lemmas~\ref{lem:hasample} and~\ref{lem:gnpull} we may, after replacing $\chi$ by a multiple and a dominating geometric norm, assume that $D$ is ample and that the section ring $R(\bar{X},D)$ is generated in degree 1. Pick a basis $s_1,\dots,s_N$ of $\Hnot(\bar{X},D)$.
    Then $g_j\coloneqq s_j/s_D\in R$. Pick $r>0$ such that $\chi'(g_j)\ge-r$ for all $j$. 
    Now consider any $f\in R\setminus k$, and set $m:=-\chi_{\bar{X},D}(f)\ge1$.    
    Write $fs_D^m$ as a polynomial $\sum_{\a\in A}c_\a s^\a$, where $A\subset\Z_{\ge0}^N$ is a nonempty subset, and where $|\a|=m$ and $c_\a\in k^\times$ for all $\a\in A$. Then $f=\sum_\a c_\a g^\a$, so $\chi'(f)\ge\min_\a\chi'(g^\a)\ge-rm$. We can therefore pick $C=r$.
\end{proof}
\begin{defi}\label{defi:T}
    Given a negative norm $\chi$ and any seminorm $\chi'$ on $R$, we define
    \[
    T(\chi; \chi')=\sup_{f\in R\setminus k}\frac{\chi'(f)}{\chi(f)}=\inf\{C>0\mid C\chi\leq \chi'\}\in[0,\infty].
    \]
\end{defi}
Thus $T(\chi;\chi')<\infty$ when $\chi$ is geometric. Note that if $\chi, \chi', \chi''$ are negative norms, then
\[
T(\chi;\chi'')\le T(\chi;\chi')T(\chi';\chi'').
\]
\begin{defi}
  A negative norm $\chi$ is \emph{bounded} if there exists $C\ge1$ such that for some (equivalently, any) geometric norm $\chi'$ we have 
  \begin{equation}\label{equ:bddnorm}
      C\chi'\le\chi\le C^{-1}\chi'.
  \end{equation}
  We write $\cN_{\mathrm{bdd}}(X)\subset\cN(X)$ for the space of bounded negative norms. 
\end{defi}
Thus $\chi$ is bounded iff $T(\chi,\chi')<\infty$ for some (equivalently, any) geometric norm $\chi'$. 
It follows that if $\chi$ and $\chi'$ are any bounded negative norms, then $T(\chi;\chi'),T(\chi';\chi)<\infty$.
\begin{defi}\label{defi:dinfty}
    For any two bounded negative norms $\chi,\chi'$ on $R$, we set
    \[
    \de_\infty(\chi, \chi')\coloneqq \log \inf\{C\geq 1\mid C\chi\leq \chi'\leq C^{-1}\chi\}=\log\max\{T(\chi;\chi'),T(\chi';\chi)\}.
    \]
\end{defi}
It is clear that $\de_\infty$ defines a metric on the space $\cN_{\mathrm{bdd}}$ of bounded negative norms, and one easily checks that the metric space $(\cN_{\mathrm{bdd}}(X),\de_\infty)$ is complete.

If $\chi$ and $\chi'$ are bounded negative norms, then so is $\chi\wedge\chi'$, as is $t\chi$ for any $t\in\R_+^\times$.
Every bounded norm is weakly bounded. However, not every negative norm is bounded or even weakly bounded.
\begin{exam}\label{exam:negnotbdd}
    Pick a function $h\colon\Z_{\ge0}\to\R_+$ such that $h(0)=0$, $h(j)>0$ for $j>0$, and $h(j+k)\le h(j)+h(k)$ for all $j,k$. Setting $\chi_h(\sum_ja_jz^j)\coloneqq\min\{-h(j)\mid a_j\ne0\}$ defines a negative norm $\chi_h$ on $R=k[z]$. Then $\chi_h$ is weakly bounded iff $\lim_{j\to\infty}h(j)=\infty$, and $\chi_h$ is bounded iff there exists $\d>0$ such that $h(j)\ge\d j$ for all $j$. Picking $h(j)=\min\{j,1\}$ thus gives a norm that is not weakly bounded, whereas if $h(j)=\sqrt{j}$, then $\chi_h$ is weakly bounded but not bounded.
\end{exam}
\begin{exam}
  The norm $\chi_v$ associated to a negative valuation $v$ is negative but not necessarily bounded. See Example~\ref{exam:infsing}.
\end{exam}
\begin{rmk}
    Using the analysis in~\cite{Mon17}, one can construct divisorial valuations on $X=\A^2$ that are negative but not negative bounded.
\end{rmk}
One of the natural situations where bounded negative norms arise comes from taking the residue norm of a geometric norm. 
\begin{lem}\label{lem:embnormequiv}
    Let $i: X=\Spec R\hookrightarrow Y=\Spec S$ be an embedding of affine varieties, and $(\bar Y, E)$ be a polarized compactification with $E$ Cartier. Let $\bar X$ be the closure of $X$ in $\bar Y$, $D\=i^* E$ and $\chi$ be the residue norm of $\chi_{\bar Y, E}$ on $R$. Then $\chi\le \chi_{\bar X, D}$ and $\cF^m_\chi R=\cF^m_{\chi_{\bar X, D}}R$ for $m\gg0$ sufficiently divisible.
\end{lem}
\begin{proof}
    By construction, we have a natural restriction map for each $m\in \Z_{>0}$,
    \[
    \cF^m_{\chi_{\bar Y, E}} S=\Hnot(\bar Y, mE)\to \Hnot(\bar X, mD)=\cF^m_{\chi_{\bar X, D}}R,
    \]
    which is surjective for $m\gg0$. Denote by $\pi\colon S\to R$ the surjection inducing $i$. Then if $\pi(g)=f$ with $g\in \cF^m_{\chi_{\bar Y, E}}S$, then restriction yields $\chi_{\bar X, D}(f)\ge -m$ and hence $\chi_{\bar Y, E}(g)\le \chi_{\bar X, D}(\pi(g))$. This implies $\chi_{\bar X, D}(f)\ge \chi(f)$, by definition of the residue norm. Moreover, for $m\gg0$, any $f\in \cF^m_{\chi_{\bar X, D}}R$ lifts to some $g\in \cF^m_{\chi_{\bar Y, E}}S$, and thus $f\in \cF^m_\chi R$ as desired.
\end{proof}

We say two bounded negative $\Z$-valued norms $\chi, \chi'$ are \emph{equivalent}, written as $\chi\sim\chi'$, if $\cF^\la_\chi R=\cF^\la_{\chi'}R$  for all $\la\gg0$ sufficiently divisible. The above lemma then says $\chi\sim \chi_{\bar X, D}$. As an immediate corollary, we have
\begin{cor}\label{cor:resbnn}
    If $\chi$ is a bounded negative norm on $R$, and $R\to S$ is a surjection, then the residue norm on $S$ is also a bounded negative norm.
\end{cor}
Let $\bar{X}$ be a compactification of $X$. Following~\cite[Definition~2.9]{BKMS}, we say a valuation $v\in\bar{X}^\val$ has \emph{linear growth} if, for any big line bundle $L$ on $\bar{X}$, there exists $C>0$ such that $v(s)\le Cm$ for $s\in\Hnot(\bar{X},mL)\setminus\{0\}$. By~\cite{BKMS}, this condition does not depend on the choice of big line bundle $L$, nor on the choice of compactification $\bar{X}$; note that $\bar{X}^\val\simeq X^\val$. This defines a subset $X^\lin\subset X^\val$ of valuations of linear growth. We also write $X^\lin_\infty\coloneqq X^\lin\cap X^\na_\infty$.
\begin{lem}\label{lem:linbdd}
    A valuation $v\in X^\val$ has linear growth iff for some (equivalently, any) bounded negative norm $\chi$ on $X$, there exists $C>0$ such that $v\le-C\chi$.
\end{lem}
\begin{proof}
    Let $\bar{X}$ be a compactification of $X$, and $D$ an ample strictly effective divisor at infinity on $\bar{X}$. Then $v(s)=v(s/s_D^m)+mv(D)$ for $m\in\Z_{>0}$ and $s\in\Hnot(\bar X,mD)\setminus\{0\}$. This implies $v\le-C\chi_{\bar{X},D}$ on $R$ iff $v(s)\le (C+v(D))m$ for $m\in\Z_{>0}$ and $s\in\Hnot(\bar X,mD)\setminus\{0\}$. The result follows.
\end{proof}
\begin{rmk}\label{rmk:linbdd}
Any negative valuation has linear growth, but a valuation of linear growth is not necessarily negative. Any Abhyankar valuation (in particular, any divisorial valuation) has linear growth, see \cite[Proposition~2.12]{BKMS}. A valuation $v$ has linear growth iff it is a nonpluripolar point, see~\cite[Proposition~4.38]{BJpp}.
\end{rmk}
\subsection{Base ideals}\label{sec:baseideals}
Let $\chi$ be a bounded negative norm and $(\bar X, D)$ a polarized compactification.  Let $\chi'\=\chi_{\bar X, D}$. If we pick $r\ge1$ sufficiently divisible, then $\chi\leq r^{-1}\chi'$ and $rD$ is Cartier, and we get natural inclusions $\cF_\chi^m R\subseteq H^0(\bar X, mrD)$ for all $m\in \N$. 
This defines a graded linear series of $R(\bar X, rD)$, which in turn gives base ideals
\[
\fb_m(\chi)\coloneqq \fb(|\cF_\chi^m R|).
\]
The base ideals depend on the ambient divisor $rD$; we suppress this dependence when it is fixed. It is clear that $\fb_\bullet(\chi)\coloneqq (\fb_m(\chi))_m$ is a graded sequence of ideals at infinity.

For $v\in X^\na$, the value $v(\fb_m(\chi))$ is the
minimum of the values of local generators at the center
of $v$. The sequence $(v(\fb_m(\chi)))_m$ is subadditive, so 
\[
v(\fb_\bullet(\chi))
:=\inf_{m\ge1}\frac{v(\fb_m(\chi))}{m}
=\lim_{m\to\infty}\frac{v(\fb_m(\chi))}{m}.
\]
\subsection{Common orthogonal bases}
As is well known, any two non-Archimedean norms on a finite-dimensional vector space over a trivially valued field admit a common orthogonal basis, see~\cite[Proposition~1.14]{BE21} or~\cite[Lemma~3.1]{AZ22}. The following result is an infinite-dimensional version in our setting.
\begin{lem}\label{lem:commorth}
    Let $\chi_0$ be a bounded negative norm, and $\chi$ be a bounded norm on $R$. Then there exists a $k$-vector space basis $(f_i)_0^\infty$ of $R$, with $f_0=1$,  that is orthogonal for both $\chi_0$ and $\chi$. More precisely, we have $\lim_{i\to\infty}\chi_0(f_i)=-\infty$, and for any nonempty finite subset $I\subset\Z_{\ge0}$ and $a\in(k^\times)^I$, we have 
    \[
    \chi_0(\sum_{i\in I}a_if_i)=\min_{i\in I}\chi_0(f_i)
    \quad\text{and}\quad
    \chi(\sum_{i\in I}a_if_i)=\min_{i\in I}\chi(f_i).
    \]
\end{lem}
\begin{proof}
    Our assumptions imply that $\Phi_{\chi_0}=-\chi_0(R\setminus 0)$ is a discrete subset of $\R_{\ge 0}$ of the form $\{\la_j\}_{j=0}^\infty$, where $0=\la_0<\la_1<\la_2<\dots$ and $\lim_{j\to\infty}\la_j=\infty$. Set $V_j\coloneqq\cF^{\la_j}_{\chi_0} R$ and $N_j\coloneqq\dim_kV_j$. We have $V_0=k$, and if $f_0\coloneqq1$, then $\{f_0\}$ is orthogonal for the restrictions of $\chi_0$ and $\chi$ to $V_0$. Inductively, suppose $j\ge0$ and that $\{f_i\}_{i=0}^{N_j-1}$ is a basis for $V_j$ orthogonal for both $\chi_0$ and $\chi$. Any extension of this basis to a basis of $V_{j+1}$ will be orthogonal for $\chi_0$, so it suffices to show that we can pick the extended basis orthogonal for $\chi$ as well. 
    
    To this end, consider the residue norm $\chi_{j+1}$ induced by $\chi$ on $V_{j+1}/V_j$. Pick a basis $\bar{f}_{N_j},\dots,\bar{f}_{N_{j+1}-1}$ of $V_{j+1}/V_j$ that is orthogonal for $\chi_{j+1}$, and lift each $\bar{f}_i$ to an element of $f_i\in V_{j+1}$ with $\chi(f_i)=\chi_{j+1}(f_i)$. We claim that the basis $\{f_i\}_{i<N_{j+1}}$ for $V_{j+1}$ is orthogonal for $\chi$. It suffices to check that $t:=\chi(\sum_{i<N_{j+1}}a_if_i)=\min_i\chi(a_if_i)=:s$ for any nonzero $a\in k^{N_{j+1}}$. Clearly $s\le t$. If $\chi(a_if_i)>s$ for all $i\ge N_j$, then $\chi(\sum_{i\ge N_j}a_if_i)>s$ and $\chi(\sum_{i<N_j}a_if_i)=\min_{i<N_j}\chi(a_if_i)=s$, so $t=s$ in this case. Now suppose $\min_{i\ge N_j}\chi(a_if_i)=s$. Then 
    \[
    t=\chi(\sum_0^{N_{j+1}-1}a_if_i)
    \le\chi_{j+1}(\sum_{i\ge N_j}a_i\bar{f}_i)
    =\min_{i\ge N_j}\chi_{j+1}(a_i\bar{f}_i)
    =\min_{i\ge N_j}\chi(a_if_i)
    =s,
    \]
    and we are done.
\end{proof}

\subsection{Homogeneous norms}
A norm $\chi$ is \emph{homogeneous} if $\chi(f^m)=m\chi(f)$ for all $f\in R$ and all $m\ge1$. For example, any valuation $v$ induces a homogeneous norm $\chi_v$, and we shall later see that any bounded negative norm can be reconstructed from such norms. 
\begin{defi}
    The \emph{homogenization} of a nonpositive norm $\chi$ is given by 
    \[
    \chi^\hom(f)=\sup_d\frac 1d \chi(f^d)=\lim_{d\to \infty} \frac 1d \chi(f^d).
    \]
\end{defi}
It is easy to see that a norm is homogeneous iff it equals its homogenization. Moreover, if $\chi'$ is a nonpositive homogeneous norm with $\chi'\ge\chi$, then $\chi'^{\hom}\ge\chi^\hom$.

The homogenization of a negative norm may fail to be negative; the homogenization of the norm in Example~\ref{exam:negnotbdd} is the zero norm. However, we will see in Corollary~\ref{cor:hombdd}
that the homogenization of a bounded negative norm is a bounded negative norm.

Further, equivalent $\Z$-valued bounded negative norms have the same homogenization:
\begin{lem}\label{lem:equivhom}
    Let $\chi\sim\chi'$ be two $\Z$-valued bounded negative norms. Then $\chi^\hom=\chi'^\hom$.
\end{lem}
\begin{proof}
    Assume there is $l\in \Z_{>0}$ such that $\cF^{ml}_\chi R=\cF^{ml}_{\chi'}R$ for all $m\in \Z_{>0}$.
    Pick $f\in R\setminus k$, and $d\gg 0$ and $m=m(d)$ with $-ml\le \chi(f^d)<-(m-1)l$. Since $\chi\sim\chi'$, we have $-ml\le \chi'(f^d)<-(m-1)l$. Thus
    \begin{align*}
    |\chi^\hom(f)-\chi'^\hom(f)|&\le |\chi^\hom(f)-\frac{\chi(f^d)}{d}|+|\frac{\chi(f^d)}{d}-\frac{\chi'(f^d)}{d}|
    +|\frac{\chi'(f^d)}{d}-\chi'^\hom(f)|\\
    &\le |\chi^\hom(f)-\frac{\chi(f^d)}{d}|+\frac ld+|\frac{\chi'(f^d)}{d}-\chi'^\hom(f)|\to 0
    \end{align*}
    as $d\to \infty$. This shows $\chi^\hom=\chi'^\hom$.
\end{proof}
\subsection{Norms of finite type}\label{sec:ftnorm}
Let $\chi$ be a negative norm on $R$. 
We denote by $\Phi_\chi\subset\R_+$ the subsemigroup of nonnegative values in the group generated by $-\chi(R\setminus 0)$. 
To $\chi$ we associate a \emph{Rees algebra}; this is the $\Phi_\chi$-graded $k[t^{\Phi_\chi}]$-algebra given by 
 \[
 \cR_\chi:=\bigoplus_{\la\in\Phi_\chi}t^\la\cF_\chi^\la R.
 \]
When $\chi$ is further weakly bounded, each graded piece in $\cR_\chi$ is finite dimensional.
\begin{lem}
    For any bounded negative norm $\chi$ on $R$, the Rees algebra $R_\chi$ is a faithfully flat $k[t^{\Phi_\chi}]$-algebra.
\end{lem}
\begin{proof}
    By Lemma~\ref{lem:commorth} there exists a $\chi$-orthogonal basis $1=f_0,f_1,\dots$ of $R$ as a $k$-vector space. Thus $\cR_\chi$ is free as a $k[t^{\Phi_\chi}]$-module, with basis $\{t^{-\chi(f_i)}f_i\}_0^\infty$, and the result follows. 
\end{proof}
The \emph{graded residue ring} of $\chi$ is the $\Phi_\chi$-graded $k$-algebra defined by 
\[
\gr_\chi R\coloneqq \bigoplus_{\la\in\Phi_\chi} \frac{\cF^\la_\chi R}{\cF^{<\la}_\chi R}.
\]
Note that $\chi$ is homogeneous (resp.\ a valuation) iff $\gr_\chi R$ is reduced (resp.\ an integral domain). 

For a weakly bounded negative norm $\chi$, as in~\cite{nakstab1}, there is a canonical approximation as we now describe.
For $\la\gg0$, the canonical $k$-algebra homomorphism $\Sym\cF_\chi^\la R\to R$ is surjective. The restriction of $\chi$ to $\cF^\la_\chi R$ then induces a canonical norm on $\Sym\cF^\la_\chi R$, which in turn induces a residue norm $\chi^{(\la)}$ on $R$.
Concretely, let $1, f_1,\dots,f_N$ be an orthogonal basis for the normed $k$-vector space $\cF^\la_\chi R$, so that $\chi(\sum_{i=1}^Na_if_i)=\min\{\chi(f_i)\mid a_i\ne0\}$ for $a\in k^N$. If $g\in R$, then we have 
\[
\chi^{(\la)}(g)=\sup\min_{\a\in A}\sum_{i=1}^N\a_i\chi(f_i),
\]
where the supremum is taken over all representations $g=\sum_{\a\in A}a_\a f^\a$, where $A\subset\Z_{\ge0}^N$ is a finite subset, and $a_\a\ne0$ for all $\a\in A$. As $\chi(f_i)<0$ for all $i>0$, the supremum must be attained. From the construction, we see
\begin{lem}
    If $0<\mu\le\la$, then $\chi^{(\mu)}\le\chi^{(\la)}\le\chi$.
    Further, we have $T(\chi^{(\la)}; \chi')\to T(\chi; \chi')$ for all seminorms $\chi'$ on $R$.
\end{lem}
\begin{proof}
    The first assertion is clear. Hence $\lim_\la T(\chi^{(\la)};\chi')=\sup_\la T(\chi^{(\la)};\chi')\in [0,\infty]$. Now 
    \[
    \sup_\la T(\chi^{(\la)};\chi') = \sup_\la\sup_f \frac{\chi'(f)}{\chi^{(\la)}(f)} = \sup_f\sup_\la\frac{\chi'(f)}{\chi^{(\la)}(f)}=\sup_f\frac{\chi'(f)}{\chi(f)} = T(\chi; \chi').
    \]
     Since $T\ge0$, terms with $\chi'(f)>0$ do not affect the supremum, so we may restrict to $\chi'(f)\le0$.
\end{proof}

The following class of norms is of crucial importance.
\begin{prop}\label{prop:ft}
    The following conditions on a negative norm $\chi$ on $R$ are equivalent:
    \begin{itemize}
        \item[(i)]
        there exists a surjection $k[z_1,\dots,z_N]\to R$, such that $\chi$ is the residue norm induced by a negative monomial valuation on $k[z_1,\dots,z_N]$;
    \item[(ii)]
         $R_\chi$ is finitely generated $k[t^{\Phi_\chi}]$-algebra;
    \item[(iii)]
        $\gr_\chi R$ is a finitely generated $k$-algebra.
    \item[(iv)] $\chi$ is weakly bounded and $\chi=\chi^{(\la)}$ for $\la\gg0$. 
    \end{itemize}
    Moreover, when these conditions are satisfied, $\chi$ is a bounded negative norm, and
    $\gr_\chi R$ has the same Krull dimension as $R$.
\end{prop}

\begin{defi}\label{defi:ft}
    A negative norm $\chi$ is \emph{of finite type} if the equivalent conditions above are satisfied.
\end{defi}
\begin{defi}
    Given a nonzero subgroup $\Lambda\subset\R$, we write $\cT_\Lambda$ for the set of negative norms $\chi$ of finite type with $\chi(R\setminus 0)\subset\Lambda$. We also write $\cT_\Lambda^\hom\subset\cT_\Lambda$ for the subset of homogeneous norms.
\end{defi}
\begin{rmk}
We mainly have in mind the cases $\La\in\{\Z,\Q,\R\}$. More generally, if $\La\subset\R_{<0}$ is any subset, one could define $\cT_\La$ as the set of norms of finite type with values in $\La$.
\end{rmk}
\begin{proof}[Proof of Proposition~\ref{prop:ft}]
    We first note that if $\chi$ satisfies~(i), then $\chi$ is a bounded negative norm (and in particular weakly bounded), as follows from Corollary~\ref{cor:resbnn}. Moreover, we may assume $f_i:=\pi(z_i)\not\in k$ for all $i$. As the numbers $t_i:=-v(z_i)$ generate a discrete subsemigroup of $\R_+$, we also see that every nonzero $g\in R$ can be written as $g=\sum_{\a\in A} a_\a f^\a$, where $A\subset\Z_+^N$ is a finite subset, $a_\a\ne0$ for all $\a$, and $\chi(g)=\min_\a\sum_i\a_iv(z_i)$. Using this, it is easy to see that we may also assume $\chi(f_i)=v(z_i)$ for all $i$. 
    
    \smallskip
    We will prove the implications (i)$\Rightarrow$(ii)$\Rightarrow$(iii)$\Rightarrow$(iv)$\Rightarrow$(i).
    First suppose~(i) holds, and consider a surjective $k$-algebra homomorphism $\pi\colon k[z_1,\dots,z_N]\to R$ such that $\chi$ is the residue norm from a negative monomial valuation $v$. Set $f_i:=\pi(z_i)$. As explained above, we may assume $\chi(f_i)=v(z_i)<0$ for all $i$. Write $\la_i:=-\chi(f_i)>0$. Then $\{t^{\la_i}f_i\}$ generate $R_\chi$ over $k[t^{{\Phi_\chi}}]$. Indeed, given $\mu\in\Phi_\chi$ and $g\in \cF^\mu_\chi R$, there is a presentation $g=\sum_\alpha a_\alpha\prod f_i^{\alpha_i}$ such that $\chi(g)=\min_\alpha\sum\alpha_i\chi(f_i)\ge -\mu$, and thus 
    $t^\mu g=\sum_\alpha a_\a t^{\mu-\sum_i\a_i\la_i} \prod (t^{\la_i}f_i)^{\alpha_i}$. 

    \smallskip
    Assuming~(ii) we have~(iii) through the identification $R_\chi/(t^{\Phi_\chi\setminus\{0\}})R_\chi\simeq \gr_\chi R$.   

  \smallskip
  To prove (iii)$\Rightarrow$(iv), suppose $\gr_\chi R$ is finitely generated, say by homogeneous generators $\bar f_0,\dots,\bar f_N$, where $f_j\in R$. We may assume $f_0=1$. 
  If we set $\la_i\coloneqq-\chi(f_i)$, then $\la_0=0$ and $\la_i>0$ for $1\le i\le N$, so the subsemigroup $\Gamma\subset\R_+$ generated by the $\la_i$ is a discrete subset of $\R_+$ containing $-\chi(R\setminus 0)$. 
  We claim that 
  \begin{equation}\label{equ:filt}
  \cF_\chi^t R=\mathrm{span}_k\{f^\a\mid \sum_i\a_i\la_i\le t\}
  \end{equation}
  for any $t>0$. Indeed, pick any nonzero $g\in\cF_\chi^t R$. Then we can write $\bar g = \sum_\alpha a_\a\prod_i \bar f_i^{\alpha_i}$, where $a_\a\ne0$ and $\sum_i\a_i\la_i=-\chi(g)$ for all $\a$. If $g=h\coloneqq\sum_\alpha a_\a \prod_i f_i^{\alpha_i}$, then we are done. Otherwise, we repeat the procedure with $g'\coloneqq g-h$. We have $-\chi(g')\in[0,-\chi(g))\cap\Gamma$. As $\Gamma$ is a discrete subset of $\R_+$, the process terminates after finitely many steps. 

  Now~\eqref{equ:filt} immediately implies that $\chi$ is weakly bounded. Pick $\la:=\max_i\la_i$, so that $f_i\in\cF_\chi^\la R$ for all $i$. We claim that $\chi^{(\la)}=\chi$. It suffices to prove $\chi^{(\la)}(g)\ge\chi(g)$ whenever $\chi(g)\le-\la$. But applying~\eqref{equ:filt} with $t=-\chi(g)$, we find a representation $g=\sum a_\a f^\a$ such that $a_\a\ne0$ and $\sum_i\a_i\la_i\le t$ for all $\a$, and this implies $\chi^{(\la)}(g)\ge\chi(g)$.

\smallskip
The final implication (iv)$\Rightarrow$(i) can be proved as follows. Assume $\chi$ is weakly bounded and that there exists $\la>0$ such that $\cF^\la_\chi R$ generates $R$, and $\chi^{(\la)}=\chi$. Let $f_0, \dots, f_N$ be an orthogonal basis for $\cF^\la_\chi R$, ordered such that if $\la_i=-\chi(f_i)$, then $\la_0\le\dots\la_N$. Then $\la_0=0$ whereas $\la_i>0$ for $i>0$, and after scaling we have $f_0=1$. Let $v$ be the monomial valuation on $S\=k[z_1, \dots, z_N]$ with $v(z_i)=-\la_i$. Then the $k$-algebra homomorphism $\pi\colon S\to R$ given by $\pi(z_i)=f_i$ is surjective, and $\chi$ is the residue norm induced by $v$. This proves~(iv)$\Rightarrow$(i).

\medskip
Assuming~(i)--(iv), let us finally prove that $\gr_\chi R$ has the same Krull dimension as $R$. Set $A:=k[t^{\Phi_\chi}]$ and $K:=\mathrm{Frac}(A)$. We know that the Rees algebra $\cR_\chi$ is a faithfully flat $A$-algebra of finite type. As the filtration is exhaustive, we have
\[
\cR_\chi\otimes_A K\simeq R\otimes_k K,
\]
which is a domain of Krull dimension $\dim R$.

Let $w$ be the monomial valuation on $A$ defined by $w(t^\lambda)=\lambda$ for $\lambda\in\Phi_\chi$. Denote by $R_w\subset K$ the corresponding valuation ring and by $\kappa(w)$ its residue field. Then
\[
\Spec(\cR_\chi\otimes_A R_w)\to\Spec R_w
\]
is of finite type and faithfully flat. Flatness gives an injection
\[
\cR_\chi\otimes_A R_w\hookrightarrow\cR_\chi\otimes_A K,
\]
so $\cR_\chi\otimes_A R_w$ is a domain. By~\cite[Tag~0B2J]{stacks-project}, its special fiber has the same dimension as its generic fiber, which, as we have seen, has dimension $\dim R$. Since the residue map $A\to\kappa(w)$ sends $t^\lambda$ to zero for every $\lambda>0$, the special fiber is
\[
(\cR_\chi\otimes_A R_w)\otimes_{R_w}\kappa(w)
\simeq \gr_\chi R\otimes_k\kappa(w).
\]
As extension of the ground field preserves Krull dimension, we conclude that $\dim\gr_\chi R=\dim R$.
\end{proof}
\begin{cor}\label{cor:fgqm}
    Any negative valuation of finite type is quasimonomial.
\end{cor}
\begin{prop}\label{prop:fthom}
    If $\chi\in \cT_\Lambda$, then $\chi^\hom\in\cT^\hom_{\Q\Lambda}$.
\end{prop}
\begin{proof}
    The same proof as in~\cite[Theorem 2.3]{nakstab1} shows the finite generation property. Indeed, if $\pi\colon k[z]\coloneqq k[z_1,\dots,z_N]\to R$ is the map inducing $\chi$ as the residue  norm of $v_\xi$, then after completion, there is a finite surjection $\hat\pi\colon k\{z;\xi\}\to \hat R$ making $(\hat R, \chi^\hom)$ a reduced $\Lambda$-strict $k$-affinoid algebra and that $\chi^\hom(R\setminus 0)\subset \Q\Lambda$. Same argument shows that there exists $C>0$ with $\chi\le \chi^\hom\le \chi+C$. In particular, if $\chi^\hom(f)=0$ for some $f\in R\setminus k$, then $\chi(f^m)\ge-C$ for all $m\ge 1$, and $\dim_k\cF^C_\chi R=\infty$, which is not possible since $\chi$ is bounded. It follows that $\chi^\hom$ is a negative norm.    
    Moreover, $\gr_{v_\xi}k[z]\cong \gr_{v_\xi}k\{z;\xi\}\to \gr_{\chi^\hom}\hat R\cong\gr_{\chi^\hom}R$ is finite. This shows $\gr_{\chi^\hom}R$ is finitely generated.
\end{proof}
\begin{cor}\label{cor:hombdd}
    If $\chi$ is a bounded negative norm, then so is $\chi^\hom$.
\end{cor}
\begin{proof}
    Pick a polarized compactification $(\bar X, D)$ for $D$ a strictly effective ample Cartier divisor. Since $\chi$ is bounded, there is some $C\ge 1$ such that $\chi\le C^{-1}\chi_{\bar X, D}$ and hence $\chi^\hom\le C^{-1}\chi_{\bar X, D}^\hom$. By Proposition~\ref{prop:fthom}, $\chi_{\bar X, D}^\hom$ is of finite type and in particular bounded. It follows that $\chi^\hom$ is bounded as well.
\end{proof}
\begin{lem}\label{lem:normimm}
    Let $\tau\colon R'\to R$ be a surjective $k$-algebra homomorphism between integral $k$-algebras of finite type. If $\chi'$ is a negative norm on $R'$ of finite type, then the residue norm $\chi$ on $R$ induced by $\tau$ is also of finite type. 
\end{lem}
\begin{proof}
    First suppose $\chi'$ is of finite type. Thus there exists a surjective $k$-algebra homomorphism $\sigma\colon R''\to R'$, where $R''\coloneqq k[z_1,\dots,z_N]$, and a negative monomial norm $\chi''$ on $R''$ such that $\chi'$ is the residue norm of $\chi''$ via $\sigma$. But then $\chi$ is the residue norm of $\chi''$ via $\tau\circ\sigma$, and hence of finite type. 
\end{proof}

We also record the following relation between finite generation of norms on $X$ and on a polarized compactification of $X$.
\begin{lem}\label{lem:lgnorm}
    Consider a compactification $\bar{X}$ of $X$ and a strictly effective ample divisor $D$ at infinity on $\bar{X}$. Let $\pi\colon R(\bar{X},D)\to R$ be the surjective map given by $\pi(s)=s/s_D^m$ on $R_m=\Hnot(\bar{X},mD)$.
\begin{itemize}
    \item[(i)]
    Suppose $\tilde\rho$ is a $\G_m$-invariant nonpositive linearly bounded norm on $R(\bar{X},D)$ such that $\cF^0_{\tilde\rho}R(\bar{X},D)=k[s_D]$. Then the induced residue seminorm $\rho$ on $R$ via $\pi$ is a nonpositive norm that satisfies
    \[
    \rho(f)=\lim_{m\to\infty}\tilde\rho(fs_D^m)=\sup_m\tilde\rho(fs_D^m),
    \]
    where we only consider $m\ge-\chi_{\bar{X},D}(f)$ so that $fs_D^m\in\Hnot(\bar{X},mD)$.
    \item[(ii)]
    If the norm $\tilde\rho$ in~(i) is of finite type, in the sense that the graded algebra is finitely generated, then the norm $\rho$ on $R$ is a bounded negative norm of finite type.
    \item[(iii)]
    Suppose $\rho$ is a bounded negative norm on $R$, and define $\tilde\rho$ as the unique $\G_m$-invariant norm on $R(\bar{X},D)$ whose restriction to $\Hnot(\bar{X},mD)$ satisfies $\tilde\rho(s)=\rho(s/s_D^m)$. Then $\tilde\rho$ is linearly bounded, nonpositive, and $\cF^0_{\tilde\rho}R=k[s_D]$. Moreover, $\rho$ is the residue norm of $\tilde\rho$ via $\pi$, and $\tilde\rho$ is the largest $\G_m$-invariant linearly bounded nonpositive norm on $R(\bar{X},D)$ whose residue norm under $\pi$ equals $\rho$.
\end{itemize}
\end{lem}
\begin{proof}
    To simplify notation, we set $\bar{R}\coloneqq R(\bar{X},D)$ and $\bar R_m\coloneqq\Hnot(\bar{X},mD)$.
    
    To prove~(i), pick any $f\in R\setminus\{0\}$. As $\pi(fs_D^m)=f$, we have $\rho(f)\ge\tilde\rho(fs_D^m)$ for all $m\ge-\chi_{\bar{X},D}(f)$. Moreover, $\tilde\rho(fs_D^{m+1})\ge\tilde\rho(fs_D^m)$. Now suppose $s\in\bar{R}$ is such that $\pi(s)=f$. Write $s$ as a finite sum $s=\sum_ms_m$, where $s_m\in\bar{R}_m$. We must prove that $\tilde\rho(s)=\min_m\tilde\rho(s_m)$ is bounded above by $\tilde\rho(fs_D^l)$ for $l\gg1$. But $fs_D^l=\sum_ms_ms_D^{l-m}$ so 
    $\tilde\rho(fs_D^l)\ge\min_m\tilde\rho(s_ms_D^{l-m})$ and $\tilde\rho(s_ms_D^{l-m})\ge\tilde\rho(s_m)+(l-m)\tilde\rho(s_D)=\tilde\rho(s_m)$ for all $m$.

    Now suppose $\tilde\rho$ is of finite type. Then $\tilde\rho(\bar{R})$ is a discrete subset of $\R_{\le0}$, so the limit in~(i) must be stationary, and $\rho$ is a negative norm on $R$.
    
    To see that $\rho$ is of finite type, set $A:=\gr_{\tilde\rho}\bar R$ and $u:=\operatorname{in}_{\tilde\rho}(s_D)$. Since $\tilde\rho(fs_D^m)=\rho(f)$ for each fixed $f\in R\setminus\{0\}$ and all sufficiently large $m$, we obtain
    \[
    \gr_\rho R\simeq (A[u^{-1}])_0\simeq A/(u-1),
    \]
    where the subscript refers to the section grading. Since $A$ is finitely generated, so is $\gr_\rho R$. Thus $\rho$ is of finite type.

    Next we turn to~(iii). That $\tilde\rho$ is a well-defined nonpositive norm with $\cF^0_{\tilde\rho}\bar{R}=k[s_D]$ is clear, and that $\tilde\rho$ is linearly bounded follows from $\rho\ge C\chi_{\bar{X},D}$ for some $C\ge1$. To prove that $\rho$ is the residue norm induced by $\tilde\rho$, we must prove that if $f\in R$, then $\tilde\rho(s)\le\rho(f)$ for all $s\in\bar{R}$ with $\pi(s)=f$, and that equality holds for at least one such $s$. On the one hand, if we set $s=fs_D^m$ for $m\gg1$, then $\pi(s)=f$ and $\tilde\rho(s)=\rho(f)$. On the other hand, if $s=\sum_ms_m$ satisfies $\pi(s)=f$, then 
    \[
    \tilde\rho(s)
    =\min_m\tilde\rho(s_m)
    =\min_m\rho(f_m)
    \le\rho(\sum f_m)
    =\rho(f),
    \]
    where $f_m=s_m/s_D^m$. 
    Finally, suppose $\tilde\rho'$ is another linearly bounded nonpositive $\G_m$-invariant norm on $\bar{R}$ with residue norm $\rho$ on $R$. We must prove that $\tilde\rho'(s)\le\tilde\rho(s)$. But by $\G_m$-invariance, it suffices to prove this when $s\in\bar{R}_m$, and then $\tilde\rho'(s)\le\rho(f)=\tilde\rho(s)$.
\end{proof}

\subsection{The Rees correspondence}
We now aim to describe the space $\cT_\Z$ of $\Z$-valued norms of finite type in geometric terms. Recall that two norms $\chi, \chi'\in \cT_\Z$ are equivalent if $\cF^\la_\chi R=\cF^\la_{\chi'}R$  for all $\la\gg0$ sufficiently divisible. The following Rees correspondence gives an identification between $\cP\cC$ and equivalence classes in $\cT_\Z$.
\begin{prop}\label{prop:Reescorr}
    There is a one-to-one correspondence between polarized compactifications of $X=\Spec R$ and equivalence classes of negative $\Z$-valued norms of finite type on $R$.
\end{prop}
\begin{proof}
Let $(\bar X, D)$ be a polarized compactification of $X$, and pick $r\gg0$ divisible enough so that $rD$ is an ample Cartier divisor. Then the norm $\chi\=r\chi_{\bar X, rD}$ is negative, and satisfies $\cF^{mr}_\chi R=\Hnot(\bar X, mrD)$ for all $m\in \Z_{\ge0}$. Thus 
we see from Proposition~\ref{prop:ft} that $\chi$ is of finite type.

Conversely, suppose $\chi$ is a $\Z$-valued negative norm of finite type. Set $R_m:=\cF^m_\chi R$ for $m\in\Z_{\ge0}$. By Proposition~\ref{prop:ft}, the $k$-algebra $\bar{R}:=\bigoplus_mR_m$ is finitely generated, so $\bar{X}:=\Proj(\bar{R})$ is a scheme of finite type over $k$. Moreover, the isomorphism $\bigoplus_m R_m\cong \bigoplus_mR_mt^m\subset R[t, t^{-1}]$ shows that $\bar X$ is a projective variety and $X\hto \bar X$ as a Zariski open set $\{t\neq 0\}\subset \bar X$.
On the other hand, by Proposition~\ref{prop:ft} again, there exists a surjective map $k[z_1,\dots,z_N]\to R$ induced by $\Sym \cF_\chi^dR$ for $d\gg0$ such that $\chi$ is the residue norm with respect to a negative $\Z$-valued monomial valuation $v_\alpha$ on $k[z_1,\dots,z_N]$. This embeds $X\hto \A^N$ and $\bar X\hto \P(1, -v_\alpha(z_1), \dots, -v_\alpha(z_N))$, and $\bar R^{(d)}$ is generated in degree $1$. Let $D\=\bar X\cap H$
where $H$ is the infinity hyperplane on the weighted projective space, we have $\chi\sim \chi_{\bar X, D}$ by Lemma~\ref{lem:embnormequiv}.
\end{proof}

\subsection{Geometric norms and homogenization}
We now study how finitely generated $\Z$-valued negative norms are affected when passing to homogenization and integral closure.
In view of Lemma~\ref{lem:equivhom}, it is enough to consider the effect on geometric norms.
\begin{prop}\label{prop:homic}
    Let $\bar{X}$ be an integrally closed compactification of $X$, and $D=\sum_ib_iE_i$ a strictly effective $\Q$-Cartier $\Q$-divisor on $\bar{X}$ at infinity. If $\chi\coloneqq\chi_{\bar{X},D}$, then:
    \begin{itemize}
        \item[(i)] 
        $\chi^{\hom}=\min_ib_i^{-1}\ord_{E_i}$; 
        \item[(ii)] 
        $\chi\sim\lfloor\chi^\hom\rfloor$, with equality if $D$ is Cartier;
        \item[(iii)] 
        if $D$ is a reduced Cartier divisor, then $\chi$ is homogeneous.
        \end{itemize}
\end{prop}
\begin{proof}
    Pick $r\in \Z_{>0}$ minimal such that $rD$ is a Cartier divisor.
    Pick any $f\in R$. By definition, we have $\chi(f)\ge -mr$ iff $fs_{rD}^m$ is a regular section of $\cO_{\bar X}(mrD)$. By  Lemma~\ref{lem:critnopole} this is equivalent to $\ord_{E_i}(fs_{rD}^m)\ge0$ for all $i$. We can rewrite this as $\chi(f)\ge -mr$ iff $\lfloor\min_ib_i^{-1}\ord_{E_i}(f)\rfloor\ge-mr$, which applied to powers of $f$ yields $\chi^\hom(f)=\min_ib_i^{-1}\ord_{E_i}(f)$, yielding~(i).The assertion~(ii) follows since $\cF^{mr}_\chi R=\cF^{mr}_{\lfloor\chi^\hom\rfloor}R$ for all $m\in \Z_{\ge0}$ and we can take $r=1$ if $D$ is Cartier. 
    When $D$ is reduced,~(i) shows that $\chi^\hom$ is $\Z$-valued, and thus~(ii) implies $\chi=\chi^\hom$, proving~(iii).
\end{proof}
\begin{prop}\label{prop:notic}
Let $\bar X$ be any compactification of $X$, and $D\in \CDiv_\infty(\bar X)_\Q$ be strictly effective and semiample. Let $\bar{X}'$ be the integral closure of $\bar X$ in $X$ and $D'$ the pullback of $D$ to $\bar{X}'$ under the finite morphism $\pi\colon\bar{X}'\to\bar{X}$. Then
\begin{equation}\label{equ:notic}
\chi_{\bar{X}',D'}^{\hom}=\chi_{\bar{X},D}^{\hom}
\quad\text{and}\quad
\chi_{\bar{X}',D'}\sim\lfloor\chi_{\bar{X},D}^{\hom}\rfloor,
\end{equation}
where the equivalence is an equality if $D$ is Cartier.
\end{prop}
In other words, integral closure is the round-down of homogenization.

\begin{proof}
   Set $\chi=\chi_{\bar{X},D}$ and $\chi':=\chi_{\bar{X}',D'}$. It suffices to prove that 
   \[
   \chi\le\chi'\le\chi^{\hom}.
   \]
   Indeed, this yields the first equality in~\eqref{equ:notic}, and the second assertion follows from Proposition~\ref{prop:homic}.
   
   The inequality $\chi\le\chi'$ is trivial. To prove that $\chi'\le\chi^{\hom}$, we may replace $D$ by a multiple and assume $D$ is Cartier.
   Consider the section rings $R(\bar{X},D)=\bigoplus_{m=0}^\infty\Hnot(\bar{X},mD)$ and 
   \[
   R(\bar{X}',D')=\bigoplus_{m=0}^\infty\Hnot(\bar{X}',mD')
    =R(\bar{X},\cO_{\bar{X}}(D)\otimes\pi_*\cO_{\bar{X}'}).
   \]
   Since the map $\pi\colon\bar{X}'\to\bar X$ is finite, $R(\bar{X}',D')$ is a finite $R(\bar{X},D)$-module.

   Now pick any $f\in R$ and set $-m:=\chi'(f)\in\Z_{\le0}$. Then $s:=s_{D'}^mf\in\Hnot(\bar{X'},mD')$. By what precedes, $s$ satisfies a monic equation
   \[
   s^d+\sum_{i=1}^d t_is^{d-i}=0,
   \]
   where $d\ge1$ and $t_i\in\Hnot(\bar{X},imD)$. If we set $g_i:=t_i/s_D^{im}\in R$, then $f^d+\sum_{i=1}^dg_if^{d-i}=0$ and $\chi^{\hom}(g_i)\ge\chi(g_i)\ge-im$. Thus
   \[
   d\chi^{\hom}(f)
   =\chi^{\hom}(f^d)
   \ge\min_i\{-im+(d-i)\chi^{\hom}(f)\},
   \]
   which implies $\chi^{\hom}(f)\ge-m=\chi'(f)$, and we are done.
\end{proof}

\section{Volumes and limit measures}\label{sec:volume}
In this section we define and study volumes and geodesics of bounded negative norms, as well as the pseudometric $\de_1$. The main part of Theorem~D is also proved.
\subsection{The volume of a bounded negative norm}
We define the \emph{volume} of a nonpositive norm $\chi$ on $R$ by 
\[
\vol(\chi)\coloneqq\limsup_{\la\to\infty}\frac{\dim_k\cF_\chi^\la R}{\la^n/n!}\in\R_+\cup\{+\infty\}.
\]
The volume function is clearly increasing: if $\chi\le\chi'$, then $\vol(\chi)\le\vol(\chi')$. It also has the homogeneity property $\vol(t\chi)=t^{-n}\vol(\chi)$ for any $t\in\R_+^\times$.
\begin{prop}\label{prop:volbounded}
    Let $\chi$ be a nonpositive norm on $R$. Then
    $\vol(\chi)>0$, and the limsup defining $\vol(\chi)$
    is a limit, possibly infinite. Moreover,
    $\vol(\chi)<\infty$ iff $\chi$ is a bounded negative norm.
    \footnote{The proof of the forward direction (which we do not use) was provided by ChatGPT.}
\end{prop}
\begin{proof}
    As $\la\mapsto\dim_k\cF^\la_\chi R$ is an increasing
    function, it suffices to consider limits and limsups
    over integers $\la=m\to\infty$.

    Pick an integrally closed polarized compactification
    $(\bar X,D)$ with $D$ a Cartier $\Z$-divisor, and set
    $\chi'=\chi_{\bar X,D}$. Then
    $\cF^m_{\chi'}R=\Hnot(\bar X,mD)$, and, as is
    well-known, see~\cite[Example~11.4.7]{PosII}, the limit
    \[
    \lim_{m\to\infty}\frac{n!}{m^n}
    \dim_k\cF_{\chi'}^mR
    \]
    exists and equals $\vol(D)\in(0,+\infty)$.
    By Lemma~\ref{lem:geomnormbd}, we have
    $\chi\ge C\chi'$ for some $C\ge1$, and hence
    $\vol(\chi)\ge C^{-n}\vol(\chi')>0$.
    If $\chi\in\Nbdd(X)$, then
    $C\chi\le\chi'$ for some $C\ge1$, so
    \[
    \dim_k\cF_\chi^mR
    \le\dim_k\cF_{\chi'}^{Cm}R=O(m^n),
    \]
    yielding $\vol(\chi)<\infty$.

    To complete the proof, it suffices to show that if    
    \[
    \liminf_{m\to\infty}
    \frac{\dim_k\cF_\chi^mR}{m^n}<\infty,
    \]
    then $\chi$ is bounded negative, and the
    limsup defining $\vol(\chi)$ is a limit.
    Under this assumption,
    $\dim_k\cF_\chi^mR<\infty$ for all $m$.

    Pick an integrally closed polarized compactification
    $(\bar X,D)$ with $D$ Cartier, and write
    $D=\sum_{i=1}^r b_iE_i$, where $b_i\in\Z_{>0}$.
    For each $i$, as in~\cite{LM}, choose an admissible flag of the form
    \[
    \bar X\supset E_i\supset Y_1\supset\dots\supset \{\text{pt}\},
    \]
    and denote by
    \[
    \nu_i\colon k(X)^\times\to\Z^n
    \]
    the valuation associated to the flag.
    Consider the semigroup
    \[
    S_i\=\{(\nu_i(f),m)\mid
    m\in\Z_{\ge0},\ 0\ne f\in\cF_\chi^mR\}.
    \]
    By~\cite[Proposition~2.6, Example~2.13]{KK12}, we have
    \[
    \#(S_i)_m=\dim_k\cF_\chi^mR.
    \]
    Moreover, $S_i$ generates $\Z^{n+1}$ as a group. Indeed, suppose $(\a,l)\in\Z^n\times\Z$. We can find nonzero elements $f,g\in R$ such that $\nu_i(f/g)=\a$. For $m\gg0$, we have $\chi(f)\ge-m-l$ and $\chi(g)\ge-m$, so $(\a,l)=(\nu_i(f), m+l)-(\nu_i(g), m)$.

    Set $M:=\{x_{n+1}\ge 0\}\subset\R^{n+1}$. By~\cite[Theorem~1.18]{KK12}, the Okounkov body $\Delta(S_i, M)$ is compact, and in particular bounded. Further, by~\cite[Corollary~1.16]{KK12}, the limit
    \[
    \lim_{m\to\infty}\frac{n!}{m^n}\dim_k\cF_\chi^mR
    \]
    exists and is finite. Thus there exists $C_i>0$ such that
    \[
    \ord_{E_i}(f)\ge -C_i m
    \quad\text{for all }0\ne f\in\cF_\chi^mR, 
    \]
    and so $\ord_{E_i}(f)\ge C_i\lfloor \chi(f)\rfloor$ for all $f\in R\setminus \{0\}$. This gives $\ord_{E_i}(f)\ge\lim_m\frac{C_i}{m}\lfloor \chi(f^m)\rfloor= C_i \chi^\hom(f)\ge C_i\chi(f)$.
    Hence we have
    \[
    \chi_{\bar X, D}\ge \min_i\ord_{E_i}\ge \min_i (C_i\chi)=(\max_iC_i)\chi,
    \]
    so that $\chi$ is negative and bounded, see Lemma~\ref{lem:geomnormbd}.
\end{proof}
\begin{exam}\label{exam:infsing}
    In general, a negative norm can have infinite volume. Let $H\colon\A^2\to\A^2$ be the automorphism given by $(z_1,z_2)\mapsto(z_1^2+z_2,z_1)$, let $v_0$ be the monomial valuation on $R=k[z_1,z_2]$ with weights $(-1,-1)$, and set $v_n\coloneqq2^{-n}H^{n,\na}(v_0)$ for $n\ge1$. As $H$ is an automorphism, $v_n$ is a negative divisorial valuation with $v_n(R\setminus 0)=-2^{-n}\Z_+$. We claim that  $v_n\ge v_{n-1}$ pointwise on $R$ for all $n\ge1$. Indeed, we have $v_1(z_1)=-1$ and $v_1(z_2)=-1/2$, which implies $v_1\ge v_0$ as $v_0$ is monomial. For $n\ge 1$, we now get $H^{n-1,\na}(v_1)\ge H^{n-1,\na}(v_0)$ and hence $v_n\ge v_{n-1}$ on $R$. It follows that the limit $v_\infty:=\lim_{n\to\infty}v_n$ exists, and is a nonpositive valuation on $R$. In fact, in the classification of~\cite{FJEigenvals}, it is an infinitely singular valuation, see~\cite[Example~2.12]{FJDynComp}, and this implies that for every nonzero $f\in R$, the sequence $n\mapsto v_n(f)$ is eventually constant. In particular, $v_\infty(f)<0$ when $f$ is nonconstant, so $v_\infty$ is a negative valuation. Now, as $H$ is an automorphism, we have $\vol(v_n)=4^n\vol(v_0)=4^n$, so $\vol(v_\infty)=\infty$.
\end{exam}

\subsection{Finite type approximation}
Recall from \S\ref{sec:ftnorm} that if $\chi\in \Nbdd(X)$, then there are canonical approximants $\chi^{(\lambda)}$ of finite type.
\begin{prop}\label{prop:canapprox}
    Let $\chi\in \Nbdd(X)$. Then $\lim_{\la\to\infty}\vol(\chi^{(\la)})=\vol(\chi)$.
\end{prop}
\begin{proof}
By construction, $\chi^{(\la)}$ is of finite type. Note that $\chi^{(\la)}=\chi \text{ on } \cF^\la_\chi R$, and by submultiplicativity, it follows that $\chi^{(\la)}\leq \chi$. Since $\chi^{(\la)}\leq \chi^{(\la')}$ for $0\ll\la<\la'$, we have $\vol(\chi^{(\la)})\leq \vol(\chi^{(\la')})\leq \vol(\chi)$ and so $\lim_{\la\to \infty}\vol(\chi^{(\la)})=\sup_\la \vol(\chi^{(\la)})\le \vol(\chi)$. To see that we have equality, it suffices to show $\vol(\chi^{(d)})\to \vol(\chi)$ for $d\in \Z_{>0}$ and $d\to \infty$. For this we fix a geometric norm $\chi_{\bar X, D}$ such that, after possibly scaling $\chi$, we have $\chi\le \chi_{\bar X, D}$. This gives a graded subalgebra $S\=\bigoplus_m\cF^m_\chi R\subset R(\bar X, D)$. Since $\chi^{(d)}\le \chi$, we also have a graded subalgebra $T(d)\=\bigoplus_m\cF^{md}_{\chi^{(d)}}R\subset S^{(d)}$, with $T(d)_1=S^{(d)}_1=S_d$. Thus~\cite[Lemma~3.21]{nakstab1} applies, and we get $\vol(\chi^{(d)})=d^{-n}\vol(T(d))\to \vol(S)=\vol(\chi)$ as desired.
\end{proof}

\subsection{Bivariate limit measures}\label{sec:limmeas}
Next, following~\cite{BLXZ, BLQ} we associate a positive measure on $\R_+^2$ to any two bounded negative norms. 
\begin{prop}\label{prop:Hprop}
    For any bounded negative norms $\chi_i$, $i=1,2$ and $x\in\R_+^2$, the limit
    \[
    H(x)\coloneqq H_\chi(x)\coloneqq\lim_{\la\to\infty}\frac{n!}{\la^n}\dim_k(\cF_{\chi_1}^{\la x_1}R\cap\cF_{\chi_2}^{\la x_2}R)
    \]
    exists and is finite. Moreover, the function $H\colon\R_+^2\to\R_+$ satisfies:
    \begin{itemize}
    \item[(i)] $H(cx)=c^nH(x)$ for $c>0$;
    \item[(ii)] $H$ is continuous, and $H^{1/n}$ is concave;
    \item[(iii)] for each $i=1,2$, the function $x_i\mapsto H(x)$ is increasing; moreover, $x_2\mapsto H(x_1,x_2)$ is constant for $x_2\ge Cx_1$ and $x_1\mapsto H(x_1,x_2)$ is constant for $x_1\ge Cx_2$.
    \end{itemize}
    Here $C$ is any constant such that $C\ge e^{\de_\infty(\chi_1, \chi_2)}$.
\end{prop}
\begin{proof}
    The limit defining $H(x)$ equals zero when $x_i=0$ for some $i$. If $x_1, x_2>0$, then the limit also exists, and 
    $H(x)=\vol(\frac{\chi_1}{x_1}\wedge\frac{\chi_2}{x_2})>0$. Properties~(i) and~(iii) are clear.
    
    It remains to prove~(ii). As $H_{\chi}(x)=H_{c\chi}(cx)=c^nH_{c\chi}(x)$ for $c>0$, we may assume there exists a geometric norm $\rho=\chi_{\bar{X},D}$ such that $\chi_i\le\rho$ for $i=1,2$. If we set $V_{m,x}\coloneqq\cF_{\chi_1}^{mx_1}R\cap\cF_{\chi_2}^{mx_2}R$, then $V_x\coloneqq\bigoplus_mV_{m,x}$ is a graded subalgebra of $R(\bar{X},MD)$ for $M\gg0$. Given $x, y\in \R^2_+$, fix $M\gg0$ such that $V_x, V_y, V_{x+y}$ define graded subalgebras of $R(\bar X, MD)$. 
    Given any choice of a valuation on $R$ with value group $\Z^n$, $V_x$ gives rise to a convex (Okounkov) body $\Delta^x\subset\R^n$ of volume $n!\vol(\Delta^x)=H(x)$, see~\cite{Bou12}. Moreover, $\Delta^x+\Delta^y\subset\Delta^{x+y}$ for $x,y\in\R_+^2$. By the Brunn--Minkowski inequality, we therefore get $H(x+y)^{1/n}\ge H(x)^{1/n}+H(y)^{1/n}$, so that $H^{1/n}$ is concave, being 1-homogeneous.

    The concavity implies that $H$ is continuous on $\R_{>0}^2$. To see continuity on all of $\R_+^2$ we argue as follows. Suppose $(y^{(l)})_l$ is a sequence in $[0,1]^2$ converging to some point $x\in[0,1]^2\setminus(0,1]^2$. First note that $\chi_i\le\rho$ implies $H(x)\le H(1,1)\le\vol(\rho)<\infty$ when $\max\{x_1, x_2\}\le1$. Thus if $x=0$, then $H(y_l)\le(\max_iy^{(l)}_i)^n\vol(\rho)\to0$. Otherwise, by symmetry we may assume $x_1=0$. Then for all $l\gg1$ we have $Cy^{(l)}_1<y_2^{(l)}$ and $H(y^{(l)})=\vol(\frac{\chi_1}{y^{(l)}_1})=(y_1^{(l)})^n\vol(\chi_1)\to 0$ as $l\to \infty$.

\end{proof}
\begin{defi}
    The \emph{bivariate limit measure} $\mu=\mu_{\chi}$ associated to bounded negative norms $\chi_1,\chi_2$ is the locally finite, Borel measure $\mu_\chi$ on $\R_{>0}^2$ defined as the distributional derivative
    \[
    \mu_\chi\coloneqq\frac{\partial^2}{\partial x_1\partial x_2}H_\chi(x),
    \]
    of the function $H_\chi\colon\R_{>0}^2\to\R$.
\end{defi}
\begin{cor}
    The support of the multivariate limit measure $\mu_\chi$ is contained in the cone 
    \[
    \{x\in\R_+^2\mid C^{-1}x_1\le x_2\le Cx_1\},
    \]
    where $C$ is such that $C\ge e^{\de_\infty(\chi_1, \chi_2)}$. Moreover, $\mu_\chi$ has the invariance property $c_*\mu_\chi=c^{-n}\mu_\chi$ for every $c>0$.
\end{cor}

The homogeneity property of $\mu_\chi$ under scaling means that we can work in ``spherical coordinates". Set $\Delta\coloneqq\{x\in\R_{>0}^2\mid x_1+x_2=1\}$, and define $\pi\colon\R_{>0}^2\to\Delta$ by $\pi(x)=x/(x_1+x_2)$.
\begin{defi}
    The \emph{spherical bivariate limit measure} $\tilde\mu_\chi$ on $\Delta$ associated to $\chi$ is the  pushforward under $\pi$ of the restriction of $\mu_\chi$ to the simplex $\{x\in\R_{>0}^2\mid\sum_ix_i<1\}$.
\end{defi}
\begin{cor}
    The spherical bivariate limit measure $\tilde\mu_\chi$ is a finite positive measure whose support is contained in the set $\{x\in\Delta\mid C^{-1}x_1\le x_2\le Cx_1\}$.
\end{cor}
\begin{rmk}\label{rmk:limrddown}
    The function $H_\chi$ does not change if we replace the norms $\chi_i$, $i=1,2$ by their round-downs $\lfloor\chi_i\rfloor$. It follows that the bivariate limit measure $\mu_\chi$ and the spherical bivariate limit measure $\tilde\mu_\chi$ also do not change.
\end{rmk}
\subsection{Cones and geodesics}\label{sec:geod}
Consider two bounded negative norms $\chi_1,\chi_2$. Following~\cite[\S3]{BLXZ} and~\cite{RemiGeodesics} we define, for $t=(t_1,t_2)\in\R_+^2\setminus\{0\}$, a new norm $\chi_t$ in terms of its filtration:
\[
\cF^\la_{\chi_t}R=\sum_{\la=t_1\la_1+t_2\la_2}\cF^{\la_1}_{\chi_1}R\cap\cF^{\la_2}_{\chi_2}R.
\]
\begin{lem}\label{lem:geod}
Let $\chi_1, \chi_2$ be two bounded negative norms. For any $t\in\R_+^2\setminus\{0\}$, $\chi_t$ is a well-defined nonpositive norm. Moreover:
\begin{itemize}
\item[(i)]
    $\chi_{(1,0)}=\chi_1$, $\chi_{(0,1)}=\chi_2$, and $\chi_{at}=a\chi_t$ for $a>0$;
\item [(ii)]
    $\chi_t$ is a bounded negative norm for any $t\ne(0,0)$;
\item[(iii)]
    any common orthogonal 
    basis $\{f_i\}_i$ of $R$ diagonalizing both $\chi_1$ and $\chi_2$ also diagonalizes $\chi_t$, and $\chi_t(f_i)=t_1\chi_1(f_i)+t_2\chi_2(f_i)$;
\item [(iv)]
    if $t_1,t_2>0$, $\chi_t$ is the smallest norm such that 
    \[  
    \chi_t\ge\theta_1\chi_1\wedge\theta_2\chi_2
    \quad\text{for all $\theta_i>0$ such that $\frac{t_1}{\theta_1}+\frac{t_2}{\theta_2}=1$}.
    \]
\end{itemize}
\end{lem}
\begin{proof}
For $t\in \R^2_+\setminus\{0\}$, by construction 
\[
\cF^\lambda_{\chi_t}R=\mathrm{span}\langle f\in R\mid t_1\chi_1(f)+t_2\chi_2(f)\ge -\lambda\rangle.
\]
Using this characterization, it is clear that $\chi_{(1,0)}=\chi_1$ and $\chi_{(0,1)}=\chi_2$. If $a>0$, then 
\[
\cF^\lambda_{\chi_{at}}R=\mathrm{span}\langle f\in R\mid t_1\chi_1(f)+t_2\chi_2(f)\ge -\lambda/a\rangle=\cF^{\la/a}_{\chi_t}R=\cF^\lambda_{a\chi_t}R.
\]
If both $\chi_1, \chi_2$ are negative norms, then $\cF^0_{\chi_t}R=k$ and hence $\chi_t$ is negative as well. Assume further that both $\chi_1, \chi_2$ are bounded. Pick a positive multiple $\rho$ of a geometric norm with $\chi_1, \chi_2\le \rho$. Then $\cF^\la_{\chi_t}R\subseteq \cF^\la_{(t_1+t_2)\rho}R$ and hence $\chi_t\le (t_1+t_2)\rho$ is also bounded. This proves~(ii). The assertion~(iii) follows directly from the definition. For any $\theta_i>0$ with $\frac{t_1}{\theta_1}+\frac{t_2}{\theta_2}=1$, we have
\[
\cF^\la_{\theta_1\chi_1\wedge \theta_2\chi_2}R=\cF^{\la/\theta_1}_{\chi_1}R\cap \cF^{\la/\theta_2}_{\chi_2}R\subset \cF^\la_{\chi_t}R,
\]
where the second inclusion follows from $t_1\frac{\lambda}{\theta_1}+t_2\frac{\la}{\theta_2}=\la$. This shows $\chi_t\ge \theta_1\chi_1\wedge \theta_2\chi_2$. 

Conversely, let $\eta$ be any norm satisfying $\eta\ge\theta_1\chi_1\wedge\theta_2\chi_2$ for all $\theta_i>0$ with $t_1/\theta_1+t_2/\theta_2=1$. Given $f\in R\setminus k$ with $-a:=t_1\chi_1(f)+t_2\chi_2(f)\ge-\lambda$, set $\theta_i:=-a/\chi_i(f)$. Then $a>0$ and
  $t_1/\theta_1+t_2/\theta_2=1$, so $f\in\cF^\lambda_{\theta_1\chi_1\wedge\theta_2\chi_2}R
  \subset\cF^\lambda_\eta R$. This easily implies $\chi_t\le\eta$, proving~(iv).
\end{proof}
\begin{defi}
    The \emph{cone} spanned by two bounded negative norms $\chi_1$, $\chi_2$ is the image of $\R_{>0}^2\ni t\mapsto\chi_t$. The \emph{geodesic} spanned by $\chi_1$ and $\chi_2$ is the image of $[0,1]\ni u\mapsto\chi_{(1-u,u)}$.
\end{defi}

Next we connect the construction above with the limit measure construction in~\S\ref{sec:limmeas}.
\begin{prop}\label{prop:volbymeas}
    Let $\chi_1$ and $\chi_2$ be bounded negative norms, and $\mu_\chi$ the associated limit measure on $\R_+^2$. Then, for any $t\in\R_{>0}^2$, we have
    \begin{itemize}
        \item[(i)] 
        $\vol(\chi_t)=\mu_\chi\{x\in\R_{>0}^2\mid t_1x_1+t_2x_2\le 1\}$;
        \item[(ii)]
        $\vol(\frac{\chi_1}{t_1}\wedge\frac{\chi_2}{t_2})=\mu_\chi\{x\in\R_{>0}^2\mid x_1\le t_1, x_2\le t_2\}$.
    \end{itemize}
    \end{prop}
\begin{proof}
    For any $m\in\Z_{>0}$, define $H_m\colon\R_{>0}^2\to\R$ by
    \[
    H_m(x_1,x_2)\coloneqq\frac{n!}{m^n}\dim_k(\cF_{\chi_1}^{mx_1}R\cap\cF_{\chi_2}^{mx_2}R).
    \]
    We first show that the distributional derivative
    \[
    \mu_m\coloneqq\frac{\partial^2H_m}{\partial x_1\partial x_2}
    \]
    defines a positive, locally finite measure on $\R_{>0}^2$.
    Pick $C\ge1$ such that $C\chi_1\le\chi_2\le C^{-1}\chi_1$. Then $x_1\mapsto H_m(x_1,x_2)$ (resp.\ $x_2\mapsto H_m(x_1,x_2)$) is constant on $x_1>Cx_2$ (resp.\ $x_2>Cx_1$) so the support of $\mu_m$ is contained in the cone $C^{-1}x_1\le x_2\le Cx_1$ of $\R_{>0}^2$. 
    Pick a basis $\{f_i\}_{i=0}^\infty$ of $R$ that diagonalizes both $\chi_1$ and $\chi_2$ as in Lemma~\ref{lem:commorth} with $f_0=1$. A direct calculation as in~\cite[(4.1)]{BLQ} shows
    \begin{equation}\label{equ:discretmeas}
        \mu_{m}=\frac{n!}{m^n}\sum_{i\ge1}\delta_{\tfrac1m(-\chi_1(f_i),-\chi_2(f_i))}.
    \end{equation}
    On any compact subset of $\R_{>0}^2$ the sequence $(H_m)_m$ is uniformly bounded, and converges to $H$ pointwise. By dominated convergence, we therefore have $H_m\to H$ in $L^1_{\mathrm{loc}}$. This implies that $\mu_m\to\mu$ in the sense of distributions, and hence in the weak sense of measures.

    We now prove~(i). By Lemma~\ref{lem:geod}, $\chi_t$ is also diagonalized by the same basis, and therefore $\dim_k\cF_{\chi_t}^mR$ is the number of $i\ge0$ such that $t_1\chi_1(f_i)+t_2\chi_2(f_i)\ge-m$.    
    In view of~\eqref{equ:discretmeas}, this gives
    \begin{align*}
    \vol(\chi_t)
    &=\lim_{m\to\infty}\frac{n!}{m^n}\dim_k\cF_{\chi_t}^mR
    =\lim_{m\to\infty}\left(\mu_{m}\{x\in\R_{>0}^2\mid t_1x_1+t_2x_2\le 1\}+\frac{n!}{m^n}\right)
    \end{align*}
    Thus for any $\varepsilon\in(0,1)$, we have
    \begin{align*}
    \lim_{m\to \infty}\mu_m\{x\in\R_{>0}^2\mid \varepsilon<t_1x_1+t_2x_2\le 1\} 
    &= \vol(\chi_t)-\lim_{m\to \infty}\mu_m\{x\in\R_{>0}^2\mid t_1x_1+t_2x_2\le \varepsilon\}\\
    &=\vol(\chi_t)-\vol(\varepsilon^{-1}\chi_t)=(1-\varepsilon^n)\vol(\chi_t),
    \end{align*}
    where the second equality follows from Lemma~\ref{lem:geod}(i). By weak convergence, this gives
    \[
    \mu\{x\in\R_{>0}^2\mid t_1x_1+t_2x_2\le 1\}\ge\lim_{m\to \infty}\mu_m\{x\in\R_{>0}^2\mid \varepsilon<t_1x_1+t_2x_2\le 1\} =(1-\varepsilon^n)\vol(\chi_t).
    \]
    On the other hand, we have
    \begin{align*}
    \mu\{x\in\R_{>0}^2\mid t_1x_1+t_2x_2\le 1\}
    &\le \mu\{x\in\R_{>0}^2\mid t_1x_1+t_2x_2< 1+\varepsilon\}\\
    &\le \liminf_{m\to \infty}\mu_m\{x\in\R_{>0}^2\mid t_1x_1+t_2x_2< 1+\varepsilon\}\\
    &=(1+\varepsilon)^{n}\vol(\chi_t).
    \end{align*}
    We get~(i) by letting $\varepsilon\to 0$.

    The proof of~(ii) is similar, and based on the fact that the basis $\{f_i\}_i$ of $R$ is orthogonal also for $\frac{\chi_1}{t_1}\wedge\frac{\chi_2}{t_2}$. The details are left to the reader.
\end{proof}
We also have volume formulas in terms of the spherical limit measure $\tilde\mu_\chi$ on the one-dimensional simplex $\Delta=\{x\in\R_{>0}^2\mid x_1+x_2=1\}\simeq(0,1)$.
\begin{cor}\label{cor:volcone}
    Let $\chi_1$ and $\chi_2$ be bounded negative norms, and $\tilde\mu_\chi$ the associated limit measure on $\Delta$. Then, for any $t\in\R_{>0}^2$, we have
    \begin{itemize}
        \item[(i)] 
        $\vol(\chi_t)=\int_\Delta(t_1x_1+t_2x_2)^{-n}\,d\tilde\mu_\chi(x)$;
        \item[(ii)]
        $\vol(\frac{\chi_1}{t_1}\wedge\frac{\chi_2}{t_2})=\int_\Delta\max\{t_1^{-1}x_1,t_2^{-1}x_2\}^{-n}\,d\tilde\mu_\chi(x)$.
    \end{itemize}
\end{cor}
\begin{thm}\label{thm:volan}
    For any two bounded negative norms $\chi_1,\chi_2$, there exists $\e>0$ such that the function $\vol(\chi_t)$ on $\R_+^2\setminus\{(0,0)\}$ extends to a smooth function $f(t)>0$ on the open cone
    \[
    \R_{>0}(1,-\e)+\R_{>0}(-\e,1).
    \] 
    Moreover $f(t)^{-1/n}$ is 1-homogeneous and concave, and affine iff there exists $c>0$ such that
    \begin{equation}\label{equ:affinevol}
        \vol(\chi_1)=\vol(\chi_1\wedge\frac{\chi_2}{c})=\vol(\frac{\chi_2}{c}).
    \end{equation}   
\end{thm}
We shall see in Corollary~\ref{cor:volaffinesat} that this is further equivalent to $\chi_1$ and $\frac{\chi_2}{c}$ having the same saturation.
\begin{proof}
    We follow the proof of Theorem~1.1 in~\cite{BLQ}. The first statement follows from the fact that the finite positive measure $\tilde\mu_\chi$ has support on the set $\{x\in\Delta\mid C^{-1}x_1\le x_2\le Cx_1\}$ for some $C\ge1$. Indeed, set $f(t)\coloneqq\int_\Delta(t_1x_1+t_2x_2)^{-n}\,d\tilde\mu(x)$.    
    As $\vol(\chi_{ct})=c^{-n}\vol(\chi_t)$ for $c>0$, $t\in\R_+^2\setminus\{(0,0)\}$, it follows that $f(t)^{-1/n}$ is 1-homogeneous, and we may assume $t_1+t_2=1$. The same argument as in~\cite{BLQ} shows that $f(t)$ is smooth on $t_1\in (-\varepsilon, 1+\varepsilon)$ for some $\varepsilon=\varepsilon(C)>0$, and that $f(t)^{-1/n}$ is concave. Moreover, $f(t)^{-1/n}$ is affine iff $\tilde{\mu}$ is supported at a single point, or equivalently, $\supp\mu_{\chi}\subset \{x_2=cx_1\}$ for some $c>0$. It remains to show that this is equivalent to~\eqref{equ:affinevol}.
    
    By Proposition~\ref{prop:volbymeas}, we have
    \[
    \vol(\chi_1)=\mu_{\chi}\{x_1\le 1\}=\mu_\chi\{x_2\le c\}=\vol(\frac{\chi_2}{c})=\mu_\chi\{x_1\le 1, x_2\le c\}=\vol(\chi_1\wedge\frac{\chi_2}{c}).
    \]
    Conversely, assume~\eqref{equ:affinevol} holds. Proposition~\ref{prop:Hprop} implies that for $x_2\ge cx_1$, we have
    \[
    \vol(\frac{\chi_1}{x_1})\ge H(x_1, x_2)\ge H(x_1, cx_1)=\vol(\frac{\chi_2}{cx_1}),
    \]
    and so $H(x_1,x_2)=\vol(\frac{\chi_1}{x_1})$; by symmetry, if $x_2\le cx_1$, $H(x_1, x_2)=\vol(\frac{\chi_2}{x_2})$. This implies $\supp\mu_\chi\subset\{x_2=cx_1\}$ as desired.
\end{proof}

\subsection{A pseudodistance}

Given two negative bounded norms $\chi_1,\chi_2$, we set 
\begin{equation}\label{equ:d1}
    \de_1(\chi_1,\chi_2)
    \coloneqq
    \vol(\chi_1)
    +\vol(\chi_2)-2\vol(\chi_1\wedge\chi_2).
\end{equation}

\begin{prop}\label{prop:pseudodis}
    Equation~\eqref{equ:d1} defines a well-defined pseudodistance on the space of bounded negative norms. Moreover, for any two such norms $\chi_1$, $\chi_2$, the path $[0,1]\ni t\to\chi_t\=\chi_{(1-t, t)}$ is a geodesic\footnote{Unlike the projective case~\cite{RemiGeodesics}, $t\mapsto\chi_t$ is not a constant speed geodesic in general.}: 
    we have 
    \[
    \de_1(\chi_1, \chi_t)+\de_1(\chi_t, \chi_2)=\de_1(\chi_1, \chi_2).
    \]
\end{prop}
\begin{rmk}
    We note that the pseudodistance $\de_1$ is not a metric on $\Nbdd(X)$. Indeed, for $\chi\le \chi_{\bar X, D}$ constructed in Lemma~\ref{lem:embnormequiv}, we have $\de_1(\chi, \chi_{\bar X, D})=0$. We shall see later that $\de_1$ restricts to a metric on the space of saturated norms, see Corollary~\ref{cor:metricsat}.
\end{rmk}
    
The proof of Proposition~\ref{prop:pseudodis} follows~\cite{Qi}.
\begin{lem}\label{lem:linalg}
    Let $U, V, W$ be finite dimensional vector spaces. Then
    \[
    \dim V+\dim U\cap W\ge\dim U\cap V+\dim V\cap W
    \]
\end{lem}
\begin{proof}
This is a direct computation.
    \begin{align*}
    \dim \frac{V}{U\cap V}+\dim\frac{W}{V\cap W}
    &=\dim \frac{V}{U\cap V}+\dim\frac{W+V}{V} = \dim\frac{W+V}{U\cap V}\\
    &\ge \dim\frac{W+V}{U\cap (V+W)} = \dim \frac{U+V+W}{U}\\
    &\ge \dim\frac{U+W}{U}
    =\dim\frac{W}{W\cap U}.
    \end{align*}
\end{proof}
\begin{proof}[Proof of Proposition~\ref{prop:pseudodis}]
    We first show that $\de_1$ is a pseudometric. It is clear that $\de_1(\chi, \chi)=0$ and $\de_1(\chi_1, \chi_2)=\de_1(\chi_2, \chi_1)$. It remains to check the triangle inequality. Let $\chi_3$ be any bounded negative norm. Unwinding definitions, this is equivalent to check
    \[
    \vol(\chi_3)+\vol(\chi_1\wedge \chi_2)\ge \vol(\chi_1\wedge \chi_3)+\vol(\chi_3\wedge \chi_2),
    \]
    which in turn is enough to check that for any $\la>0$, we have
    \[
    \dim\cF^\la_{\chi_3}R + \dim\cF^\la_{\chi_1}R\cap \cF^\la_{\chi_2}R\ge \dim \cF^\la_{\chi_1}R\cap \cF^\la_{\chi_3}R+\dim \cF^\la_{\chi_3}R\cap \cF^\la_{\chi_2}R.
    \]
    This follows from the previous lemma by taking $U=\cF^\la_{\chi_1}R, V=\cF^\la_{\chi_3}R, W=\cF^\la_{\chi_2}R$.

    Next we show that $\chi_t$ is indeed a geodesic. Again by writing out definitions, it suffices to show 
    \begin{equation}\label{equ:geoddis}
    \dim\cF^\la_{\chi_t}R-\dim \cF^\la_{\chi_t}R\cap \cF^\la_{\chi_1}R = \dim \cF^\la_{\chi_2}R\cap \cF^\la_{\chi_t}R-\dim\cF^\la_{\chi_1}R\cap\cF^\la_{\chi_2}R.
    \end{equation}
    Before verifying this, we claim that 
    \[
    \cF^\la_{\chi_t}R+\cF^\la_{\chi_1}R = \cF^\la_{\chi_t}R\cap\cF^\la_{\chi_2}R+\cF^\la_{\chi_1}R.
    \]
    Indeed, it suffices to show $\subseteq$, since the reverse inclusion is clear. To see this, write $f\in \cF^\la_{\chi_t}R$ as $f=\sum f_{\mu, \nu}$ for $f_{\mu, \nu}\in \cF^\mu_{\chi_1}R\cap \cF^\nu_{\chi_2}R$, with $(1-t)\mu+t\nu=\la$. If $\mu\le \la$, then $f_{\mu, \nu}\in \cF^\la_{\chi_1}R$; if $\mu>\la$, then we must have $\nu<\la$, making $f_{\mu, \nu}\in \cF^\la_{\chi_t}R\cap\cF^\la_{\chi_2}R$. This verifies the claim, and we now finish the proof of~\eqref{equ:geoddis}. Using the claim, and Lemma~\ref{lem:geod}, we have
    \[
    \frac{\cF^\la_{\chi_t}R}{\cF^\la_{\chi_t}R\cap \cF^\la_{\chi_1}R}\cong \frac{\cF^\la_{\chi_t}R+\cF^\la_{\chi_1}R}{\cF^\la_{\chi_1}R}=\frac{\cF^\la_{\chi_2}R\cap \cF^\la_{\chi_t}R+\cF^\la_{\chi_1}R}{\cF^\la_{\chi_1}R}\cong \frac{\cF^\la_{\chi_2}R\cap \cF^\la_{\chi_t}R}{\cF^\la_{\chi_1}R\cap \cF^\la_{\chi_2}R},
    \]
    as desired.
\end{proof}

\section{The Fubini--Study operator and saturated norms}\label{sec:FSop}
In this section we relate negative norms on $R$ with functions on $X^\na$ through the Fubini--Study and the infimum norm operators. Via the Fubini--Study operator, this gives an isometric isomorphism between $\cT^\hom_\Q$ and $\cH(X)$, both equipped with the $\de_\infty$ metric. We also study the saturation of bounded negative norms, and show that the $\de_1$ pseudometric restricts to a metric on the space of saturated norms. We finish the proof of Theorems~B--D. 
\subsection{The Fubini--Study operator}
The \emph{Fubini--Study operator} associates to any bounded negative norm $\chi$ a function 
\[
\FS(\chi)\colon X^\na\to\R\cup\{+\infty\}
\]
given by 
\[
\FS(\chi)(v)=\sup_{f\in R\setminus k}\frac{v(f)}{\chi(f)}.
\]
\begin{exam}\label{exam:FSmon}
    If $R=k[z_1,\dots,z_n]$ and $\chi$ is a monomial negative valuation on $R$,  then $\FS(\chi)=\max_i\a_i^{-1}\log^+|z_i|$, where $\a_i\coloneqq -\chi(z_i)>0$. In particular, $\FS(\chi)$ is a real Fubini--Study function. Moreover if $\Lambda\subset\R$ is a nonzero subgroup, and $\a_i\in\Lambda$ for all $i$, then $\FS(\chi)\in\cH_\Lambda$.
\end{exam}
\begin{lem}\label{lem:FSimm}
    Let $\iota\colon X\hto X'$ be a closed immersion of affine varieties, induced by a surjective $k$-algebra homomorphism $\tau\colon R'\to R$, where $X=\Spec R$ and $X'=\Spec R'$. Further, let $\chi'$ be a bounded negative norm on $R'$ and let $\chi$ be the residue norm of $\chi'$ by $\tau$. Then $\chi$ is a bounded negative norm, and if we set $\f\coloneqq\FS(\chi)$, $\f'\coloneqq\FS(\chi')$, then $\f=\f'\circ\iota^\na$.
\end{lem}
\begin{proof}
    We already know from Corollary~\ref{cor:resbnn} that $\chi$ is a bounded negative norm. Pick any $v\in X^\na$ and set $v':=\iota^\na(v)$. Then 
    \begin{multline*}
    \f(v)
    =\sup\{\tfrac{v(f)}{\chi(f)}\mid f\in R\setminus k\}
    =\sup\{\tfrac{v(f)}{\chi'(f')}\mid f'\in R'\setminus k, \tau(f')=f\}\\
    =\sup\{\tfrac{v(\tau(f'))}{\chi'(f')}\mid f'\in R'\setminus k\}
    =\sup\{\tfrac{v'(f')}{\chi'(f')}\mid f'\in R'\setminus k\}
    =\f'(v'),
    \end{multline*}
    and we are done.
\end{proof}
\begin{prop}\label{prop:FSft}
    For any nonzero subgroup $\Lambda\subset\R$, we have $\FS(\cT_\La)=\cH_\La$.
\end{prop}
\begin{proof}
    First consider any $\chi\in\cT_\Lambda$. By assumption, there exists a negative monomial norm $\chi'$ on $R'=k[z_1,\dots,z_N]$ and a surjective $k$-algebra homomorphism $\tau\colon R'\to R$ such that $\chi$ is the residue norm of $\chi'$ by $\tau$. By Proposition~\ref{prop:ft} we may assume $\chi'(R'\setminus k)\subset-\Lambda$. Example~\ref{exam:FSmon} shows that $\FS(\chi')\in\cH_\Lambda(\A^N)$, so by Lemma~\ref{lem:FSimm} and Corollary~\ref{cor:FSfin}, we have $\FS(\chi)\in\cH_\Lambda$. 

    Conversely, given any $\f=\max_i\la_i^{-1}\log^+|f_i|\in \cH_\La$, there is a surjective morphism $R'\to R$ sending $z_i$ to $f_i$. Let $\chi$ be the residue norm of the monomial norm $\chi'$ on $R'$ such that $\chi'(z_i)=-\la_i$. Then it follows easily from the definition that $\FS(\chi)=\f$.
\end{proof}
Despite the name, $\FS(\chi)$ will not be a Fubini--Study function for general bounded negative norms. However, we have 
\begin{prop}\label{prop:FSloggr}
    For any bounded negative norm $\chi$, the function $\FS(\chi)$ is lsc and of logarithmic growth.
\end{prop}
We leave it to the reader to verify the following formal properties.
\begin{lem}\label{lem:FSformal}
    For any bounded negative norm $\chi$, the function $\FS(\chi)$ is homogeneous with respect to the action of $\R_+^\times$ on $X^\na$. We also have $\FS(t\chi)=t^{-1}\FS(\chi)$ for any $t\in\R_+^\times$. Moreover, if $\chi\le\chi'$, then $\FS(\chi)\le\FS(\chi')$.
\end{lem}
\begin{proof}[Proof of Proposition~\ref{prop:FSloggr}]
    For any $f\in R\setminus k$, the function $X^\na\ni v\mapsto v(f)/\chi(f)\in\R\cup\{-\infty\}$ is continuous, so it follows that $\f:=\FS(\chi)$ is lsc. Pick any negative norm $\chi'\in \cT_\Q$ of finite type and set $\f'\coloneqq\FS(\chi')\in \cH(X)$. Then $C\chi'\le\chi\le C^{-1}\chi'$ for some $C\ge 1$. By Lemma~\ref{lem:FSformal} this gives $C^{-1}\f'\le\f\le C\f'$, so since $\f'$ is a Fubini--Study function, $\f$ is of logarithmic growth.
\end{proof}
Recall from~\S\ref{sec:loggr} that $\cL^\infty(X)$ is the space of functions of logarithmic growth on $X$.
As an immediate consequence of Proposition~\ref{prop:FSloggr}, we have
\[
\FS\colon (\Nbdd(X), \de_\infty)\to (\cL^\infty(X), \de_\infty)
\]
is distance non-increasing.

\begin{lem}\label{lem:fsnormhom}
    For any bounded negative norm $\chi$, $\FS(\chi)= \FS(\chi^\hom)$.
\end{lem}
\begin{proof}
    By Corollary~\ref{cor:hombdd}, we have that $\chi^\hom$ is also a bounded negative norm. 
    Since $\chi\leq \chi^\hom$, we have $\FS(\chi)\le \FS(\chi^\hom)$ by Lemma~\ref{lem:FSformal}.
   Conversely, let $v\in X^\na$. We may assume $v\in X^\na_\infty$, since otherwise $\FS(\chi)(v)=\FS(\chi^\hom)(v)=0$. For any $0<\varepsilon< \min\{\FS(\chi^\hom)(v),1\}$, there is $f\in R\setminus k$ with $\FS(\chi^\hom)(v)<\frac{v(f)}{\chi^\hom(f)}+\varepsilon$ and $v(f)< 0$. Pick $d\geq 1$ such that  $\chi^\hom(f)<\frac 1d\chi(f^d)-\varepsilon\chi^\hom(f)\leq\frac{1-\varepsilon}{d}\chi(f^d)$. Then $\FS(\chi^\hom)(v)< \frac{v(f^d)}{(1-\varepsilon)\chi(f^d)}+\varepsilon\leq \frac{1}{1-\varepsilon}\FS(\chi)(v)+\varepsilon$. We are done by letting $\varepsilon\to 0$.
\end{proof}

The function $\FS(\chi)$ for a bounded negative norm $\chi$ can be computed more explicitly as follows.
\begin{lem}\label{lem:FSgeomnorm}
    If $\chi=\chi_{\bar{X},D}$ is a geometric norm, 
    then $\FS(\chi)=\psi_D$.
\end{lem}
\begin{proof}
    By definition, $\f:=\FS(\chi)$ satisfies $\f(v)=\sup_{f\in R\setminus k}\frac{v(f)}{\chi(f)}$. Pick any $f\in R\setminus k$. Then $\chi(f)=-m$, where $m$ is minimal such that $fs_D^m$ is a regular section of $\cO_{\bar{X}}(mD)$. This implies $v(f)=v(s_D^mf)-v(s_D^m)\ge0-mv(D)$, so $\f(v)\le v(D)$. 
    
    For the reverse inequality we may assume $v(D)>0$. 
    As $D$ is assumed semiample, we can find $m\ge1$ and a section $s\in H^0(\bar{X},mD)$ that does not vanish at the center of $v$ on $\bar{X}$. If we set $f:=s/s_D^m$, then $f\in R\setminus k$, $\chi(f)=-m$, and $v(f)=-mv(D)<0$, so $\f(v)\ge v(D)$. 
\end{proof}

More generally, use the notation of~\S\ref{sec:baseideals}, \ie pick a polarized compactification $(\bar X,D)$ and $r\in\Z_{\ge1}$ such that $rD$ is Cartier and $\chi\le r^{-1}\chi_{\bar X,D}$. Let $\fb_\bullet(\chi)$ be the corresponding graded sequence of base ideals at infinity, computed inside $R(\bar X,rD)$.

\begin{lem}\label{lem:FSbaseideal}
    In the above notation, we have 
    \[
    \FS(\chi)(v)=rv(D)-v(\fb_\bullet(\chi)),
    \]
    for any $v\in X^\na$.
\end{lem}
\begin{proof}
    We begin with the observation that 
    \[
    \FS(\chi)(v)
    =\sup_{m\ge1}\sup_{0\ne f\in\cF_\chi^mR}
      \frac{-v(f)}m.
    \]
    Note that for any $f\in R\setminus k$ with $v(f)<0$ and $k\ge 1$ we have
    \[
    \frac{v(f^k)}{\lfloor k\chi(f)\rfloor}\le \frac{v(f)}{\chi(f)}\le \frac{v(f)}{\chi^\hom(f)}.
    \]
    Taking supremum over $k$ and $f$ gives the desired equality.
    Finally, by the definition of the base ideals,
    \[  
    \FS(\chi)(v)
    =\sup_{m\ge1}\sup_{0\ne f\in\cF_\chi^mR}
      \frac{-v(f)}m
    =\sup_{m\ge1}\left(
      rv(D)-\frac{v(\fb_m(\chi))}{m}\right)
    =rv(D)-v(\fb_\bullet(\chi)),
    \]
    completing the proof.
\end{proof}

\subsection{The infimum norm}
Next we define an operator in the opposite direction of the Fubini--Study operator. To any function $\f\in\cL^\infty(X)$ we set
\[
    \IN(\f)(f)\coloneqq\inf_{w\in X^\na_\infty}\frac{w(f)}{\f(w)}\in[-\infty,0)
\]
for $f\in R\setminus k$. 
This defines a homogeneous negative norm, which we call the \emph{infimum norm} associated to $\f$. The following properties are clear from the definitions. 
\begin{lem}\label{lem:INformal}
    For any $\varphi, \psi\in\cL^\infty$ with $\varphi\leq \psi$, and $t\in \R^\times_{>0}$, we have $\IN(t\varphi)=t^{-1}\IN(\varphi)$ and $\IN(\varphi)\leq \IN(\psi)$.
\end{lem}
\begin{lem}\label{lem:INFScomp}
    For any $\chi\in \Nbdd(X)$, we have $\IN\circ\FS(\chi)\geq \chi$, with equality iff $\chi$ is homogeneous.
\end{lem}
\begin{proof}
    For any $f\in R\setminus k$ and $w\in X^\na_\infty$, we have $\FS(\chi)(w)\geq \frac{w(f)}{\chi(f)}$ and so
    \[
    \IN\circ\FS(\chi)(f)=\inf_w\frac{w(f)}{\FS(\chi)(w)}\geq \chi(f).
    \]
    Since $\IN\circ\FS(\chi)$ is homogeneous, we have $\chi\le \chi^\hom\le \IN\circ\FS(\chi^\hom)=\IN\circ\FS(\chi)$ using Lemma~\ref{lem:fsnormhom}. Thus if equality holds, then $\chi$ is homogeneous.
    Now assume $\chi$ is homogeneous. Fix $f\in R\setminus k$. By the maximum modulus principle (see~\cite[Proposition~1.3.1]{Ber90}) for the completion of $(R, \chi)$, $\chi(f)=\min_{w\geq \chi}w(f)$. Thus for any $f\in R\setminus k$, we can find $w\ge \chi$ with $\chi(f)=w(f)$, and it follows that $\FS(\chi)(w)=1$.
    Hence
    \[
    \IN\circ\FS(\chi)(f) \leq \frac{w(f)}{\FS(\chi)(w)}=\chi(f),
    \]
    and equality follows.
\end{proof}

\begin{prop}\label{prop:INFScomp}
    For any $\chi\in\Nbdd(X)$, $\IN\circ \FS(\chi)=\chi^\hom$.
\end{prop}
\begin{proof}
    This follows from Lemma~\ref{lem:fsnormhom} and Lemma~\ref{lem:INFScomp}:
    \[
    \IN\circ\FS(\chi)=\IN\circ \FS(\chi^\hom)=\chi^\hom.
    \]
\end{proof}

We claim that $\IN(\f)$ is a homogeneous negative bounded norm. To see this, 
we first consider the following special case.
\begin{lem}\label{lem:INPLplus}
    If $\f=\p_{\bar{X},D}$ for a semiample, strictly effective Cartier divisor $D$ at infinity on an integrally closed compactification $\bar{X}$ of $X$, then $\IN(\f)=\chi_{\bar{X},D}^{\hom}\in\cT^\hom_\Q$.
\end{lem}
\begin{proof}
    By Lemma~\ref{lem:FSgeomnorm} $\varphi=\FS(\chi_{\bar X, D})$, and so $\IN(\varphi)=\chi_{\bar X, D}^\hom$ by Proposition~\ref{prop:INFScomp}, which is further finitely generated by Proposition~\ref{prop:fthom}.
\end{proof}
As $\f\mapsto\IN(\f)$ is homogeneous and increasing, we conclude
\begin{lem}\label{lem:innormbd}
    For any function $\f\in\cL^\infty$, the infimum norm $\chi=\IN(\f)$ is a homogeneous bounded negative norm. If $\f$ is further a Fubini-Study function, then $\IN(\f)$ is of finite type.
\end{lem}
\begin{proof}
Since $\f$ is of logarithmic growth, for any $\psi=\psi_{\bar X, D}\in \PL^+(X)$ with $D$ Cartier, there exists $C\geq 1$ such that 
\[
C^{-1}\psi\leq \f\leq C\psi.
\]
Applying $\IN$ to the above, and combining with Lemma~\ref{lem:INPLplus} yields
\[
C\chi_{\bar X, D}^\hom\leq \IN(\varphi)\leq C^{-1}\chi_{\bar X, D}^\hom.
\]
Note that by Corollary~\ref{cor:hombdd}, we have $\chi_{\bar X, D}^\hom$ is bounded, and hence so is $\IN(\f)$.
Now if $\varphi$ is further a Fubini-Study function, then by Proposition~\ref{prop:FSft} there is a norm $\chi$ of finite type such that $\varphi=\FS(\chi)$. Proposition~\ref{prop:INFScomp} in turn gives $\IN(\varphi)=\chi^\hom$, which is also of finite type, by Proposition~\ref{prop:fthom}.
\end{proof}
The proof shows that the $\IN$ operator also gives a distance non-increasing map 
\[
\IN\colon (\cL^\infty(X), \de_\infty)\to (\cN^\hom_{\mathrm{bdd}}(X), \de_\infty).
\]

\begin{cor}\label{cor:FSisom}
    For $\chi, \chi'\in \Nbdd(X)$, we have $\de_\infty(\FS(\chi), \FS(\chi'))=\de_\infty(\chi^\hom, \chi'^\hom)$.
\end{cor}
\begin{proof}
    Combining Proposition~\ref{prop:INFScomp} and Lemma~\ref{lem:fsnormhom}, we have
    \begin{align*}
    \de_\infty(\chi^\hom, \chi'^\hom)
    &=\de_\infty(\IN\circ\FS(\chi), \IN\circ\FS(\chi'))
    \leq \de_\infty(\FS(\chi), \FS(\chi'))\\
    &=\de_\infty(\FS(\chi^\hom), \FS(\chi'^\hom))\leq \de_\infty(\chi^\hom, \chi'^\hom).
    \end{align*}
\end{proof}
\begin{cor}\label{cor:FSsurj}
    For any nonzero subgroup $\La\subset\R$, $\FS\colon(\cT_\La, \de_\infty)\to (\cH_\La, \de_\infty)$ is a surjective distance non-increasing map which factors as 
    \begin{itemize}
        \item a surjective distance non-increasing map $(\cT_\La, \de_\infty)\to (\cT_{\Q\La}^\hom, \de_\infty)$;
        \item an isomorphism $(\cT_{\Q\La}^\hom, \de_\infty)\to (\cH_{\Q\La}, \de_\infty)$, with inverse given by $\IN$.
    \end{itemize}
\end{cor}
\begin{proof}
    Surjectivity of $\FS$ is a direct consequence of Proposition~\ref{prop:FSft}. By Propostion~\ref{prop:fthom}, $\FS$ factors through homogenization.
    It follows from Corollary~\ref{cor:FSisom} and Proposition~\ref{prop:INFScomp} that $\FS: (\cT^\hom_{\Q\Lambda}, \de_\infty)\to (\cH_{\Q\Lambda}, \de_\infty)$ is an isometry with inverse given by $\IN$. This implies homogenization is surjective. It is also distance reducing since
    \[
    \de_\infty(\chi^\hom, \chi'^\hom)=\de_\infty(\FS(\chi), \FS(\chi'))\leq \de_\infty(\chi, \chi').
    \]
\end{proof}
Write $\cT^{\mathrm{int}}_\Z\subset\cT_\Z$ for the space of $\Z$-valued negative norms of finite type corresponding integrally closed polarized compactifications $(\bar X,D)$ under the Rees correspondence, see Proposition~\ref{prop:Reescorr}.
\begin{prop}\label{prop:TQhom}
    Homogenization provides a bijection $\cT^{\mathrm{int}}_\Z/{\sim}\simto\cT^\hom_\Q$, with inverse given by round-down.
\end{prop}  
\begin{proof}
    Pick any $\chi\in\cT^{\mathrm{int}}_\Z$, and we may assume $\chi=\chi_{\bar{X},D}$ by Lemma~\ref{lem:equivhom}, where $(\bar{X}, D)$ is an integrally closed polarized compactification.
    Now $\chi^{\hom}$ is $\Q$-valued by Proposition~\ref{prop:homic}, and it is negative and of finite type by Proposition~\ref{prop:fthom}. Thus $\chi^\hom\in\cT_\Q^\hom$, and we have $\chi\sim\lfloor\chi^\hom\rfloor$ by Proposition~\ref{prop:homic}.

    Conversely, suppose $\chi'\in\cT^\hom_\Q$, that is, $\chi'$ is a negative homogeneous norm of finite type. Then by Corollary~\ref{cor:FSsurj} there is $\chi\in \cT_\Z$ with $\chi^\hom=\chi'$. By the Rees correspondence, there is a polarized compactification $(\bar X, D)$ with $\chi\sim \chi_{\bar X, D}$.
    Taking the integral closure of $(\bar X, D)$ gives some $\Tilde{\chi}\in \cT_\Z^\mathrm{int}$ with $\Tilde{\chi}\sim\lfloor\chi'\rfloor$, by Proposition~\ref{prop:notic}.
\end{proof}
\begin{cor}
    The map $(\bar X, D)\to \FS(\chi_{\bar X, D})$ restricts to a bijection between $\cP\cC^{\mathrm{int}}$ and $\cH(X)$.
\end{cor}

\subsection{The Fubini--Study envelope}
Recall from~\S\ref{sec:FSenv} that the Fubini--Study envelope of $\f\in\cL^\infty$ is
    \[
    Q(\varphi)\coloneqq \sup\{\psi\in \cH(X)\mid \psi\leq \varphi\}.
    \]
\begin{lem}\label{lem:inenv} 
    For $\f\in \cL^\infty$, we have
    $\IN(\varphi)=\IN(Q(\varphi)).$
\end{lem}
\begin{proof}
    Since $Q(\varphi)\leq\varphi$, we have $\IN(\varphi)\geq \IN(Q(\varphi))$. To see the other direction, it suffices to show that for any $f\in R\setminus k$, and $\la>0$, if $\IN(\varphi)(f)\ge -\la$, then $\IN(Q(\varphi))(f)\ge -\la$. Unwinding definitions, it will suffice to show that if $\log^+|f|\le \la\varphi$, then $\log^+|f|\le \la Q(\varphi)$. To this end, note that $\varphi$ is of logarithmic growth, and we may pick $\rho\in \cH(X)$ with $\rho\le \la\varphi$. Then $\max\{\log^+|f|, \rho\}\in \cH(X)$ with $\max\{\log^+|f|, \rho\}\le \la\varphi$. It follows that $\log^+|f|\le \max\{\log^+|f|, \rho\}\le Q(\la\varphi)=\la Q(\varphi)$.
\end{proof}
Similar to Proposition~\ref{lem:INFScomp}, we have
\begin{prop}\label{prop:FSINcomp}
    For $\f\in \cL^\infty$,
    $\FS(\IN(\varphi))=Q(\varphi)$.
\end{prop}
\begin{proof}
    Both functions vanish identically on $X^\beth$.
    Fix $v\in X^\na_\infty$. For any $f\in R$, we have $\IN(\varphi)(f)\leq \frac{v(f)}{\varphi(v)}$. Hence 
    \[
    \FS(\IN(\varphi))(v)=\sup_f \frac{v(f)}{\IN(\varphi)(f)}\leq \varphi(v).
    \]
    Apply the above to $Q(\varphi)$ and use Lemma~\ref{lem:inenv}, we get 
    \[
    \FS(\IN(\varphi))=\FS(\IN(Q(\varphi)))\le Q(\varphi).
    \]
    Conversely, for any $\psi\in \cH(X)$ with $\psi\leq \varphi$, by Corollary~\ref{cor:FSsurj}, Lemma~\ref{lem:FSformal} and Lemma~\ref{lem:INformal} we have
    \[
    \psi=\FS(\IN(\psi))\le \FS(\IN(\f)),
    \]
    and therefore $Q(\varphi)\leq \FS(\IN(\varphi))$.
\end{proof}
\begin{cor}
    We have $\IN\circ \FS\circ \IN=\IN$, and $\FS\circ\IN\circ\FS=\FS$.
\end{cor}
\begin{proof}
    This is an immediate consequence of Lemma~\ref{lem:fsnormhom}, Proposition~\ref{prop:FSINcomp}, Lemma~\ref{lem:inenv} and Proposition~\ref{prop:INFScomp}.
\end{proof}
\subsection{The Darvas metric on \texorpdfstring{$\cH_\R(X)$}{HR(X)}}
Recall from Corollary~\ref{cor:FSsurj} that the surjection $\FS\colon (\cT_\R, \de_\infty)\twoheadrightarrow(\cH_\R, \de_\infty)$ factors through an isometric isomorphism
\[
\FS: (\cT_\R^\hom, \de_\infty)\rightarrow (\cH_\R, \de_\infty),
\]
with inverse given by $\IN$. Thus the pseudometric $\de_1$ on $\cT_\R$ defines a pseudometric on $\cH_\R(X)$. Concretely, for any $\varphi, \psi\in \cH_\R(X)$,
\[
\de_1(\varphi, \psi)\= \de_1(\IN(\varphi), \IN(\psi)).
\]
We now show that the Fubini--Study operator identifies $(\cH_\R, \de_1)$ as the Hausdorff quotient of $(\cT_\R, \de_1)$.
\begin{thm}\label{thm:metricH}
    The pseudodistance $\de_1$ is a metric on $\cH_\R(X)$, uniquely characterized by
    \[
    \varphi\ge \psi \implies \de_1(\varphi,\psi)=\vol(\varphi)-\vol(\psi);
    \]
    \[
    \de_1(\varphi, \psi)=\inf\{\de_1(\varphi, \tau)+\de_1(\tau, \psi)\mid \tau\in \cH_\R(X), \tau\le \varphi, \psi\}
    \] 
    for all $\varphi, \psi\in \cH_\R(X)$.
\end{thm}
Recall from~\S\ref{sec:strongP} that for $\f\in \cH_\R$, $\vol(\f)=(\f^n)$. The proof relies on the following volume equality:
\begin{lem}\label{lem:volequal}
    For any $\chi\in\cT_\R^\hom$, we have $\vol(\chi)=\vol(\FS(\chi))$.
\end{lem}
\begin{proof}
    Put $\f=\FS(\chi)\in \cH_\R$, and we may approximate $\varphi=\FS(\chi)$ by decreasing local uniform limits of $\varphi_i\in \cH_\Q$. Then Corollary~\ref{cor:FSsurj} gives
    \[
    \de_\infty(\IN(\varphi_i), \chi)=\de_\infty(\varphi_i, \varphi)\to 0.
    \]
    By homogeneity of volumes, we have $\vol(\IN(\varphi_i))\to \vol(\chi)$, and we may further assume $\chi\in \cT^\hom_\Q$. Now $\FS(\chi)\in \cH(X)=\PL^+(X)$ by Lemma~\ref{lem:FSPL}, and we may then pick an integrally closed polarized compactification $(\bar X, D)$ with $\FS(\chi)=\psi_{\bar X,D}=\FS(\chi_{\bar X, D})$. This gives $\chi=\chi_{\bar X, D}^\hom$ by Proposition~\ref{prop:INFScomp}, and $\lfloor\chi\rfloor\sim \chi_{\bar X, D}$ by Proposition~\ref{prop:notic}. Therefore we have
    \[
    \vol(\chi)=\vol(\lfloor\chi\rfloor)=\vol(\chi_{\bar X, D})=\vol(D)=\vol(\FS(\chi)),
    \]
    as desired.
\end{proof}

\begin{proof}[Proof of Theorem~\ref{thm:metricH}]
    Uniqueness is clear once $\de_1$ is given by the formulas above. Let $\varphi, \psi\in \cH_\R(X)$, and let $\chi\=\IN(\varphi), \chi'\=\IN(\psi)$. Then  $\varphi=\FS(\chi), \psi=\FS(\chi')$. Thus Lemma~\ref{lem:volequal} yields $\vol(\varphi)=\vol(\chi), \vol(\psi)=\vol(\chi')$. Now if $\varphi\ge \psi$, then Lemma~\ref{lem:INformal} shows $\chi\ge \chi'$ and we have
    \[
    \de_1(\varphi, \psi)=\de_1(\chi,\chi')=\vol(\chi)-\vol(\chi')=\vol(\varphi)-\vol(\psi).
    \]
    In general, using Proposition~\ref{prop:canapprox} and Proposition~\ref{prop:fthom}, for any $\varepsilon>0$, we can find $\chi''\in \cT_\R^\hom$ with $\chi''\le \chi\wedge \chi'$ such that $\de_1(\chi'', \chi\wedge\chi')<\varepsilon$. If $\tau\=\FS(\chi'')$, then 
    \begin{align*}
    \de_1(\varphi, \tau)+d_1(\tau, \psi)&=\de_1(\chi, \chi'')+\de_1(\chi'', \chi')\le \de_1(\chi, \chi\wedge \chi')+\de_1(\chi\wedge\chi', \chi')+2\varepsilon\\
    &=\de_1(\chi, \chi')+2\varepsilon=\de_1(\varphi, \psi)+2\varepsilon,
    \end{align*}
    and letting $\varepsilon\to 0$, this shows the second equality. 

    Finally, to see that $\de_1$ is a metric on $\cH_\R(X)$, we assume $\de_1(\varphi, \psi)=0$. Then there exists $(\tau_i)$ in $\cH_\R(X)$ with $\tau_i\le \varphi, \psi$ and $\vol(\tau_i)$ converging to both $\vol(\varphi)$ and $\vol(\psi)$. By Lemma~\ref{lem:volweakconv}, this implies $\tau_i$ converges to both $\varphi$ and $\psi$ pointwise on $X^\div$; hence $\varphi=\psi$.
\end{proof}
\subsection{Saturation}
Following~\cite{nakstab1, BLQ, Qi} we study next the $\de_1$ distance on the space of saturated norms.
For a function $\varphi\in \cL^\infty$, we set $\IN_{\Xdiv_\infty}(\varphi)$ to be the negative norm defined as
\[
\IN_{\Xdiv_\infty}(\varphi)(f)\=\inf_{v\in \Xdiv_\infty}\frac{v(f)}{\varphi(v)},
\]
for all $f\in R$. 
\begin{defi}
    Let $\chi\in\Nbdd(X)$. The \emph{saturation} of $\chi$ is defined to be
    \[
    \bar\chi\=\IN_{\Xdiv_\infty}(\FS(\chi)).
    \]
\end{defi}
\begin{defi}
    A bounded negative norm $\chi$ is \emph{saturated} if $\bar\chi=\chi$. Denote by $\cN^s(X)$ the space of saturated negative norms.
\end{defi}
By monotonicity, if $\chi\le \chi'$ are two bounded negative norms, then $\bar\chi\le \bar\chi'$. For any bounded negative norm, the saturation $\bar\chi$ is also bounded:
\begin{lem}\label{lem:satprop}
    For any bounded negative norm $\chi$,
    \begin{itemize}
        \item[(i)] $\chi\le\chi^\hom\le\bar\chi$.
        \item[(ii)] if $\chi$ is geometric, then  $\chi^\hom=\bar\chi$; in particular, $\bar\chi\in\Nbdd(X)$.
        \item[(iii)] $\FS(\chi)(v)=\FS(\chi^\hom)(v)=\FS(\bar \chi)(v)$ for all $v\in \Xdiv_\infty$.
    \end{itemize}
\end{lem}
\begin{proof}
    By Lemma~\ref{lem:INFScomp} and Proposition~\ref{prop:INFScomp}, we have 
    $\chi\le \chi^\hom=\IN(\FS(\chi))\le \bar\chi$. 
    Suppose $\chi=\chi_{\bar X, D}$ is geometric for $D=\sum_ib_iE_i$. After taking the integral closure, we may assume $\bar X$ is integrally closed by Proposition~\ref{prop:notic} and Lemma~\ref{lem:fsnormhom}.
    Then for any $f\in R\setminus k$,
    \[
    \bar\chi(f)\leq \min_i\left\{\frac{\ord_{E_i}(f)}{\FS(\chi)(\ord_{E_i})}\right\} 
    =\min_i\left\{\frac{\ord_{E_i}(f)}{b_i}\right\}=\chi^\hom(f),
    \]
    where the second equality follows from Lemma~\ref{lem:FSgeomnorm}. This proves~(ii) in view of~(i) and Corollary~\ref{cor:hombdd}.

    Finally, by~(i) we have $\FS(\chi)(v)= \FS(\chi^\hom)(v)\le \FS(\bar \chi)(v)$ for all divisorial valuations $v\in \Xdiv_\infty$. On the other hand, by definition of $\bar\chi$, we have for any $v\in \Xdiv_\infty$, and $f\in R\setminus k$, $\bar\chi(f)\le \frac{v(f)}{\FS(\chi)(v)}$. It follows that
    \[
    \FS(\bar\chi)(v)=\sup_f\frac{v(f)}{\bar \chi(f)}\le\FS(\chi)(v).
    \]
    This proves~(iii).
\end{proof}
\begin{prop}\label{prop:volsat}
    For a bounded negative norm $\chi$, we have $\vol(\chi)=\vol(\chi^\hom)=\vol(\bar \chi)$.
\end{prop}
\begin{proof}
    By Lemma~\ref{lem:satprop}(i), it suffices to show that $\vol(\chi)=\vol(\bar\chi)$. Pick a polarized compactification $(\bar X,D)$ with $D$ Cartier. After possibly rescaling, we may assume $\chi\le\bar\chi\le\chi_{\bar X,D}$.
    Then $\cF_\chi^m R\subset \cF_{\bar \chi}^m R\subset H^0(\bar X, mD)$. We claim that 
    \[
    v(\fb_\bullet(\chi))=v(\fb_\bullet(\bar\chi))
    \] 
    for all divisorial valuations $v$ on $X$. Indeed, when $v$ is centered on $X$ both sides vanish since the base ideals are trivial on $X$, and when $v$ is centered at infinity, the desired equality follows from Lemma~\ref{lem:FSbaseideal} and Lemma~\ref{lem:satprop}~(iii).

    By~\cite[Lemma 22]{Sze15}, it now suffices to show that $\cF_\chi^\bullet R$ contains an ample series. This follows from Lemma~\ref{lem:geomnormbd}. Indeed, choose an integer $l\ge1$ such that $\chi\ge l\chi_{\bar X,D}$, using Lemma~\ref{lem:geomnormbd}. For $m$ divisible by $l$, the corresponding linear series contains
    \[
    s_D^{m-m/l}\Hnot(\bar X,(m/l)D).
    \]
    Thus $\cF_\chi^\bullet R$, and hence also $\cF_{\bar\chi}^\bullet R$, contains an ample series, with ample part $D/l$ and effective part $(1-1/l)D$. 
\end{proof}
\begin{cor}\label{cor:volineq}
    If $\chi,\chi'$ are bounded negative norms such that $\FS(\chi)(v)\le \FS(\chi')(v)$ for all $v\in X^\div_\infty$, then $\vol(\chi)\le\vol(\chi')$.
\end{cor}
The following is an analogue of~\cite[Proposition~3.12]{BLQ}.
\begin{prop}\label{prop:equalvol}
     Suppose that $\chi$ and $\chi'$ are bounded negative norms such that $\chi\le\chi'$ and $\vol(\chi)=\vol(\chi')$. Then $\FS(\chi)=\FS(\chi')$ on the set $X^\lin$ of valuations of linear growth.
\end{prop}
\begin{proof}
     Up to scaling we may assume $\chi\le\chi'\le \chi_{\bar X, D}$ for a geometric norm $\chi_{\bar X, D}$. For simplicity of notation, we write $V_m\=\cF^m_\chi R\subseteq \cF^m_{\chi'}R\eqqcolon V_m'$. Suppose to the contrary that for some $v\in X^\lin$, $0<\FS(\chi)(v)<\FS(\chi')(v)$. It is easy to see that $v\in X^\lin_\infty$. Moreover, in view of the observation in the proof of Lemma~\ref{lem:FSbaseideal}, after possibly scaling $v$, we may assume 
     \[
     1=\sup_m\sup_{f\in V_m}\frac{v(f)}{-m}< \sup_m\sup_{f\in V_m'}\frac{v(f)}{-m}.
     \]
     Thus there exists $l\in \Z_{>0}$ and $f\in V_l'\setminus V_l$ such that 
     \[
     -k=v(f)<-l.
     \]
     Now for any $i\in \Z_{>0}$, consider the map
     \[
     \phi\colon V_m\to V_{m+il}'/V_{m+il}, g\mapsto f^ig.
     \]
     Any $g\in \ker(\phi)$ satisfies
     \[
     v(g)=v(f^ig)-v(f^i)\ge -(m+il)+ik=-m+i(k-l).
     \]

    By Lemma~\ref{lem:linbdd}, there exists $B>0$ such that $v(g)\le Bm$ for all $m\ge1$ and $0\ne g\in V_m$. Choose $c\in\Z_{>0}$ with $c(k-l)>B+1$. Taking $i=cm$, the preceding inequality shows that $\phi$ is injective for every $m$. Hence
     
     \[
     \frac{\dim(V_{m+cml}'/V_{m+cml})}{(m+cml)^n/n!}\ge \frac{1}{(1+cl)^n} \frac{\dim V_m}{m^n/n!}\to \frac{1}{(1+cl)^n}\vol(\chi)>0,
     \]
    a contradiction.
\end{proof}

As a direct corollary, we finish the proof of Theorem~C.
\begin{cor}\label{cor:FSisometry}
    The surjection $\FS: (\cT_\R, \de_1)\twoheadrightarrow(\cH_\R, \de_1)$ is an isometry. Its restriction to $(\cT_\R^\hom, \de_1)$ is an isometric isomorphism, with inverse given by $\IN$.
\end{cor}
\begin{proof}
    Thanks to Theorem~\ref{thm:metricH}, it remains to prove that $\de_1=0$ along fibers of $\FS$. Assume $\FS(\chi)=\FS(\chi')$. Applying $\IN$ gives $\de_1(\chi^\hom, \chi'^\hom)=\de_1(\FS(\chi), \FS(\chi'))=0$.
    Further $\de_1(\chi, \chi^\hom)=\vol(\chi^\hom)-\vol(\chi)=0$ by Proposition~\ref{prop:volsat}.
    Thus,
    \[
    \de_1(\chi, \chi')\le \de_1(\chi, \chi^\hom)+\de_1(\chi^\hom, \chi'^\hom)+\de_1(\chi'^\hom, \chi')=0.
    \]
\end{proof}
The next two consequences finish the proof of Theorem~D.
\begin{cor}\label{cor:metricsat}
    The pseudometric $\de_1$ restricts to a metric on $\cN^s(X)$.
\end{cor}
\begin{proof}
    Assume $\de_1(\chi, \chi')=0$ for $\chi, \chi'\in \cN^s$. Then $\vol(\chi)=\vol(\chi\wedge \chi')=\vol(\chi')$, and Proposition~\ref{prop:equalvol} implies in particular $\FS(\chi)=\FS(\chi\wedge\chi')=\FS(\chi')$ on $\Xdiv_\infty$. Thus $\chi=\bar\chi=\bar\chi'=\chi'$, as desired.
\end{proof}

\begin{cor}\label{cor:volaffinesat}
    Let $\chi_1, \chi_2$ be bounded negative norms, and $\chi_t$ the geodesic joining them. Then $t\mapsto \vol(\chi_t)^{-1/n}$ is affine iff there exists $c>0$ such that $\bar \chi_1=\bar{c\chi_2}$.
\end{cor}
\begin{proof}
    By Theorem~\ref{thm:volan}, $\vol(\chi_t)^{-1/n}$ is affine iff there exists $c>0$ such that $\vol(\chi_1)=\vol(\chi_1\wedge c\chi_2)=\vol(c\chi_2)$, and so $\FS(\chi_1)=\FS(\chi_1\wedge c\chi_2)=\FS(c\chi_2)$ on $X^\div_\infty$ by Proposition~\ref{prop:equalvol}, and hence $\bar \chi_1=c\bar\chi_2$. 
    Conversely, if $\bar\chi_1=c\bar \chi_2$, then $\de_1(\chi_1, c\chi_2)=0$ by Proposition~\ref{prop:volsat}. This gives $\vol(\chi_1)=\vol(\chi_1\wedge c\chi_2)=\vol(c\chi_2)$ and we conclude by Theorem~\ref{thm:volan} again.
\end{proof}

\section{Darvas metrics and completions}\label{sec:metriccomp}
In this section, $n\ge2$. We study Darvas metrics on $\cE(X)$ and $\bar \cE(X)$.

\subsection{The Darvas metric on \texorpdfstring{$\cE(X)$}{E(X)}}
Using decreasing approximation by decreasing nets in $\cH_\R$, the $\de_1$ distance extends canonically to $\cE(X)$. To do so, we first extend the volume function on $\cE(X)$ to lsc functions of logarithmic growth. For $\f\in \cL^\infty$,
we define its \emph{volume} to be 
\[
\vol(\varphi)\= \sup \{(\psi^n)\mid \psi\in \PSH(X), \psi\le \varphi\}=\sup \{(\psi^n)\mid \psi\in \PSH(X)\cap \cL^\infty, \psi\le \varphi\}.
\]
For $\psi\in \PSH(X)\cap\cL^\infty$, we have $\psi\le \varphi$ iff $\psi\le P(\varphi)$, and it follows that $\vol(\varphi)=\vol(P(\varphi)).$ It is also clear that $\varphi\mapsto \vol(\varphi)$ is increasing and homogeneous of degree $n$. The Darvas metric $\de_1$ on $\cH_\R$, see Theorem~\ref{thm:metricH}, satisfies
\begin{equation}\label{equ:rooftop}
\de_1(\f, \p)=\vol(\f)+\vol(\p)-2\vol(\f\wedge\p),
\end{equation}
for any $\f, \p\in \cH_\R$.
\begin{prop}\label{prop:volcts}
    The volume functional $\vol$ is continuous along increasing nets of lsc functions of sublogarithmic growth in $\cL^\infty$.
\end{prop}
\begin{proof}
    Let $(\varphi_i)_i$ be an increasing net of lsc functions in $\cL^\infty$ converging to $\varphi\=\sup_i\varphi_i\in \cL^\infty$. Then $\vol(\varphi)\ge \sup_i\vol(\varphi_i)$ by monotonicity. On the other hand, for any $0<t<1$, we can find $\psi\in \PSH(X)\cap\cL^\infty$ with $\psi\le \varphi$ such that $\vol(\psi)>t\vol(\varphi)$. Now for each $i$, $V_i\=\{\psi-t^{-1}\varphi_i<0\}$ defines an open subset of $X^\na_\rho$ for $\rho\in \cH(X)$. By compactness of $X^\na_\rho$, we have $\psi<t^{-1}\varphi_i$ for $i\gg0$. Thus $t\vol(\varphi)<\vol(\psi)\le t^{-n}\vol(\varphi_i)\le t^{-n}\sup_i\vol(\varphi_i)$ for all $i\gg 0$. Letting $t\to 1$ we get $\vol(\varphi)\le \sup_i\vol(\varphi_i)$ as desired. 
\end{proof}
We may then extend Lemma~\ref{lem:volequal} to the following analog of~\cite[Theorem~5.1]{nakstab1}.
\begin{lem}\label{lem:volequalbdd}
    For any $\chi\in \Nbdd(X)$, we have $\vol(\chi)=\vol(\FS(\chi))$.
\end{lem}
\begin{proof}
    By Proposition~\ref{prop:canapprox}, $\vol(\chi)=\lim\vol(\chi^{(\la)})$ for $\chi^{(\la)}\in\cT_\R$. We also have $\FS(\chi^{(\la)})$ increases to $\FS(\chi)$ pointwise, and hence $\vol(\FS(\chi^{(\la)}))\to \vol(\FS(\chi))$ by Proposition~\ref{prop:volcts}. Thus it suffices to show the equality for $\chi\in \cT_\R$. This follows from Lemma~\ref{lem:volequal} and Proposition~\ref{prop:volsat}.
\end{proof}

\begin{thm}\label{thm:d1P}
    There exists a unique metric $\de_1$ on $\cE(X)$, extending the one on $\cH_\R(X)$ and continuous along decreasing nets.
\end{thm}
We first show that the limit is independent of the approximations.
\begin{lem}
    Let $\varphi\in\cE(X)$, and let $(\varphi_j), (\varphi_j')$ be two decreasing nets of functions in $\cH(X)$ converging to $\varphi$. Then $\de_1(\varphi_j, \varphi_j')\to 0$ as $j\to\infty$.
\end{lem}
\begin{proof}
    Put $\chi_j\=\IN(\varphi_j)$ and $\chi_j'=\IN(\varphi_j')$. Then 
    \[
    \de_1(\varphi_j, \varphi_j')=\de_1(\chi_j, \chi_j')=\vol(\varphi_j)+\vol(\varphi_j')-2\vol(\chi_j\wedge \chi_j').
    \]
    We claim that $\vol(\varphi)\le \vol(\chi_j\wedge \chi_j')$ for all $j$. Indeed, by Lemma~\ref{lem:volequalbdd}, we have $\vol(\chi_j\wedge \chi_j')=\vol(\FS(\chi_j\wedge \chi_j'))=\vol(\FS(\IN(\varphi_j\wedge \varphi_j')))=\vol(Q(\varphi_j\wedge \varphi_j'))$, where the last equality follows from Proposition~\ref{prop:FSINcomp}. On the other hand, since $\varphi\le \varphi_j\wedge \varphi_j'$, we have $\vol(\varphi)\le \vol(\varphi_j\wedge \varphi_j')=\vol(P(\varphi_j\wedge\varphi_j'))=\vol(Q(\varphi_j\wedge\varphi_j'))$, where the last equality follows from Lemma~\ref{lem:FSpshenv}. Now using the claim, we deduce
    \[
    \de_1(\varphi_j, \varphi_j')\le \vol(\varphi_j)+\vol(\varphi_j')-2\vol(\varphi)\to 0,
    \]
    since the volume function is continuous along decreasing nets.
\end{proof}
\begin{proof}[Proof of Theorem~\ref{thm:d1P}]
    Given $\varphi, \psi\in \cE(X)$, pick decreasing nets $(\varphi_i), (\psi_i)$ in $\cH(X)$ decreasing to $\varphi, \psi$ respectively. Then by triangle inequality we have
    \begin{align*}
    |\de_1(\varphi_i, \psi_i)-\de_1(\varphi_j, \psi_j)|
    &\le |\de_1(\varphi_i, \psi_i)-\de_1(\varphi_j, \psi_i)|+|\de_1(\varphi_j, \psi_i)-\de_1(\varphi_j, \psi_j)|\\
    &\le \de_1(\varphi_i, \varphi_j)+\de_1(\psi_i, \psi_j)\to 0,
    \end{align*}
    as $i, j\to \infty$. This shows $(\de_1(\varphi_i, \psi_i))_i$ is a Cauchy net, and hence the limit $\de_1(\varphi, \psi)\=\lim_i\de_1(\varphi_i, \psi_i)$ exists, and it is further independent of the choice of converging nets by the previous lemma. This shows $\de_1$ extends uniquely to a pseudometric on $\cE(X)$. Suppose $\de_1(\varphi, \psi)=0$. Then $\lim_i\de_1(\varphi_i, \psi_i)=0$, and by Theorem~\ref{thm:metricH} there exist $\tau_i\in \cH_\R(X)$ with $\tau_i\le \varphi_i\wedge \psi_i$ such that $\lim_i\de_1(\varphi_i, \tau_i)=0$ and $\lim_i\de_1(\tau_i, \psi_i)=0$. By Lemma~\ref{lem:volweakconv} we have $\tau_i$ converges to both $\varphi$ and $\psi$ weakly, and this shows $\varphi=\psi$.

    It remains to prove continuity along decreasing nets. We first note that 
    \begin{equation}\label{equ:volformula}
    \de_1(\f,\p)=\vol(\p)-\vol(\f)
    \end{equation}
    for $\f,\p\in\cE(X)$ satisfying $\f\le\p$. Indeed, pick decreasing nets $(\f_i)_i$ and $(\p_i)_i$ in $\cH(X)$ converging to $\f$ and $\p$, respectively. After replacing $\p_i$ with $\f_i\vee\p_i$, we may assume $\f_i\le \p_i$. Then $\de_1(\f_i,\p_i)=\vol(\p_i)-\vol(\f_i)$, and passing to the limit yields~\eqref{equ:volformula} by the continuity of volume along decreasing nets. If $\f_i\searrow\f$ in $\cE(X)$, this gives
    \[
    \de_1(\f_i,\f)
    =\vol(\f_i)-\vol(\f)\to0,
    \]
    and continuity along decreasing nets follows from the triangle inequality.    
\end{proof}
\begin{rmk}
    Contrary to the setting of~\cite[Theorem~5.5]{nakstab1},
    if $\varphi,\psi\in\cE(X)$, there may be no $\tau\in\cE(X)$ with $\tau\le \varphi\wedge \psi$. Indeed, let $X=\A^1\times \G_m$ and $\varphi=\psi_{D_{(0,1)}+D_{(-1,0)}}, \psi=\psi_{D_{(0,-1)}+D_{(-1,0)}}$ as in Example~\ref{exam:posvolloggr}. Then $\varphi\wedge\psi\=\psi_{D_{(-1,0)}}\notin \cE(X)$.
\end{rmk}
Given $\varphi, \psi\in \cE(X)$, as before we set
\[
    \du(\varphi, \psi)\=\de_1(\varphi, \psi)+|\log \vol(\varphi)-\log\vol(\psi)|.
\]
Recall that the Monge--Amp\`ere operator defines a dense embedding
\[\MA\colon (\bar\cE(X), \de_\infty)\hookrightarrow(\cC_\infty(X), \de_\infty),\]
where
$\bar\cE(X)$ is the set of numerical classes $[\varphi]$ defined by $\varphi\in \cE(X)$.
\begin{prop}
    The canonical map $(\cE(X), \du)\to (\bar \cE(X), \de_\infty)$ is continuous, and extends uniquely to a continuous map to their metric completions.
\end{prop}
\begin{proof}
    Assume first that $(\varphi_i)$ is a net in $\cH(X)$ so that $\du(\varphi_i, \varphi)\to 0$ for some $\varphi\in \cH(X)$. Then $\vol(\varphi_i)\to \vol(\varphi)$ and by Theorem~\ref{thm:metricH} we may find $\tau_i\in \cH_\R(X)$ with $\tau_i\le \varphi_i\wedge \varphi$ such that 
    \[
    \de_1(\varphi_i, \tau_i)+\de_1(\tau_i, \varphi)\to 0.
    \]
    Thus $\vol(\tau_i)\to \vol(\varphi)$, and further, since $(\tau_i, \varphi_i^{n-1})\le (\varphi_i^n)$, we have $\ii(\tau_i, \varphi_i)\to 0$. This gives $\de_\infty([\tau_i], [\varphi_i])\to 0$ by Proposition~\ref{prop:metrictop}. The same argument shows $\de_\infty([\tau_i], [\f])\to 0$, and thus 
    \[
    \de_\infty([\varphi_i], [\varphi])\le \de_\infty([\f_i], [\tau_i])+\de_\infty([\tau_i], [\varphi])\to 0,
    \]
    as desired.
    In general, if $(\varphi_i)$ is a net in $\cE(X)$ with $\du(\varphi_i, \varphi)\to 0$ for some $\varphi\in\cE(X)$, then we may pick decreasing nets $(\psi_{ij}), (\psi_j)$ in $\cH(X)$ converging to $\varphi_i, \varphi$ respectively. Then $\lim_i\lim_j\du(\p_{ij}, \psi_j)=0$, and the same argument as above shows $\lim_i\lim_j\de_\infty([\p_{ij}], [\psi_j])=0$. This combined with Example~\ref{exam:decstrong} gives $\de_\infty([\varphi_i], [\varphi])\to 0$.  

    To prove the second assertion, it suffices to show that if $(\f_i)$ is a Cauchy net in $(\cE(X), \du)$, then $([\f_i])$ is Cauchy in $(\bar \cE(X), \de_\infty)$. We may wlog assume $(\f_i)$ is a net in $\cH(X)$, by picking a decreasing net in $\cH(X)$ converging to each $\f_i$.  First note that $\vol(\f_i)\to C$ for some $C>0$.
    Now using Theorem~\ref{thm:metricH}, we can find $\tau_{ij}\in \cH_\R(X)$ with $\tau_{ij}\le \f_i\wedge\f_j$ so that 
    \[
    \de_1(\f_i, \tau_{ij})+\de_1(\tau_{ij}, \f_j)\to 0.
    \]
    Since $\tau_{ij}\le \f_i$, we have $\bar\ii(\tau_{ij}, \f_i)\to 0$. Similarly $\bar\ii(\tau_{ij}, \f_j)\to 0$. Thus
    \[
    \de_\infty([\f_i], [\f_j])\le \de_\infty([\f_i], [\tau_{ij}])+\de_\infty([\tau_{ij}], [\f_j])\to 0
    \]
    by Proposition~\ref{prop:metrictop}.
\end{proof}

While we do not expect $(\cE(X), \du)$ to be complete in general, we point out one situation when it is.
\begin{prop}\label{prop:Ecomplete}
    Assume $X$ is smooth with no nonconstant global invertible functions, and that it admits a compactification $\bar X$ with $q(\bar X)=0$ and rational singularities.
    Then $(\cE(X), \du)$ is complete.
\end{prop}
\begin{rmk}
    This is for example the case when $X$ is a smooth affine open subset of a klt Fano variety with no nonconstant global invertible functions. Typical examples of such varieties arise when they admit asymptotically conical Calabi--Yau metrics; see~\cite{Wu26} for details. 
\end{rmk}
\begin{proof}
    First note that by Lemma~\ref{lem:irregzero}, we have $\bar \cE(X)=\cE(X)$.
    Let $(\f_i)$ be a Cauchy net in $(\cE(X), \du)$. We may also wlog assume $(\f_i)$ is a net in $\cH(X)$, by picking a decreasing net in $\cH(X)$ converging to each $\f_i$.  First note that $\vol(\f_i)\to C$ for some $C>0$.
    Now using Theorem~\ref{thm:metricH}, we can find $\tau_{ij}\in \cH_\R(X)$ with $\tau_{ij}\le \f_i\wedge\f_j$ so that 
    \[
    \de_1(\f_i, \tau_{ij})+\de_1(\tau_{ij}, \f_j)\to 0.
    \]
    Since $\tau_{ij}\le \f_i$, we have $\bar\ii(\tau_{ij}, \f_i)\to 0$. Similarly $\bar\ii(\tau_{ij}, \f_j)\to 0$. Thus
    \[
    \de_\infty([\f_i], [\f_j])\le \de_\infty([\f_i], [\tau_{ij}])+\de_\infty([\tau_{ij}], [\f_j])\to 0
    \]
    by Proposition~\ref{prop:metrictop}. By Proposition~\ref{prop:metricequiv}, this shows $(\MA(\f_i))$ is a Cauchy net in $(\cC_\infty(X), \de_\infty)$ and hence there exists $\mu\in \cC_\infty(X)$ with $\MA(\f_i)\to \mu$ in $\de_\infty$; see Corollary~\ref{cor:curvinfcomplete}. On the other hand, since $X$ is smooth, there exists $\f\in \cE(X)$ with $\MA(\f)=\mu$ by Theorem~\ref{thm:metricMAsol}, and we have $[\f_i]\to [\f]$ strongly in $\cE(X)$. We will show that $\f_i\to \f$ in $\du$.
    
    Indeed, $(\f_i)$ is a minimizing net for $\mu=\MA(\f)$. By Proposition~\ref{prop:numtrivlim} and Lemma~\ref{lem:ntrivCDiv}, every subnet has a further subnet along which $\f_i-\f\to\tau$ weakly for some numerically trivial $\tau\in\PL(X)\otimes_\Q\R$. Since $\bar\cE(X)=\cE(X)$, we must have $\tau\equiv0$. Thus every subnet has a further subnet converging weakly to $\f$, so that $\f_i\to\f$ weakly.
It remains to show that $\de_1(\f_i, \f)\to 0$. But since $\f_i\vee\f\to \f$ weakly, we have $\vol(\f_i\vee\f)\to \vol(\f)$ by upper-semicontinuity. Thus
    \[
    \de_1(\f_i, \f)\le \de_1(\f_i, \f_i\vee\f)+\de_1(\f_i\vee\f, \f)\to 0,
    \]
    and we are done.
\end{proof}

\subsection{The regularized Fubini--Study operator}
Recall that the Fubini--Study operator is a map $\FS: \Nbdd\to \cL^\infty$. For any $\chi\in \Nbdd$ we set $\FS^\star(\chi)\=\FS(\chi)^\star$ to be the usc regularization of $\FS(\chi)$. 

\begin{prop}\label{prop:regFS}
    For any $\chi, \chi'\in \Nbdd$, we have
    \begin{itemize}
        \item[(i)] $\FS^\star(\chi)=\FS(\chi)$ on $\Xdiv$;
        \item[(ii)] $\bar\chi=\IN(\FS^\star(\chi))$;
        \item [(iii)]$\bar\chi=\bar\chi'$ iff $\FS(\chi)=\FS(\chi')$ on $\Xdiv$ iff $\FS^\star(\chi)=\FS^\star(\chi')$;
        \item [(iv)]if the envelope property holds, then $\vol(\FS^\star(\chi))=\vol(\FS(\chi))$ and $\de_1(\FS(\chi^{(\la)}), \FS^\star(\chi))\to 0$ as $\la\to \infty$.
    \end{itemize}
\end{prop}
\begin{proof}
    Since $\FS(\chi)=\sup_\la\FS(\chi^{(\la)})\in \cL^\infty$ and $\FS(\chi^{(\la)})\in \cH_\R$,~(i) is a consequence of Lemma~\ref{lem:divneg}. It follows that $\bar\chi=\IN_{\Xdiv_\infty}(\FS^\star(\chi))\ge \IN(\FS^\star(\chi))$. For the reverse inequality, note that for any $f\in R\setminus k$ and any $v\in \Xdiv_\infty$, we have
    \[
    \FS^\star(\chi)(v)-\frac{v(f)}{\bar\chi(f)}\ge 0,
    \]
    an inequality which extends to all $X^\na_\infty$ since $v\mapsto \FS^\star(\chi)(v)-\frac{v(f)}{\bar\chi(f)}$ is usc and $\Xdiv_\infty\subset X^\na_\infty$ is dense. Rearranging, we have $\bar\chi(f)\le \frac{v(f)}{\FS^\star(\chi)(v)}$ for all $v\in X^\na_\infty$, and~(ii) follows. 
    If $\bar\chi=\bar\chi'$, then $\FS(\chi)=\FS(\chi')$ on $\Xdiv$ by Lemma~\ref{lem:satprop}. Since these functions are lsc and $X^{\mathrm{div}}_\infty$ is dense in $X^\na_\infty$, their usc regularizations agree on $X^\na_\infty$, and hence on $X^\na$, since both vanish on $X^\beth$. Conversely, $\FS^\star(\chi)=\FS^\star(\chi')$ implies $\bar\chi=\bar\chi'$ by~(ii). This proves~(iii).
    Finally, assume the envelope property and we prove~(iv). Note that in this case $\FS^\star(\chi)\in  \cE(X)\cap \cL^\infty$. Further, since $\FS^\star(\chi)=\FS(\chi)=\sup_\la\FS(\chi^{(\la)})$ on $\Xdiv$, the increasing net $\FS(\chi^{(\la)})$ converges weakly to $\FS^\star(\chi)$, by Theorem~\ref{thm:weaktop}. Thus Corollary~\ref{cor:volincreas} gives $\vol(\FS^\star(\chi))=\sup_\la\vol(\FS(\chi^{(\la)}))=\vol(\FS(\chi))$, where the second equality follows from Proposition~\ref{prop:volcts}. Pick a decreasing net $(\varphi_i)$ in $\cH(X)$ converging to $\FS^\star(\chi)$, we get
    \begin{align}\label{equ:d1approx}
    \de_1(\FS(\chi^{(\la)}), \FS^\star(\chi))
    &=\lim_i(\vol(\varphi_i)-\vol(\FS(\chi^{(\la)})))\\
    &=\vol(\FS(\chi))-\vol(\FS(\chi^{(\la)}))\to 0,
    \end{align}
    again by Proposition~\ref{prop:volcts}. 
\end{proof}

We now state and prove the isometry statement discussed in the introduction. For the statement below, we set
\[
\cE_\uparrow(X)\=\{\varphi\in\cE(X)\mid \varphi=Q(\varphi)^\star\}.
\]
Note that $\cE_\uparrow(X)\subset \cL^\infty$.
\begin{thm}
    Assume $X$ has the envelope property. Then the regularized Fubini--Study operator defines a surjective isometry $\FS^\star\colon (\Nbdd, \de_1)\twoheadrightarrow (\cE_\uparrow(X), \de_1)$, which restricts to an isometric isomorphism $\FS^\star\colon (\cN^s, \de_1)\to (\cE_\uparrow(X), \de_1)$, with inverse $\IN\colon (\cE_\uparrow(X), \de_1)\to (\cN^s, \de_1)$.
\end{thm}
\begin{proof}
    We first show surjectivity. Let $\varphi\in \cE_\uparrow(X)$ and set $\chi:=\IN(\varphi)$. Since $\varphi$ is usc, we have $\chi=\IN_{\Xdiv_\infty}(\varphi)$, as in the proof of Proposition~\ref{prop:regFS}(ii). This shows $\chi\in \cN^s$, and $\FS^\star(\chi)=Q(\varphi)^\star=\varphi$, thanks to Proposition~\ref{prop:FSINcomp}. 
    
    Now pick $\chi_1, \chi_2\in \Nbdd$. Then $\chi_i^{(\la)}\in \cT_\R$, and $\f_i^{(\la)}:=\FS(\chi_i^{(\la)})\in\cH_\R(X)$ for $\la\gg0$. Also set $\f_i:=\FS^\star(\chi_i)$. By~\eqref{equ:d1approx} and Proposition~\ref{prop:volcts} we have 
    \[
    \de_1(\f_i^{(\la)}, \f_i)=\de_1(\chi_i^{(\la)}, \chi_i)\to 0
    \]
    as  $\la\to \infty$, for $i=1,2$.
    Moreover, Corollary~\ref{cor:FSisometry} yields
    \[
    \de_1(\chi_1^{(\la)}, \chi_2^{(\la)})=\de_1(\f_1^{(\la)}, \f_2^{(\la)})
    \]
    for all $\la\gg0$. 
    The triangle inequality and the above convergences give
    \[
    \de_1(\f_1,\f_2)
    =\lim_\la\de_1(\f_1^{(\la)},\f_2^{(\la)})
    =\lim_\la\de_1(\chi_1^{(\la)},\chi_2^{(\la)})
    =\de_1(\chi_1,\chi_2).
    \]
    It remains to show that $\FS^\star$ restricts to an injection on $\cN^s$, and this is Proposition~\ref{prop:regFS}.
\end{proof}
\subsection{Darvas metrics on numerical classes}
We end with studying a Darvas metric on $\BBPF_{n-1}(X)$, using the rooftop formula~\eqref{equ:rooftop}.
For any $\alpha, \beta\in\BBPF_{n-1}(X)$, we define
\[
\vol(\a\wedge\b)\=\sup\{\vol(\gamma)\mid \gamma\le\a,\b,\ \gamma\in\BPF_{n-1}(X)\}.
\]
\begin{lem}
    The extended volume is continuous along decreasing nets in $\BBPF_{n-1}(X)$.
\end{lem}
\begin{proof}
    Suppose $(\alpha_i), (\beta_i)$ are decreasing nets in $\BBPF_{n-1}(X)$ converging to $\alpha, \beta\in \BBPF_{n-1}(X)$ respectively. By monotonicity, we trivially have $\vol(\alpha\wedge\beta)\le\liminf_i \vol(\alpha_i\wedge\beta_i)$. Note also that since $\alpha_i\to \alpha$ and $\beta_i\to \beta$ strongly, for any $0<t<1$, we have
    \[
    t\alpha_i\le \alpha\le \alpha_i, \quad t\beta_i\le \beta\le \beta_i
    \]
    for all $i\gg 0$.
    If $\gamma\in \CNum(X)$ is big and nef so that $\gamma\le \alpha_i\wedge\beta_i$, then $t\gamma\le t\alpha_i\le \alpha$ and the same applies to $\beta$. Thus 
    \[
    \vol(\alpha\wedge\beta)\ge t^n\limsup_i\vol(\alpha_i\wedge\beta_i),
    \]
    and we are done by letting $t\to 1$.
\end{proof}
For $\alpha, \beta\in \BBPF_{n-1}(X)$, we define
\[
\de_1(\a,\b)\=\vol(\a)+\vol(\b)-2\vol(\a\wedge\b).
\]
\begin{prop}
    The assignment defines a metric on $\BBPF_{n-1}(X)$.
\end{prop}
\begin{proof}
    It is clear that $\de_1(\alpha, \alpha)=0$ and it is symmetric. By the previous lemma, if $(\alpha_i), (\beta_i)$ are decreasing nets in $\BBPF_{n-1}(X)\cap \CNum(X)$ converging to $\alpha, \beta\in \BBPF_{n-1}(X)$, we also have 
    $\de_1(\alpha, \beta)=\lim_i\de_1(\alpha_i, \beta_i)$. Thus it is enough to prove the triangle inequality for $\alpha, \beta, \gamma\in \BBPF_{n-1}(X)\cap\CNum(X)$, which is equivalent to 
    \begin{equation}\label{equ:trangleineq}
        \vol(\alpha)+\vol(\beta\wedge\gamma)\ge \vol(\alpha\wedge\beta)+\vol(\alpha\wedge \gamma).
    \end{equation}
    If one of the volumes on the right hand side is zero, then the inequality trivially holds. So suppose there exist $\theta\le \alpha\wedge\beta$ and $\chi\le \alpha\wedge \gamma$, where $\theta, \chi\in \CNum(X)$ are big and nef. Pick a large enough smooth model $Y$ so that all of $\theta, \chi, \alpha, \beta, \gamma\in \Num(Y)$, and for any $0<\varepsilon\ll1$ pick $\R$-divisors $D_\theta\in \theta, D_\chi\in \chi, D_{\alpha, \varepsilon}\in (1+\varepsilon)\alpha$ on $Y$ so that $D_\theta, D_\chi\le D_{\alpha, \varepsilon}$. Then Lemma~\ref{lem:linalg} applied to 
    \[
    U=\Hnot(Y, \lfloor mD_\theta\rfloor), V=\Hnot(Y, \lfloor mD_{\alpha, \varepsilon}\rfloor), W=\Hnot(Y, \lfloor mD_\chi\rfloor)
    \]
    for $m\gg0$ gives
    \[
    (1+\varepsilon)^n\vol(\alpha)+\vol(D_\theta\wedge D_\chi)\ge \vol(\theta)+\vol(\chi).
    \]
    We claim that $\vol(D_\theta\wedge D_\chi)\le \vol(\beta\wedge\gamma)$. Indeed, if $\vol(D_\theta\wedge D_\chi)=0$, this trivially holds; if $D_\theta\wedge D_\chi$ is big, the claim follows from Fujita's approximation. Taking $\varepsilon\to 0$ and supremum over all such $\theta, \chi$ yields~\eqref{equ:trangleineq}.

    Finally, suppose $\de_1(\alpha, \beta)=0$. Then $\vol(\alpha)=\vol(\alpha\wedge\beta)=\vol(\beta)$. Pick $\theta_i\in \CNum(X)$ big and nef so that $\theta_i\le \alpha\wedge\beta$ and $\vol(\theta_i)\to \vol(\alpha\wedge\beta)$. Then 
    \[
    1\le \frac{(\theta_i, \alpha^{n-1})}{(\theta_i^n)^{1/n}(\alpha^n)^{1-1/n}}\le \frac{(\alpha^{n})^{1/n}}{(\theta_i^n)^{1/n}}\to 1,
    \]
    and thus $\bar\ii(\theta_i, \alpha)\to 0$ and symmetrically $\bar\ii(\theta_i, \beta)\to 0$. Proposition~\ref{prop:metrictop} gives $\theta_i\to \alpha$ and $\theta_i\to \beta$ strongly, which yields $\alpha=\beta$, and $\de_1$ is a metric.
\end{proof}
We define a modified metric on $\BBPF_{n-1}(X)$, by setting for $\a, \b\in \BBPF_{n-1}(X)$,
\[
\du(\a, \b)\=\de_1(\a, \b)+|\log\vol(\a)-\log\vol(\b)|.
\]
\begin{prop}
    A net $([\a_i])_i$ in $\BBPF_{n-1}(X)$ is a Cauchy net for $\du$ iff it is a Cauchy net for $\de_\infty$. In particular, $(\BBPF_{n-1}(X), \du)$ is complete. 
\end{prop}
\begin{proof}
    If $\delta_{ij}\=\de_\infty(\a_i, \a_j)\to 0$, then by Proposition~\ref{prop:metrictop} we have $(\log\vol(\a_i))$ is a Cauchy net, and so $\vol(\a_i)\to C$ for some $C>0$. Further,
    \[
    \de_1(\a_i, \a_j)\le \vol(\a_i)+\vol(\a_j)-\vol(e^{-\delta_{ij}}\a_i)-\vol(e^{-\delta_{ij}}\a_j)=(1-e^{-n\delta_{ij}})(\vol(\a_i)+\vol(\a_j))\to 0.
    \] 
    This shows $(\a_i)$ is a Cauchy net for $\du$. 

    Conversely, if $\du(\a_i, \a_j)\to 0$, then again $\vol(\a_i)\to C>0$, and we may pick $\alpha_{ij}\in \CNum(X)$ big and nef with $\alpha_{ij}\le \a_i\wedge\a_j$ so that 
    \[
    \vol(\a_i)+\vol(\a_j)-2\vol(\alpha_{ij})\to 0.
    \]
    This implies $\bar\ii(\a_i, \a_{ij})\to 0$ and $\bar\ii(\a_j, \a_{ij})\to 0$ and by Proposition~\ref{prop:metrictop} we have $\de_\infty(\a_i, \a_j)\to 0$.
\end{proof}
Restricting to $\bar\cE(X)$ and applying Theorem~A, we obtain
\begin{cor}
    The metric completion of $(\bar\cE(X), \du)$ can be identified with $\cC_\infty(X)$. If $X$ is smooth, then $(\bar \cE(X), \du)$ is complete.
\end{cor}

\printbibliography

\end{document}